\documentclass[a4paper]{article}
\pdfoutput=1
\usepackage[utf8]{inputenc}
\usepackage[T1]{fontenc}
\usepackage{textcomp}

\usepackage[english]{babel}
\usepackage{csquotes}
\usepackage{mathtools,amssymb,amsthm}
\usepackage{enumitem}

\usepackage{array}

\usepackage{graphicx}
\usepackage{xcolor}
\usepackage{float}
\usepackage{extarrows}
\usepackage{tikz}
\usetikzlibrary{arrows.meta}
\usepackage[framemethod=tikz]{mdframed}

\usepackage[scaled]{helvet}
\renewcommand{\emph}[1]{\textbf{#1}}

\usepackage{thmtools}
\declaretheoremstyle[
spaceabove=\topsep, spacebelow=\topsep,
headfont=\normalfont\bfseries,
postheadspace=.5em,
qed=$\square$
]{whiteqed}

\declaretheorem[style=definition,parent=section,title=Theorem]{Thm}
\declaretheorem[style=definition,sibling=Thm,title=Proposition]{Prop}
\declaretheorem[style=definition,sibling=Thm,title=Lemma]{Lem}
\declaretheorem[style=definition,sibling=Thm,title=Corollary]{Cor}
\declaretheorem[style=definition,numbered=no]{Claim}

\declaretheorem[style=whiteqed,sibling=Thm,title=Definition]{Def}

\declaretheorem[style=whiteqed,sibling=Thm,title=Remark]{Rmk}
\declaretheorem[style=whiteqed,sibling=Thm,title=Example]{Expl}

\declaretheorem[style=whiteqed,sibling=Thm,title=Theorem]{Thm*}
\declaretheorem[style=whiteqed,sibling=Thm,title=Proposition]{Prop*}
\declaretheorem[style=whiteqed,sibling=Thm,title=Corollary]{Cor*}
\declaretheorem[style=whiteqed,sibling=Thm,title=Lemma]{Lem*}

\usepackage{bm}
\newcommand{\bvec}[1]{\bm{#1}}

\newcommand{\smid}{\mathrel{;}}
\mathchardef\mhyphen="2D
\usepackage{stmaryrd}
\newcommand{\dbracket}[1]{\left \llbracket #1 \right \rrbracket}

\newcommand{\Set}[1]{\left \{\, #1 \, \right \}}

\newcommand{\id}{\mathrm{id}}
\newcommand{\op}{\mathrm{op}}
\newcommand{\isoarrow}{\stackrel{\sim}{\to}}
\newcommand{\llim}{\varprojlim}
\newcommand{\rlim}{\varinjlim}

\DeclareMathOperator{\Hom}{Hom}

\DeclareMathOperator{\Spec}{Spec}
\DeclareMathOperator{\Idl}{Idl}

\newcommand{\modsp}[2]{#1\mhyphen\mathbf{Mod}#2}
\newcommand{\fp}{\omega}
\DeclareMathOperator{\Mor}{Mor}
\DeclareMathOperator{\RanExt}{Ran}
\DeclareMathOperator{\LanExt}{Lan}
\DeclareMathOperator{\Diag}{Diag}

\makeatletter

\@addtoreset{equation}{section}
\makeatother

\usepackage[style=alphabetic,giveninits,sorting=nyt,useprefix=true,related=false]{biblatex}
\DeclareSortingNamekeyTemplate{
	\keypart{\namepart{family}}
	\keypart{\namepart{prefix}}
	\keypart{\namepart{given}}
	\keypart{\namepart{suffix}}
}
\renewbibmacro{begentry}{\midsentence} 
\DeclareFieldFormat[misc]{title}{``#1''}
\DeclareFieldFormat[online]{title}{#1}
\DeclareFieldFormat[article]{volume}{\mkbibbold{#1}}
\DeclareFieldFormat{postnote}{#1}
\DeclareListFormat*{location}{}

\DeclareFieldFormat*{eventtitle}{}
\DeclareFieldFormat*{eventdate}{}
\DeclareFieldFormat*{venue}{}
\DeclareListFormat*{organization}{}
\renewbibmacro*{event+venue+date}{
	\iffieldundef{eventtitle}
	{}
	{%
		\printfield{eventtitle}%
	}%
	\ifboolexpr{
		test {\iffieldundef{venue}}
		and
		test {\iffieldundef{eventyear}}
	}
	{}
	{\setunit*{\addspace}%
		\printtext{%
			\printfield{venue}%
			\setunit*{\addcomma\space}%
			\printeventdate}}%
	\newunit
}

\makeatletter
\newcommand*{\importbibfrom}[1]{%
	\def\blx@bblfile{%
		\blx@secinit
		\begingroup
		\blx@bblstart
		\InputIfFileExists{#1.bbl}
		{\blx@info@noline{... file '#1.bbl' found}%
			\global\toggletrue{blx@bbldone}}
		{\blx@info@noline{... file '#1.bbl' not found}%
			\typeout{No file #1.bbl.}}%
		\blx@bblend
		\endgroup
		\csnumgdef{blx@labelnumber@\the\c@refsection}{0}}}

\global\let\blx@rerun@biber\relax
\makeatother

\importbibfrom{bib_for_arXiv}

\allowdisplaybreaks[4]

\usepackage[bookmarks=true,bookmarksnumbered=true]{hyperref}

\usepackage[affil-it]{authblk}

\title{\Large Limits, Colimits, and Spectra of Modelled Spaces}
\author{Hisashi Aratake%
	\thanks{\textit{E-mail address}: \texttt{aratake@kurims.kyoto-u.ac.jp}
		}
	}
\affil{Research Institute for Mathematical Sciences,\\
	Kyoto University, Kyoto, Japan}
\date{}

\begin{document}

\maketitle

\begin{abstract}
	It is well-known that the construction of Zariski spectra of (commutative) rings yields a dual adjunction
	between the category of rings and the category of locally ringed spaces.
	There are many constructions of spectra of algebras in various contexts giving such adjunctions.
	Michel Coste unified them in the language of categorical logic by showing that, for an appropriate triple $ (T_0,T,\Lambda) $ (which we call a spatial Coste context),
	each $ T_0 $-model can be associated with a $ T $-modelled space and 
	that this yields a dual adjunction between the category of $ T_0 $-models and the category of $ T $-modelled spaces and ``admissible'' morphisms.
	However, most of his proofs remain unpublished.
	In this paper, we introduce an alternative construction of spectra of $ T_0 $-models and give another proof of Coste adjunction.
	Moreover, we also extend spectra of $ T_0 $-models to relative spectra of $ T_0 $-modelled spaces
	and prove the existence of limits and colimits in the involved categories of modelled spaces.
	We can deduce, for instance, that the category of ringed spaces whose stalks are fields is complete and cocomplete.
\end{abstract}

\setcounter{section}{-1}

\section{Introduction}
Spectra of algebras are ubiquitous in mathematics. One of the most famous spectra is the Zariski spectrum.
If we write $ \Spec_Z(A) $ for the Zariski spectrum of a (commutative) ring $ A $,
we have a functor $ \Spec_Z \colon \mathbf{Ring}^{\op} \to \mathbf{LRS} $ from the category of rings to that of locally ringed spaces, 
and it is a right adjoint of the global section functor $ \Gamma \colon \mathbf{LRS} \to \mathbf{Ring}^{\op} $.
\[ \tikz{
	\node (L) at (0,0) {$ \mathbf{LRS} $};
	\node (R) at (4,0) {$ \mathbf{Ring}^{\op} $};
	\draw[bend left=15pt,->] (L) to node[above] {$ \Gamma $} node[midway,name=U] {} (R);
	\draw[bend left=15pt,->] (R) to node[below] {$ \Spec_Z $} node[midway,name=D] {} (L);
	\path (U) to node[sloped] {$ \dashv $} (D);
} \]
Many other spectra constructions yield adjunctions between categories of algebras 
and of ``locally algebra-ed spaces'' (see the examples in \S\ref{subsec:Coste-context}).
In some sense, a spectrum of an algebra is the best possible approximation to a free ``local'' algebra 
when the forgetful functor from the category of ``local'' algebras does not have a left adjoint.

Motivated by the \'{e}tale topos of a locally ringed topos \cite{HakimThesis} (cf.\ \cite{Wra1976}),
topos-theorists developed the theory of \textit{topos-theoretic spectra}.
Here, all algebras and spaces are replaced by algebra-ed toposes, 
and the problem turns into obtaining a right biadjoint of the forgetful functor from the bicategory of locally algebra-ed toposes.
Tierney \cite{Tier1976a} constructed Zariski spectra of ringed elementary toposes.
\[ \tikz{
	\node (L) at (0,0) {$ \mathbf{LocRingedToposes} $};
	\node (R) at (6,0) {$ \mathbf{RingedToposes} $};
	\draw[bend left=10pt,->] (L) to node[above] {forgetful} node[midway,name=U] {} (R);
	\draw[bend left=10pt,->] (R) to node[below] {$ \Spec $} node[midway,name=D] {} (L);
	\path (U) to node[sloped] {$ \dashv $} (D);
}  \]
Cole established a general theory\footnote{
	His results had long been unpublished, while his manuscript was circulated among specialists.
	Johnstone's book \cite[pp.~205--209]{JohTT} had been the only published source for Cole's theory.
	Cole's paper was finally published as \cite{Cole2016}.}.
He considered the situation where a class of homomorphisms between algebras in toposes (called \textit{admissible class} of morphisms by Johnstone) is given
so that every homomorphism admits an extremal-admissible factorization.
Then, via the notion of classifying topos, the construction of spectra can be seen as a $ 2 $-categorical construction in the bicategory of toposes and geometric morphisms.
Employing categorical logic, Coste \cite{Coste1979} introduced a framework,
which we will call a \textit{Coste context} (\emph{Definition \ref{def:Coste-context}}),
including all the known examples and whose conditions are much easier to verify.
A Coste context is an appropriate triple $ (T_0,T,\Lambda) $ of two theories $ T_0 \subseteq T $ and a set of ``etale maps'' between $ T_0 $-models.
Besides exploiting categorical logic, another essential difference of Coste's theory from Cole's is using etale maps instead of admissible maps.
Etale maps determine the class of admissible maps, and Coste proved a factorization theorem to apply Cole's theory.
He then obtained a biadjunction between the bicategory of $ T_0 $-modelled toposes and that of $ T $-modelled toposes and admissible morphisms.
He also outlined concrete representations of spectra of set-based models and models in Grothendieck toposes given by explicit sites.
These works led to the discovery of real spectra of rings (\cite{CosteRoy1979b}, \cite{CosteRoy1980a}), which now play an important role in real algebraic geometry.
Dubuc \cite{Dub2000} generalized the Cole--Coste theory of spectra in the $ 2 $-categorical language similarly to Cole's theory,
but he took as the primitive notion not admissible maps but etale maps in the same spirit with Coste.
Recently, Osmond (\cite{Osm2021a}) revisited the Cole--Coste theory with great details and compared it with Diers spectra (\cite{Osm2020a}, \cite{Osm2020b}).
His thesis \cite{OsmondThesis} contains some additional results on Coste--Diers adjunction (Theorem 4.2.5.1), Dubuc spectra (\S5.3), 
and bilimits/bicolimits of locally modelled toposes (Chapter 6).
Another recent paper \cite{JPV2022} seems to be closely related to Coste's theory
and points out that $ \mathbb{F}_1 $-schemes and $ C^{\infty} $-schemes could be obtained from their theory of spectra.

In this paper, rather than working in the topos-theoretic context, we will restrict ourselves to \textit{modelled spaces}, 
i.e.\ topological spaces equipped with sheaves of models.
At the cost of generality, we hope our presentation will be accessible for non-specialists of topos theory.
We will assume the reader is familiar with first-order logic, elementary category theory, and sheaves on topological spaces.
We will occasionally refer to topos theory but will not make essential use of it except in \S\S\ref{subsec:comparison-with-Coste-spectra}, \ref{subsec:standard-context}.

We demonstrate that a Coste context (satisfying an additional ``spatiality'' assumption)
is an appropriate setup to deal with limits and colimits in categories of modelled spaces.
Let $ (T_0,T,\Lambda) $ be a spatial Coste context.
We write $ \modsp{T_0}{\mathbf{Sp}} $ for the category of $ T_0 $-modelled spaces
and $ \modsp{\mathbb{A}}{\mathbf{Sp}} $ for the category of $ T $-modelled spaces and admissible morphisms (\S\ref{subsec:introducing-modelled-spaces}).
Our main contributions are as follows:
\begin{itemize}
	\item We provide an alternative construction of spectra of $ T_0 $-models and show that it is equivalent to Coste's original construction.
	\item We extend the construction to ``relative spectra'' of $ T_0 $-modelled spaces,
	and show that, for any $ T $-modelled space $ (X,P) $, there exists an adjunction
	\[ \tikz{
		\node (L) at (0,0) {$ \modsp{\mathbb{A}}{\mathbf{Sp}}/(X,P) $};
		\node (R) at (6,0) {$ \modsp{T_0}{\mathbf{Sp}}/(X,P) $};
		\draw[bend left=10pt,->] (L) to node[above] {forgetful} node[midway,name=U] {} (R);
		\draw[bend left=10pt,->] (R) to node[below] {$ \Spec $} node[midway,name=D] {} (L);
		\path (D) to node[sloped] {$ \vdash $} (U);
	}.  \]
	This yields another proof of Coste adjunction.
	\item We show completeness and cocompleteness of $ \modsp{\mathbb{A}}{\mathbf{Sp}} $.
\end{itemize}

Coste's paper \cite{Coste1979} contains only a brief outline of his dissertation, and most proofs are omitted.
Unfortunately, the details remain unpublished.
The first two sections of this paper contain detailed proofs of some results in \cite{Coste1979}.
Our exposition could be regarded as complementary to \cite{Osm2021a}, and these papers cover most of the results in \cite{Coste1979}.
In \S\ref{subsec:preliminaries}, we will review first-order categorical logic and locally finitely presentable categories.
The rest of \S\ref{sec:Coste-context} is devoted to introducing Coste's framework and filling in the details in his paper.
After recalling generalities on modelled spaces in \S\ref{subsec:introducing-modelled-spaces}, we construct spectra of $ T_0 $-models in \S\ref{subsec:spectra-of-models}.
Our construction of the underlying space of a spectrum is essentially the same as Coste's,
but we define its structure sheaf similarly to the usual construction of Zariski spectra.
The proof of the dual adjunction $ \Gamma \dashv \Spec $ will be postponed to \S\ref{sec:spectra}.
In \S\ref{subsec:comparison-with-Coste-spectra}, we show that our and Coste's constructions yield an equivalent modelled toposes.
In \S\ref{subsec:standard-context}, we also give criteria for a $ T_0 $-model to be represented as the global section model of the structure sheaf of its spectrum.
We will not use the results in \S\S\ref{subsec:comparison-with-Coste-spectra}, \ref{subsec:standard-context} in the remaining part of this paper,
and the reader unfamiliar with topos theory and categorical logic can safely skip them.

In \S\ref{sec:colimits}, we then prove the existence of small colimits in $ \modsp{\mathbb{A}}{\mathbf{Sp}} $.
This is well-known for $ \mathbf{LRS} $ \cite[Chap.\ I, Proposition 1.6]{DemaGab1980}.
The construction of colimits depends only on the notions introduced before \S\ref{subsec:introducing-modelled-spaces} and thus does not exploit spectra of models.
In \S\ref{sec:spectra}, we construct relative spectra of $ T_0 $-modelled spaces over an arbitrary $ T $-modelled space and prove the forgetful-$ \Spec $ adjunction.
Here, we have to use (the underlying spaces of) spectra of $ T_0 $-models, and this is why we discuss them in advance.
Finally, in \S\ref{sec:limits}, we consider limits of modelled spaces.
We first make some observations on limits of $ T_0 $-modelled spaces.
Then, combining the results on relative spectra and limits of $ T_0 $-modelled spaces, we obtain limits in $ \modsp{\mathbb{A}}{\mathbf{Sp}} $.
The methods in the last two sections are extensions of the arguments in \cite{Gillam2011} (cf.\ \cite{BrandLimits}, \cite[\S12.7]{BlechThesis})
used to show completeness of $ \mathbf{LRS} $.
Consequently, we can show, for instance, that the category of ringed spaces whose stalks are fields is complete and cocomplete.

\paragraph{Notations and conventions}
A tuple of elements of a set $ A $ will be denoted in boldface letters (say, $ \bvec{a} $). We will write $ \bvec{a} \in A $.
For a topological space $ X $, the poset of open subsets of $ X $ will be denoted by $ \mathcal{O}(X) $.
Let $ \mathbf{Sp} $ denote the category of topological spaces and continuous maps.

\section{Review of Coste's Framework} \label{sec:Coste-context}
We begin with introducing the general settings of Coste's theory of spectra \cite{Coste1979}.

\subsection{Cartesian Theories and Locally Finitely Presentable Categories} \label{subsec:preliminaries}
We first recall some elements of first-order categorical logic and the theory of locally finitely presentable categories.
We will exploit relationship between cartesian categories, cartesian theories, and locally finitely presentable categories.
All the prerequisites below (except \emph{Lemma \ref{lem:useful-lemma}}) are covered in \cite[Chapters D1 and D2]{Elephant},
and most of our conventions for notations and terminologies owe to it.

A \emph{cartesian category} is a category having all finite limits.
If $ \mathcal{C},\mathcal{D} $ are cartesian categories, a \emph{cartesian functor} $ F \colon \mathcal{C} \to \mathcal{D} $ is a functor preserving all finite limits.
Let $ \mathfrak{Cart}(\mathcal{C},\mathcal{D}) $ denote the category of cartesian functors from $ \mathcal{C} $ to $ \mathcal{D} $ and natural transformations.

For the logical counterpart of cartesian categories,  we consider cartesian theories.
For simplicity, we will treat single-sorted languages only.
Multi-sorted languages will be necessary to establish \emph{Proposition \ref{prop:cart-cat-vs-cart-th}}(2) below,
but we will mention that result only for the sake of completeness and will not use it.
A first-order language consists of relation symbols and function symbols associated with natural numbers as arities.
The notions of term and structure are defined as usual (cf.\ \cite[\S3.1]{Ara2021}).
Formulas will be defined below so that they are built up with $ \land, \exists $ only.
Every term and formula will always be equipped with a tuple, say, $ \bvec{u} $ of variables containing all the free variables in it.
Axioms of a theory will be sequents $ \varphi \vdash \psi $.

\begin{Def}
	Let $ \mathcal{L} $ be a first-order language.
	Atomic formulas are either $ \top $, $ \rho(\bvec{u}) $ (for a relation symbol $ \rho $ in $ \mathcal{L} $), 
	or $ s(\bvec{u}) = t(\bvec{u}) $ (for $ \mathcal{L} $-terms $ s,t $).
	A finite conjunction of atomic formulas is called a \emph{Horn formula}.
	
	Let $ T $ be an arbitrary first-order $ \mathcal{L} $-theory.
	The class of \emph{$ T $-cartesian formulas} is defined as follows:
	atomic formulas are $ T $-cartesian, and finite conjunctions of $ T $-cartesian formulas are $ T $-cartesian.
	The formula $ \exists v \varphi(\bvec{u},v) $ is $ T $-cartesian provided $ v $ is a variable not contained in $ \bvec{u} $,
	$ \varphi(\bvec{u},v) $ is $ T $-cartesian, and the sequent $ \varphi \land \varphi[v'/v] \vdash v=v' $ is provable in $ T $.
	A sequent $ \varphi \vdash \psi $ is \emph{$ T $-cartesian} if both $ \varphi $ and $ \psi $ are $ T $-cartesian.
	A \emph{cartesian theory} is a set $ T $ of sequents which admits a well-founded partial ordering such that
	each axiom is cartesian relative to the subtheory consisting of the axioms preceding it.
	When we fix a cartesian theory $ T $, we refer to $ T $-cartesian formulas as cartesian formulas.
\end{Def}

For a cartesian theory $ T $, the notions of $ T $-model and homomorphism between them are defined as usual.
Let $ T\mhyphen\mathbf{Mod} $ denote the category of (set-based) $ T $-models and homomorphisms.
We will often say ``$ T $-models'' to mention those in $ \mathbf{Set} $,
and will make sure to say ``$ T $-models in $ \mathcal{C} $'' for a cartesian category $ \mathcal{C} $ apart from $ \mathbf{Set} $.
For an $ \mathcal{L} $-structure $ A $ in $ \mathcal{C} $ and an $ \mathcal{L} $-formula $ \varphi(\bvec{u}) $ interpretable in $ A $,
we will write $ \dbracket{\varphi}_A $ for the interpretation of $ \varphi $ as a subobject of $ A^{\bvec{u}} $.

Let us proceed to the third concept.
\begin{Def}
	Let $ \mathcal{K} $ be a locally small category with filtered colimits.
	We say an object $ M \in \mathcal{K} $ is \emph{finitely presentable} (a.k.a.\ $ \omega $-presentable) if it has the following property:
	for any filtered diagram $ D \colon \mathcal{J} \to \mathcal{K} $, the canonical map in $ \mathbf{Set} $
	\[ \rlim_{i \in \mathcal{J}} \mathcal{K}(M,D_i) \to \mathcal{K}(M,\rlim_{i \in \mathcal{J}} D_i) \]
	is an isomorphism.
\end{Def}
Finite presentability can be expressed as a combination of the following two conditions:
\begin{enumerate}[label=(\alph*)]
	\item Any morphism $ M \to \rlim_i D_i $ factors through some coprojection $ D_i \to \rlim_i D_i $.
	\item Such a factorization of a morphism $ M \to \rlim_i D_i $ is essentially unique:
	if a morphism $ M \to \rlim_i D_i $ factors through both $ D_i $ and $ D_j $,
	then there exist morphisms $ i \xrightarrow{u} k \xleftarrow{v} j $ in $ \mathcal{J} $ such that the composites of the two paths from $ M $ to $ D_k $ coincide.
	\[ \tikzset{cross/.style={preaction={-,draw=white,line width=6pt}}}
	
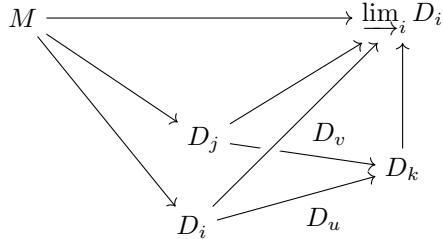
\tikz[auto]{
		\node (UL) at (2,0,-1) {$ D_j $};
		\node (DL) at (3,0,2) {$ D_i $};
		\node (UR) at (5,0,0) {$ D_k $};
		\node (T) at (0,2,0) {$ M $};
		\node (T') at (5,2,0) {$ \rlim_i D_i $};
		
		\draw[->] (UL) to node {$ D_v $} (UR);
		\draw[->] (DL) to node[swap] {$ D_u $} (UR);
		\draw[->,cross] (DL) to (T');
		\draw[->] (UL) to (T');
		\draw[->] (UR) to (T');
		
		\draw[->] (T) to (DL);
		\draw[->] (T) to (T');
		\draw[->] (T) to (UL);
	} \]
\end{enumerate}

\begin{Def}
	We say a locally small category $ \mathcal{K} $ is \emph{locally finitely presentable} (a.k.a.\ locally $ \omega $-presentable) if it satisfies the following conditions:
	\begin{itemize}
		\item It has all small colimits.
		\item There exists a small set $ \mathcal{A} $ of finitely presentable objects in $ \mathcal{K} $ such that
		any object in $ \mathcal{K} $ is a filtered colimit of objects in $ \mathcal{A} $.
	\end{itemize}
	Let $ \mathcal{K}_{\fp} $ denote the full subcategory of finitely presentable objects of $ \mathcal{K} $.
\end{Def}
If $ \mathcal{K} $ is locally finitely presentable, then it also has all small limits, and $ \mathcal{K}_{\fp} $ is essentially small.

These three concepts (cartesian categories, cartesian theories, and locally finitely presentable categories) mutually correspond as follows:

\begin{Prop*}[{\cite[Theorem D1.4.7, Example D1.4.8]{Elephant}}] \label{prop:cart-cat-vs-cart-th}
	\ 
	
	\begin{enumerate}
		\item For any cartesian theory $ T $, we can construct a cartesian category $ \mathcal{C}_T $
		(called the \emph{syntactic category} of $ T $) as follows: objects of $ \mathcal{C}_T $ are $ T $-cartesian formulas. 
		For a formula $ \varphi(\bvec{u}) $, the corresponding object is denoted by $ \{\bvec{u}.\,\varphi\} $.
		A morphism $ \{\bvec{u}.\,\varphi\} \to \{\bvec{v}.\,\psi\} $ is an equivalence class of formulas $ \theta(\bvec{u},\bvec{v}) $ for which the sequents
		\begin{gather*}
			\theta \vdash \varphi \land \psi \\
			\theta \land \theta[\bvec{v'}/\bvec{v}] \vdash \bvec{v}=\bvec{v'} \\
			\varphi \vdash \exists \bvec{v} \theta
		\end{gather*}
		are provable in $ T $.
		Here, two formulas $ \theta(\bvec{u},\bvec{v}), \theta'(\bvec{u},\bvec{v}) $ are defined to be equivalent 
		if the sequents $ \theta \vdash \theta' $ and $ \theta' \vdash \theta $ are provable in $ T $.
		The tuples $ \bvec{u},\bvec{v} $ are usually assumed to be disjoint, but we will occasionally write, e.g.,
		$ \{\bvec{uv}.\,\psi\} \to \{\bvec{u}.\,\varphi\} $ if no confusion arises.
		
		Moreover, $ \mathcal{C}_T $ ``classifies $ T $-models,'' i.e.\ there is an equivalence of categories
		$ T\mhyphen\mathbf{Mod} \simeq \mathfrak{Cart}(\mathcal{C}_{T},\mathbf{Set}) $.
		
		\item For any small cartesian category $ \mathcal{C} $, there exists a (multi-sorted) cartesian theory $ T_{\mathcal{C}} $ such that
		the categories $ T_{\mathcal{C}}\mhyphen\mathbf{Mod} $ and $ \mathfrak{Cart}(\mathcal{C},\mathbf{Set}) $ are equivalent.
		\qedhere
	\end{enumerate}
\end{Prop*}

\begin{Prop*}[{\cite[Proposition D2.3.4]{Elephant}}]
	\ 
	
	\begin{enumerate}
		\item For a small cartesian category $ \mathcal{C} $, the category $ \mathfrak{Cart}(\mathcal{C},\mathbf{Set}) $ is locally finitely presentable.
		Limits and filtered colimits in $ \mathfrak{Cart}(\mathcal{C},\mathbf{Set}) $ can be computed as in $ \mathbf{Set}^{\mathcal{C}} $.
		
		\item For a locally finitely presentable category $ \mathcal{K} $, 
		the category $ \mathcal{K}_{\fp} $ is closed under finite colimits in $ \mathcal{K} $, and hence $ (\mathcal{K}_{\fp})^{\op} $ is cartesian.
		
		\item In the above notations, $ \mathfrak{Cart}(\mathcal{C},\mathbf{Set}) $ (resp.\ $ \mathcal{K} $)
		is an ind-completion of $ \mathcal{C}^{\op} $ (resp.\ $ \mathcal{K}_{\fp} $)
		via the Yoneda embedding $ \mathcal{C}^{\op} \hookrightarrow \mathfrak{Cart}(\mathcal{C},\mathbf{Set}) $ 
		(resp.\ via the embedding $ \mathcal{K}_{\fp} \hookrightarrow \mathcal{K}$). 
		Therefore, $ \mathcal{C} $ can be recovered as $ (\mathfrak{Cart}(\mathcal{C},\mathbf{Set})_{\fp})^{\op} $,
		and $ \mathcal{K} $ can be recovered as $ \mathfrak{Cart}((\mathcal{K}_{\fp})^{\op},\mathbf{Set}) $.
		\qedhere
	\end{enumerate}
\end{Prop*}

This proposition is part of the \textit{Gabriel-Ulmer duality},
which asserts that the following two $ 2 $-categories are dually biequivalent:
\begin{itemize}
	\item the $ 2 $-category of small cartesian categories, cartesian functors, and natural transformations,
	\item the $ 2 $-category of locally finitely presentable categories, functors preserving limits and filtered colimits, and natural transformations.
\end{itemize}

For a cartesian theory $ T $, by putting $ \mathcal{C}=\mathcal{C}_T $ and $ \mathcal{K} = T\mhyphen\mathbf{Mod} $ above,
we have a canonical categorical equivalence $ \mathcal{C}_{T} \simeq (T\mhyphen\mathbf{Mod}_{\fp})^{\op} $.

\begin{Prop*}[{\cite[Lemma D2.4.1]{Elephant}}] \label{prop:free-model}
	For a $ T $-cartesian formula $ \varphi(\bvec{u}) $, we will write $ M_{\varphi} $ for the finitely presentable $ T $-model corresponding to it.
	Then, for any $ T $-model $ A $, solutions of $ \varphi(\bvec{u}) $ in $ A $ are in bijection with homomorphisms $ M_{\varphi} \to A $:
	\[ \dbracket{\varphi}_A := \Set{\bvec{a} \in A \smid A \models \varphi(\bvec{a})} \cong \Hom_{T\mhyphen\mathbf{Mod}}(M_{\varphi},A).  \]
	We will often identify a solution with the corresponding homomorphism.
\end{Prop*}

Here is a list of facts we will need about syntactic categories:
\begin{Prop*}[{\cite[Lemmas D1.3.8, D1.4.4]{Elephant}}] \label{prop:generalities-on-C_T}
	Let $ T $ be a cartesian theory.
	
	\begin{enumerate}
		\item Any object of $ \mathcal{C}_T $ is isomorphic to $ \{\bvec{u}.\, \varphi \} $ for some Horn formula $ \varphi(\bvec{u}) $.
		
		\item In the arrow category of $ \mathcal{C}_T $, any morphism $ [\theta] \colon \{\bvec{u}.\,\varphi\} \to \{\bvec{v}.\,\psi\} $ is isomorphic to one of the form
		\[ [\bvec{u'}=\bvec{u''}] \colon \{\bvec{u'v'}.\,\varphi' \} \to \{\bvec{u''}.\,\psi' \}, \] 
		where $ \varphi' $ and $ \psi' $ are Horn formulas, and $ T \models \varphi'(\bvec{u'},\bvec{v'}) \vdash \psi'(\bvec{u'}) $.
		
		\item A morphism $ [\theta] \colon \{\bvec{u}.\,\varphi\} \to \{\bvec{v}.\,\psi\} $ is monic exactly when
		\[ T \models \theta(\bvec{u},\bvec{v}) \land \theta(\bvec{u'},\bvec{v}) \vdash \bvec{u}=\bvec{u'}. \qedhere \]
	\end{enumerate}
\end{Prop*}

Closing this preliminary subsection, we record an easy but useful lemma:
\begin{Lem} \label{lem:useful-lemma}
	Let $ \mathcal{K} $ be a locally finitely presentable category, and $ A = \rlim_i A_i $ a filtered colimit in $ \mathcal{K} $ with coprojections $ \alpha_i \colon A_i \to A $.
	If we have a commutative square 
	\[ \tikz[auto]{
		\node (UL) at (0,2) {$ M $};
		\node (UR) at (2,2) {$ A_i $};
		\node (DL) at (0,0) {$ N $};
		\node (DR) at (2,0) {$ A $};
		\draw[->] (UL) to node {$ m $} (UR);
		\draw[->] (UL) to node {$ k $} (DL);
		\draw[->] (DL) to node {$ n $} (DR);
		\draw[->] (UR) to node {$ \alpha_i $} (DR);
	}  \]
	with $ M,N \in \mathcal{K}_{\fp} $, then there exist morphisms $ i \to j $ in the indexing category and $ N \to A_{j} $ in $ \mathcal{K} $ 
	such that the following diagram commutes:
	\[ \tikz[auto]{
		\node (UL) at (0,2) {$ M $};
		\node (UR) at (2,2) {$ A_i $};
		\node (DL) at (0,0) {$ N $};
		\node (DR) at (2,0) {$ A_{j} $};
		\draw[->] (UL) to node {$ m $} (UR);
		\draw[->] (UL) to node {$ k $} (DL);
		\draw[->] (DL) to (DR);
		\draw[->] (UR) to (DR);
	}.  \]
\end{Lem}
\begin{proof}
	Use finite presentability of $ N $ to obtain a factorization $ N \to A_{i'} \to A $ of $ n $,
	and use finite presentability of $ M $ to take $ i \to j \leftarrow i' $ such that the composites of the two paths from $ M $ to $ A_j $ coincide.
\end{proof}

\subsection{Admissible Morphisms and $ \mathcal{V}_A $}
Let $ \mathcal{K} $ be a locally finitely presentable category.
Coste \cite{Coste1979} developed a method to construct a factorization system in $ \mathcal{K} $ from a class of morphisms.

\begin{Def} \label{def:admissible-map}
	Let $ \mathcal{V} $ be a class of morphisms between finitely presentable objects.
	We say a morphism $ \alpha \colon A \to B $ in $ \mathcal{K} $ is \emph{$ \mathcal{V} $-admissible}
	when, for each commutative diagram
	\[ \tikz[auto]{
		\node (UL) at (0,2) {$ M $};
		\node (UR) at (2,2) {$ N $};
		\node (DL) at (0,0) {$ A $};
		\node (DR) at (2,0) {$ B $};
		\draw[->] (UL) to node {$ k $} (UR);
		\draw[->] (UL) to (DL);
		\draw[->] (DL) to node {$ \alpha $} (DR);
		\draw[->] (UR) to (DR);
	}  \]
	with $ k \in \mathcal{V} $, there exists a unique filler $ \beta \colon N \to A $ making the two triangles commutative:
	\[ \tikz[auto]{
		\node (UL) at (0,2) {$ M $};
		\node (UR) at (2,2) {$ N $};
		\node (DL) at (0,0) {$ A $};
		\node (DR) at (2,0) {$ B $};
		\draw[->] (UL) to node {$ k $} (UR);
		\draw[->] (UL) to (DL);
		\draw[->] (DL) to node {$ \alpha $} (DR);
		\draw[->] (UR) to (DR);
		\draw[dashed,->] (UR) to node {$ \exists ! \beta $} (DL);
	}.  \]
\end{Def}
The class of $ \mathcal{V} $-admissible morphisms is often denoted by $ \mathcal{V}^{\bot} $ in the literature.
Let $ \mathbf{2} $ be the category $ \{0 \to 1\} $, and $ \Mor(\mathcal{K}) $ be the class of morphisms in $ \mathcal{K} $.

\begin{Prop*}[{\cite[Proposition 3.14]{Osm2020a}}] \label{prop:admissibility-is-invariant}
	Let $ \tilde{\mathcal{V}} \subseteq \Mor(\mathcal{K}) $ be the smallest class of morphisms containing $ \mathcal{V} $ 
	and all isomorphisms between finitely presentable objects and which is
	\begin{itemize}
		\item closed under composition,
		\item right cancellative, i.e.\ $ k \in \tilde{\mathcal{V}} $ whenever $ l , kl \in \tilde{\mathcal{V}} $,
		\item closed under pushouts along morphisms in $ \mathcal{K}_{\fp} $, and 
		\item closed under finite colimits in $ \mathcal{K}^{\mathbf{2}} $.
	\end{itemize}
	Then $ \tilde{\mathcal{V}} \subseteq \Mor(\mathcal{K}_{\fp}) $ (recall that $ \mathcal{K}_{\fp} $ is closed under finite colimits in $ \mathcal{K} $),
	and $ \mathcal{V} $-admissible morphisms are $ \tilde{\mathcal{V}} $-admissible.
\end{Prop*}

We will call a class $ \mathcal{V} $ of morphisms \emph{saturated} if $ \tilde{\mathcal{V}} = \mathcal{V} $.
We fix a saturated class $ \mathcal{V} \subseteq \Mor(\mathcal{K}_{\fp}) $ for the moment.

Coste introduced another class of morphisms to obtain a factorization system whose right class is $ \mathcal{V}^{\bot} $.
\begin{Def}
	For $ A \in \mathcal{K}$, we define the class $ \mathcal{V}_A $ of morphisms with domain $ A $ as follows:
	a morphism $ \lambda \colon A \to B $ in $ \mathcal{K} $ belongs to $ \mathcal{V}_A $ if it can be represented as a pushout
	\[ \tikz[auto]{
		\node (UL) at (0,2) {$ M $};
		\node (UR) at (2,2) {$ N $};
		\node (DL) at (0,0) {$ A $};
		\node (DR) at (2,0) {$ B $};
		\draw[->] (UL) to node {$ k $} (UR);
		\draw[->] (UL) to (DL);
		\draw[->] (DL) to node {$ \lambda $} (DR);
		\draw[->] (UR) to (DR);
		\node (po) at (1.5,0.5) {$ \ulcorner $};
	}  \]
	of a morphism $ k \in \mathcal{V} $ along a morphism in $ \mathcal{K} $.
	If $ \lambda \in \mathcal{V}_A $, we will write $ A_{\lambda} $ for its codomain.
\end{Def}
We will often regard $ \mathcal{V}_A $ as a full subcategory of the comma category $ A \downarrow \mathcal{K} $.
Thus, a morphism $ \lambda \to \mu $ in $ \mathcal{V}_A $ is given by a commutative diagram in $ \mathcal{K} $
\[ \tikz[auto]{
	\node (UL) at (0,0) {$ A_{\lambda} $};
	\node (UR) at (2,0) {$ A_{\mu} $};
	\node (DL) at (1,1) {$ A $};
	\draw[->] (UL) to (UR);
	\draw[->] (DL) to node[swap] {$ \lambda $} (UL);
	\draw[->] (DL) to node {$ \mu $} (UR);
}. \]
Notice that $ \mathcal{V}_A $ is essentially small since $ \mathcal{V} $ and the class of spans of the form $ A \leftarrow M \xrightarrow{k} N $ with $ k \in \mathcal{V} $ are essentially small.

The following lemma is the most important tool, and we will frequently use it in the remainder of this paper.
\begin{Lem}[Lifting lemma, {\cite[Sub-lemma 12]{Anel2009}}] \label{lem:lifting-lemma}
	\ 
	
	\noindent If we have a diagram in $ \mathcal{K} $
	\[ \tikz[auto]{
		\node (UL) at (0,2) {$ M $};
		\node (UR) at (2,2) {$ N $};
		\node (DL) at (0,0) {$ A $};
		\node (DR) at (2,0) {$ B $};
		\node (Out) at (4,2) {$ L $};
		
		\draw[->] (UL) to node {$ k $} (UR);
		\draw[->] (UL) to (DL);
		\draw[->] (DL) to node {$ \lambda $} (DR);
		\draw[->] (UR) to (DR);
		\draw[->] (Out) to node {$ m $} (DR);
		
		\node (po) at (1.5,0.5) {$ \ulcorner $};
	},  \]
	with $ M,N,L \in \mathcal{K}_{\fp} $,
	then the pushout square can be factored into two pushouts 
	\[ \tikz[auto]{
		\node (UL) at (0,3) {$ M $};
		\node (UR) at (2,3) {$ N $};
		\node (ML) at (0,1.5) {$ M' $};
		\node (MR) at (2,1.5) {$ N' $};
		\node (DL) at (0,0) {$ A $};
		\node (DR) at (2,0) {$ B $};
		\node (Out) at (4,1.5) {$ L $};
		
		\draw[->] (UL) to node {$ k $} (UR);
		\draw[->] (UL) to (ML);
		\draw[->] (ML) to (DL);
		\draw[->] (DL) to node {$ \lambda $} (DR);
		\draw[->] (UR) to (MR);
		\draw[->] (MR) to (DR);
		\draw[->] (ML) to (MR);
		\draw[dashed,->] (Out) to (MR);
		\draw[->] (Out) to node {$ m $} (DR);
		
		\node (po1) at (1.5,0.5) {$ \ulcorner $};
		\node (po1) at (1.5,2) {$ \ulcorner $};
	}  \]
	such that $ M',N' \in \mathcal{K}_{\fp} $ and $ m $ lifts as above.
\end{Lem}
\begin{proof}
	Obviously, for any filtered category $ \mathcal{J} $ and any object $ j \in \mathcal{J} $,
	the comma category $ j \downarrow \mathcal{J} $ is again filtered, and the projection $ j \downarrow \mathcal{J} \to \mathcal{J} $ is a final functor.
	Thus, $ A $ can be represented as a filtered colimit $ \rlim_i M_i $ of finitely presentable objects 
	such that the indexing category of $ \{M_i\}_i $ has an initial object $ 0 $,
	and the given $ M \to A $ factors through the coprojection $ M_0 \to A $ (and hence every coprojection $ M_i \to A$).
	Let $ M_i \to N_i $ be a pushout of $ k $ along $ M \to M_i $ for each $ i $.
	Then, for each $ i $, the initially given pushout can be factored as
	\[ \tikz[auto]{
		\node (UL) at (0,3) {$ M $};
		\node (UR) at (2,3) {$ N $};
		\node (ML) at (0,1.5) {$ M_i $};
		\node (MR) at (2,1.5) {$ N_i $};
		\node (DL) at (0,0) {$ A $};
		\node (DR) at (2,0) {$ B $};
		
		\draw[->] (UL) to node {$ k $} (UR);
		\draw[->] (UL) to (ML);
		\draw[->] (ML) to (DL);
		\draw[->] (DL) to node {$ \lambda $} (DR);
		\draw[->] (UR) to (MR);
		\draw[->] (MR) to (DR);
		\draw[->] (ML) to (MR);
		
		\node (po1) at (1.5,0.5) {$ \ulcorner $};
		\node (po1) at (1.5,2) {$ \ulcorner $};
	}.  \]
	Moreover, by commutation of colimits, we have $ B \cong \rlim_i N_i $.
	Therefore, $ m $ factors through some $ N_i $ as desired.
\end{proof}

\begin{Lem}[{\cite[Proposition 3.3.2]{Coste1979}}] \label{lem:closure-properties-of-V_A}
	Let $ \mathcal{V} $ be a saturated class of morphisms in $ \mathcal{K}_{\fp} $.
	
	\begin{enumerate}
		\item If $ (\lambda \colon A \to A_{\lambda}) \in \mathcal{V}_A $ and $ (\mu \colon A_{\lambda} \to A_{\mu}) \in \mathcal{V}_{A_{\lambda}} $,
		then the composite $ \mu\lambda \colon A \to A_{\mu} $ is in $ \mathcal{V}_A $.
		
		\item If $ (\lambda \colon A \to A_{\lambda}) \in \mathcal{V}_A $ and $ (\mu\lambda \colon A \to A_{\mu}) \in \mathcal{V}_A $
		for a morphism $ \mu \colon A_{\lambda} \to A_{\mu} $, then $ \mu \in \mathcal{V}_{A_{\lambda}} $.
		
	\end{enumerate}
\end{Lem}
\begin{proof}
	\noindent (1) If there exists a pushout
	\[ \tikz[auto]{
		\node (UL) at (0,2) {$ L $};
		\node (UR) at (2,2) {$ K $};
		\node (DL) at (0,0) {$ A_{\lambda} $};
		\node (DR) at (2,0) {$ A_{\mu} $};
		\draw[->] (UL) to node {$ l \in \mathcal{V} $} (UR);
		\draw[->] (UL) to node {$ m $} (DL);
		\draw[->] (DL) to node {$ \mu $} (DR);
		\draw[->] (UR) to (DR);
		\node (po) at (1.5,0.5) {$ \ulcorner $};
	},  \]
	then, by the previous lemma, we can find a diagram
	\[ \tikz[auto]{
		\node (UL) at (0,2) {$ M $};
		\node (UR) at (2,2) {$ N $};
		\node (DL) at (0,0) {$ A $};
		\node (DR) at (2,0) {$ A_{\lambda} $};
		\node (Out) at (4,2) {$ L $};
		
		\draw[->] (UL) to node {$ k \in \mathcal{V} $} (UR);
		\draw[->] (UL) to (DL);
		\draw[->] (DL) to node {$ \lambda $} (DR);
		\draw[->] (UR) to (DR);
		\draw[->] (Out) to node {$ m $} (DR);
		\draw[dashed,->] (Out) to node {$ n $} (UR);
		
		\node (po) at (1.5,0.5) {$ \ulcorner $};
	}.  \]
	Consider the pushouts below
	\[ \tikz[auto]{
		\node (UL) at (0,3) {$ L $};
		\node (UR) at (2,3) {$ K $};
		\node (ML) at (0,1.5) {$ N $};
		\node (MR) at (2,1.5) {$ n_*K $};
		\node (DL) at (0,0) {$ A_{\lambda} $};
		\node (DR) at (2,0) {$ A_{\mu} $};
		
		\draw[->] (UL) to node {$ l $} (UR);
		\draw[->] (UL) to node {$ n $} (ML);
		\draw[->] (ML) to (DL);
		\draw[->] (DL) to node {$ \mu $} (DR);
		\draw[->] (UR) to (MR);
		\draw[->] (MR) to (DR);
		\draw[->] (ML) to (MR);
		\draw[bend right=40pt,->] (UL) to node[left] {$ m $} (DL);
		
		\node (po1) at (1.5,0.5) {$ \ulcorner $};
		\node (po1) at (1.5,2) {$ \ulcorner $};
	}.  \]
	Then, $ \mu\lambda $ is a pushout of $ (M \xrightarrow{k} N \to n_*K) \in \mathcal{V} $.
	
	\noindent (2) If there exists a pushout
	\[ \tikz[auto]{
		\node (UL) at (0,2) {$ M $};
		\node (UR) at (2,2) {$ N $};
		\node (DL) at (0,0) {$ A $};
		\node (DR) at (2,0) {$ A_{\lambda} $};
		
		\draw[->] (UL) to node {$ k \in \mathcal{V} $} (UR);
		\draw[->] (UL) to node {$ m $} (DL);
		\draw[->] (DL) to node {$ \lambda $} (DR);
		\draw[->] (UR) to node {$ n $} (DR);
		
		\node (po) at (1.5,0.5) {$ \ulcorner $};
	},  \]
	then, by an argument similar to the proof of the previous lemma, we can find a diagram
	\[ \tikz[auto]{
		\node (UL) at (0,2) {$ L $};
		\node (UR) at (2,2) {$ K $};
		\node (DL) at (0,0) {$ A $};
		\node (DR) at (2,0) {$ A_{\mu} $};
		\draw[->] (UL) to node {$ l \in \mathcal{V} $} (UR);
		\draw[->] (UL) to (DL);
		\draw[->] (DL) to node {$ \mu\lambda $} (DR);
		\draw[->] (UR) to node {$ h $} (DR);
		
		\node (Out) at (-2,2) {$ M $};
		\draw[->] (Out) to node {$ m $} (DL);
		\draw[dashed,->] (Out) to node {$ m' $} (UL);
		
		\node (po) at (1.5,0.5) {$ \ulcorner $};
	}.  \]
	Again by applying the lemma to $ \mu n $,
	we can replace $ l $, $m' $, etc., so that there exists a diagram
	\[ \tikzset{cross/.style={preaction={-,draw=white,line width=6pt}}}
	
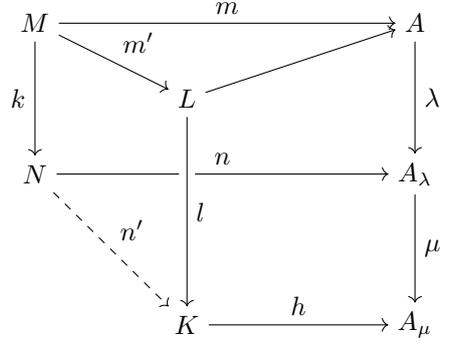
\tikz[auto]{
		\node (UL) at (0,2) {$ M $};
		\node (DL) at (0,0) {$ N $};
		\node (fUL) at (2,1) {$ L $};
		\node (fBL) at (2,-2) {$ K $};
		\node (UR) at (5,2) {$ A $};
		\node (DR) at (5,0) {$ A_{\lambda} $};
		\node (BR) at (5,-2) {$ A_{\mu} $};
		
		\draw[->] (UL) to node {$ m $} (UR);
		\draw[->] (UL) to node[left] {$ k $} (DL);
		\draw[->] (DL) to node[above] {$ n $} (DR);
		\draw[->] (UR) to node {$ \lambda $} (DR);
		
		\draw[->] (UL) to node {$ m' $}(fUL);
		\draw[dashed,->] (DL) to node {$ n' $} (fBL);
		\draw[->,cross] (fUL) to node {$ l $} (fBL);
		\draw[->] (fUL) to (UR);
		\draw[->] (fBL) to node {$ h $} (BR);
		\draw[->] (DR) to node {$ \mu $} (BR);
	}, \]
	where $ hn'=\mu n $.
	Since $ hn'k=hlm' $ and $ M \in \mathcal{K}_{\fp} $, we may further assume that $ n'k=lm' $.
	By taking a pushout of $ k $ along $ m' $, we obtain the desired diagram
	\[ \tikzset{cross/.style={preaction={-,draw=white,line width=6pt}}}
	
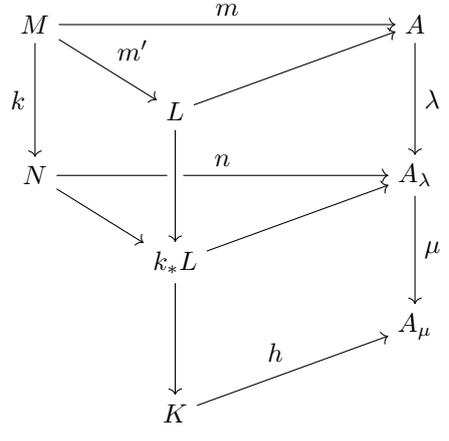
\tikz[auto]{
		\node (UL) at (0,2,0) {$ M $};
		\node (DL) at (0,0,0) {$ N $};
		\node (fUL) at (3,2,3) {$ L $};
		\node (fDL) at (3,0,3) {$ k_*L $};
		\node (fBL) at (3,-2,3) {$ K $};
		\node (UR) at (5,2,0) {$ A $};
		\node (DR) at (5,0,0) {$ A_{\lambda} $};
		\node (BR) at (5,-2,0) {$ A_{\mu} $};
		
		\draw[->] (UL) to node {$ m $} (UR);
		\draw[->] (UL) to node[left] {$ k $} (DL);
		\draw[->] (DL) to node[above] {$ n $} (DR);
		\draw[->] (UR) to node {$ \lambda $} (DR);
		
		\draw[->] (UL) to node {$ m' $}(fUL);
		\draw[->] (DL) to (fDL);
		\draw[->,cross] (fUL) to (fDL);
		\draw[->] (fUL) to (UR);
		\draw[->] (fDL) to (DR);
		\draw[->] (fDL) to (fBL);
		\draw[->] (fBL) to node {$ h $} (BR);
		\draw[->] (DR) to node {$ \mu $} (BR);
	}, \]
	where each vertical arrow between finitely presentable objects is in $ \mathcal{V} $, and each vertical square is a pushout.
	Here, we have used the right cancellation property of $ \mathcal{V} $ to show $ (k_*L \to K) \in \mathcal{V} $.
\end{proof}

\begin{Prop} \label{prop:V_A-is-cocartesian}
	$ \mathcal{V}_A $ is a cocartesian category, i.e.\ it has all finite colimits.
\end{Prop}
\begin{proof}
	The initial object of $ \mathcal{V}_A $ is the identity $ \id_A $.
	Given a commutative diagram in $ \mathcal{K} $
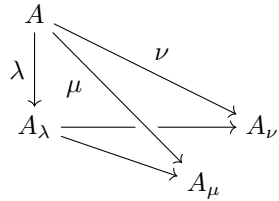
	\[ \tikzset{cross/.style={preaction={-,draw=white,line width=6pt}}}
	\tikz[auto]{
		\node (UL) at (0,0,0) {$ A_{\lambda} $};
		\node (DL) at (3,0,2) {$ A_{\mu} $};
		\node (UR) at (3,0,0) {$ A_{\nu} $};
		\node (T) at (0,1.5,0) {$ A $};
		
		\draw[->] (UL) to (UR);
		\draw[->] (UL) to (DL);
		
		\draw[->,cross] (T) to node[swap,pos=0.3] {$ \mu $} (DL);
		\draw[->] (T) to node {$ \nu $} (UR);
		\draw[->] (T) to node[swap] {$ \lambda $} (UL);
	} \]
	with $ \lambda,\mu,\nu \in \mathcal{V}_A $,
	the previous lemma implies that a pushout of the bottom span in $ \mathcal{K} $ yields a pushout in $ \mathcal{V}_A $.
\end{proof}

Based on these settings, Coste proved the following theorem:
\begin{Thm}[Factorization theorem] \label{thm:factorization-thm}
	For any morphism $ \alpha \colon A \to B $ in $ \mathcal{K} $,
	there exists a factorization
	\[ \tikz[auto]{
		\node (L) at (0,1.5) {$ A $};
		\node (C) at (2,0) {$ C $};
		\node (R) at (4,1.5) {$ B $};
		\draw[->] (L) to node {$ \alpha $} (R);
		\draw[->] (L) to node[swap] {$ \beta $} (C);
		\draw[->] (C) to node[swap] {$ \gamma $} (R);
	} \]
	with $ \gamma $ $ \mathcal{V} $-admissible, which is initial among factorizations $ A \to C' \xrightarrow{\gamma'} B $ of $ \alpha $ with $ \gamma' $ $ \mathcal{V} $-admissible.
\end{Thm}
\begin{proof}
	Consider the comma square in the $ 2 $-category $ \mathfrak{CAT} $ of categories
	\[ \tikz[auto]{
		\node (UL) at (0,2) {$ \mathcal{V}_A \downarrow \alpha $};
		\node (UR) at (2,2) {$ 1 $};
		\node (DL) at (0,0) {$ \mathcal{V}_A  $};
		\node (DR) at (2,0) {$ A \downarrow \mathcal{K} $};
		\node[rotate=45] (center) at (1,1) {$ \Rightarrow $};
		\draw[->] (UL) to (UR);
		\draw[->] (UL) to (DL);
		\draw[->] (DL) to (DR);
		\draw[->] (UR) to node {$ \alpha $} (DR);
	}.  \]
	Then $ \mathcal{V}_A \downarrow \alpha $ is filtered.
	Since filtered colimits in $ A \downarrow \mathcal{K} $ are created by the forgetful functor $ A \downarrow \mathcal{K} \to \mathcal{K} $,
	we have a factorization 
	\[ \tikz[auto]{
		\node (L) at (0,1.5) {$ A $};
		\node (C) at (2,0) {$ \displaystyle \rlim_{\lambda \in \mathcal{V}_A \downarrow \alpha} A_{\lambda} $};
		\node (R) at (4,1.5) {$ B $};
		\draw[->] (L) to node {$ \alpha $} (R);
		\draw[->] (L) to (C);
		\draw[->] (C) to (R);
	}. \]
	This is the desired factorization. All the remaining details can be found in \cite[\S2.3]{Anel2009}.
\end{proof}

We will not use this theorem at all, but we will occasionally point out its importance in relation to the general theory of topos-theoretic spectra.

\subsection{Coste Contexts} \label{subsec:Coste-context}
Next, let us consider the case when $ \mathcal{K} = T_0\mhyphen\mathbf{Mod} $ for a cartesian theory $ T_0 $.
By \emph{Proposition \ref{prop:generalities-on-C_T}} and the equivalence $ T_0\mhyphen\mathbf{Mod}_{\fp} \simeq (\mathcal{C}_{T_0})^{\op} $,
giving a morphism in $ T_0\mhyphen\mathbf{Mod}_{\fp} $ is equivalent to 
giving a pair $ (\varphi(\bvec{u}),\psi(\bvec{u},\bvec{v})) $ of Horn formulas such that $ T_0 \models \psi(\bvec{u},\bvec{v}) \vdash \varphi(\bvec{u}) $.
In order to give a setting for spectra, we also need to specify the class of ``local algebras'' which will appear as stalks of structure sheaves.
These observations lead to the following definition:

\begin{Def} \label{def:Coste-context}
	A \emph{Coste context} is a triple $ (T_0,T,\Lambda) $ such that
	\begin{itemize}
		\item $ T_0 $ is a cartesian $ \mathcal{L} $-theory,
		\item $ \Lambda $ is a set of pairs $ (\varphi(\bvec{u}),\psi(\bvec{u},\bvec{v})) $ of Horn $ \mathcal{L} $-formulas with $ T_0 \models \psi(\bvec{u},\bvec{v}) \vdash \varphi(\bvec{u}) $, and
		\item $ T $ is a coherent $ \mathcal{L} $-theory obtained by adding to $ T_0 $ some axioms of the form
		\[ \varphi(\bvec{u}) \vdash \bigvee_i \exists \bvec{v}_i \psi_i(\bvec{u},\bvec{v}_i) \]
		for finitely many pairs $ (\varphi(\bvec{u}),\psi_i(\bvec{u},\bvec{v}_i)) \in \Lambda $. \qedhere
	\end{itemize}
\end{Def}
Coste \cite{Coste1979} called it a \textit{localisation triple}.
The disjunction in the axiom of $ T $ can be the empty disjunction, i.e.\ an axiom can be of the form $ \varphi(\bvec{u}) \vdash \bot $ 
for some formula $ \varphi(\bvec{u}) $ which does not necessarily appear in $ \Lambda $.
We remark that the condition on $ T $ is not so restrictive since any coherent theory can be axiomatized in this form \cite[Proposition D1.3.10(iii)]{Elephant} for some $ \Lambda $.

Hereafter, we fix a Coste context $ (T_0,T,\Lambda) $.
Let $ \mathcal{V} \subseteq \Mor(T_0\mhyphen\mathbf{Mod}_{\fp}) $ be the saturated class generated by 
the morphisms $ M_{\varphi} \to M_{\psi} $ corresponding to $ \{\bvec{uv}.\,\psi\} \to \{\bvec{u}.\,\varphi\} $ for $ (\varphi,\psi) \in \Lambda $
(recall the notation in \emph{Proposition \ref{prop:free-model}}).
We will simply say admissible morphisms to mean $ \mathcal{V} $-admissible morphisms.
By \emph{Proposition \ref{prop:admissibility-is-invariant}}, a morphism $ \alpha \colon A \to B $ is admissible if and only if it has the unique lifting property
\[ \tikz[auto]{
	\node (UL) at (0,2) {$ M_{\varphi} $};
	\node (UR) at (2,2) {$ M_{\psi} $};
	\node (DL) at (0,0) {$ A $};
	\node (DR) at (2,0) {$ B $};
	\draw[->] (UL) to (UR);
	\draw[->] (UL) to node[swap] {$ \bvec{a} $} (DL);
	\draw[->] (DL) to node {$ \alpha $} (DR);
	\draw[->] (UR) to node {$ \langle \alpha(\bvec{a}),\bvec{b} \rangle $} (DR);
	\draw[dashed,->] (UR) to node[swap] {$ \exists! $} (DL);
}  \]
for each $ (\varphi,\psi) \in \Lambda $.
In other words, if $ A \models \varphi(\bvec{a}) $ and $ B \models \psi(\alpha(\bvec{a}), \bvec{b}) $,
then there exists a unique tuple $ \bvec{a'} \in A $ such that $ A \models \psi(\bvec{a},\bvec{a'}) $ and $ \alpha(\bvec{a'}) = \bvec{b} $.
Therefore, $ \alpha \colon A \to B $ is admissible iff the diagram in $ \mathbf{Set} $
\[ \tikz[auto]{
	\node (UL) at (0,2) {$ \dbracket{\psi}_A $};
	\node (UR) at (2,2) {$ \dbracket{\psi}_B $};
	\node (DL) at (0,0) {$ \dbracket{\varphi}_A $};
	\node (DR) at (2,0) {$ \dbracket{\varphi}_B $};
	\draw[->] (UL) to node {$ \alpha $} (UR);
	\draw[->] (UL) to (DL);
	\draw[->] (DL) to node {$ \alpha $} (DR);
	\draw[->] (UR) to (DR);
}  \]
is a pullback for each $ (\varphi,\psi) \in \Lambda $.

\begin{Lem} \label{lem:adm-map-reflect-T-models}
	Admissible morphisms reflect the property of being a $ T $-model.
\end{Lem}
\begin{proof}
	Suppose that $ \alpha \colon A \to B $ is admissible, and that $ B $ is a $ T $-model.
	Let $ \varphi \vdash \bigvee_i \exists \bvec{v}_i \psi_i(\bvec{v}_i) $ be an additional axiom in $ T $.
	If $ A \models \varphi(\bvec{a}) $, then $ B \models \varphi(\alpha(\bvec{a})) $, and there exists $ i $ such that $ B \models \exists \bvec{v}_i \psi_i(\alpha(\bvec{a}), \bvec{v}_i) $.
	Choose $ \bvec{b} \in B $ so that $ B \models \psi_i(\alpha(\bvec{a}), \bvec{b}) $.
	Then, by admissibility, there exists a unique tuple $ \bvec{a'} \in A $ such that $ A \models \psi_i(\bvec{a},\bvec{a'}) $ and $ \alpha(\bvec{a'}) = \bvec{b} $.
	Thus, we showed $ A \models \varphi \vdash \bigvee_i \exists \bvec{v}_i \psi_i(\bvec{v}_i) $.
	Notice that, if $ (\varphi(\bvec{u}) \vdash \bot) \in T$,
	then $ B \models \varphi(\alpha(\bvec{a})) \vdash \bot $ implies $ A \models \varphi(\bvec{a}) \vdash \bot $ for an arbitrary homomorphism $ \alpha $.
\end{proof}

The factorization theorem shows that any homomorphism $ \alpha \colon A \to B $ with $ B $ a $ T $-model has a factorization 
$ A \xrightarrow{\beta} C \xrightarrow{\gamma} B $ initial among those with $ \gamma $ admissible.
By the proof of the theorem, $ C $ is a filtered colimit of $ A_{\lambda} $'s.
The previous lemma shows that $ C $ is a $ T $-model.
A morphism $ A \to C $ with $ C \models T $ which is a filtered colimit of $ A_{\lambda} $'s
is called \textit{a localization of $ A $} in \cite[Definition 3.4.2]{Coste1979}.
A filtered diagram in $ \mathcal{V}_A $ producing a localization of $ A $ will be considered as ``a point of $ \Spec(A) $''
(\emph{Definition \ref{def:Spec(A)}}; cf.\ \cite[Proposition 4.1.12]{Osm2021a}).

We now provide several known examples of Coste contexts and the corresponding spectra.
Each spectrum is equipped with a structure sheaf of $ T $-models, and here we also describe its stalks.
We will explain in the next section the general construction of spectra and concrete descriptions of $ \mathcal{V}_A $ for the Zariski and DL contexts.
The first example is the trivial Coste context $ (T_0,T_0,\varnothing) $, for which any homomorphism is admissible.

\begin{Expl} \label{expl:Zariski-context}
	Let $ T_0 $ be the theory of commutative rings.
	The theory $ T $ of local rings contains the additional axioms
	\[ 0=1 \vdash \bot, \qquad  u=u \vdash \exists v(uv=1) \lor \exists v ((1-u)v=1). \]
	Put	$ \Lambda = \{(u=u,uv=1),\,(u=u,(1-u)v=1) \} $.
	Then, admissible homomorphisms are homomorphisms reflecting invertible elements.
	We will call this triple the \emph{Zariski context}.
	
	Of course, spectra for the Zariski context are the Zariski spectra $ \Spec_Z(A) $.
	It is the set of prime ideals of $ A $ equipped with the topology whose basic open sets are 
	$ D_a := \Set{\mathfrak{p} \in \Spec_Z(A) \smid a \notin \mathfrak{p}} $ for $ a \in A $.
	The stalk of the structure sheaf at a prime ideal $ \mathfrak{p} $ is the localization $ A_{\mathfrak{p}} $.
\end{Expl}

\begin{Expl} \label{expl:contexts-for-rings}
	The following are examples of Coste contexts for which we will not describe the details of their spectra.
	The details can be found in \cite[Chapter V]{JohSS} (see also \cite{Kenn1976} and \cite{Joh1977}).
	Let $ T_0 $ be the theory of commutative rings.
	\begin{enumerate}
		\item The theory $ T $ of integral domains contains the additional axioms
		\[ 0=1 \vdash \bot, \qquad uu'=0 \vdash u=0 \lor u'=0. \]
		Put $ \Lambda := \{(uu'=0,u=0),\,(uu'=0,u'=0) \} $.
		The quantification $ \exists v $ is suppressed for a variable $ v $ distinct from $ u,u' $.
		Then admissible homomorphisms are injective homomorphisms.
		Endow the set $ \Spec_Z(A) $ with the ``inverse topology'' of the Zariski spectrum.
		The stalk at $ \mathfrak{p} $ is $ A/\mathfrak{p} $.
		This spectrum is called the domain spectrum of $ A $.
		
		\item The theory $ T $ of fields contains the additional axioms
		\[ 0=1 \vdash \bot, \qquad u=u \vdash u=0 \lor \exists v (uv=1). \]
		Put $ \Lambda := \{(u=u,u=0),\,(u=u,uv=1) \} $.
		Then admissible homomorphisms are injective homomorphisms reflecting invertible elements.
		Here, every homomorphism between fields is admissible.
		Endow the set $ \Spec_Z(A) $ with the ``patch topology'' of the Zariski spectrum.
		The stalk at $ \mathfrak{p} $ is the fraction field of $ A/\mathfrak{p} $, which is also isomorphic to the residue field of $ A_{\mathfrak{p}} $.
		This spectrum is called the field spectrum of $ A $.
		
		\item The theory $ T $ of indecomposable rings contains the additional axioms
		\[ 0=1 \vdash \bot, \qquad u^2=u \vdash u=0 \lor u=1. \]
		Put $ \Lambda := \{(u^2=u,u=0),\,(u^2=u,u=1) \} $.
		Then admissible homomorphisms are homomorphisms injective on idempotents.
		The underlying space of the spectrum is the Stone space of the Boolean algebra of idempotents of $ A $.
		The stalk at an ultrafilter $ \mathfrak{u} $ is the localization $ A[\mathfrak{u}^{-1}] $. 
		This is the Pierce spectrum of $ A $ introduced in \cite{Pierce1967}. \qedhere
	\end{enumerate}
\end{Expl}

\begin{Expl} \label{expl:DL-context}
	Let $ T_0 $ be the theory of (bounded) distributive lattices.
	To avoid confusion between joins/meets and the logical connectives $ \lor $/$ \land $, we use $ \sup $/$ \inf $ to denote joins/meets.
	The theory $ T $ of local lattices contains the additional axioms
	\[ 0=1 \vdash \bot, \qquad \sup(u,u') =1 \vdash u=1 \lor u'=1. \]
	Put $ \Lambda = \{ (\sup(u,u')=1,u=1),\,(\sup(u,u')=1,u'=1)\} $.
	Then admissible homomorphisms are homomorphisms reflecting $ 1 $.
	We will call this triple the \emph{DL context}.
	
	Let $ L $ be a distributive lattice.
	The set $ \Spec_D(L) $ of prime filters of $ L $, equipped with the topology generated by the basic open subsets 
	$ D(a) := \Set{\mathfrak{p} \in \Spec_D(L) \smid a \in \mathfrak{p}} $ for $ a \in L $, is called the \emph{(prime filter) spectrum of $ L $}.
	An open set of $ \Spec_D(L) $ is identified with an ideal of $ L $:
	the poset $ \mathcal{O}(\Spec_D(L)) $ of open subsets and the poset $ \Idl(L) $ of ideals are isomorphic.
	As we will see in \emph{Proposition \ref{prop:spec(A)-is-spectral}},
	any spectrum in our setting will be homeomorphic to $ \Spec_D(L) $ for some $ L $.
	The stalk at $ \mathfrak{p} $ is the filter-quotient $ L/\mathfrak{p} $, i.e.\ the quotient of $ L $ by the equivalence relation $ a \sim b \iff \exists c \in \mathfrak{p},\, \inf(a,c)=\inf(b,c) $.
\end{Expl}
Some other examples can be found in \cite{Coste1979}, \cite[pp.~205--209]{JohTT} and \cite[Chapter 9]{OsmondThesis} (cf.\ \cite[\S4]{Anel2009}).
Notice that, in general, the model of global sections of the structure sheaf of $ \Spec(A) $ need not recover the original model $ A $.
This holds for the Zariski context, the DL context, and Pierce spectra.
In \S\ref{subsec:standard-context}, we will give criteria for a $ T_0 $-model to be isomorphic to the global section model of the structure sheaf.

\section{Spatial Coste Contexts and Spectra} \label{sec:spatial-contexts}
In this section, we will construct spectra for set-based models under the ``spatiality'' assumption.
Here, we will only give the construction and will not prove that it constitutes a dual adjunction.
In \S\ref{sec:spectra}, we will prove it from a more general ``relative spectral adjunction.''
We refer the reader to \cite{Osm2021a} for topos-theoretic spectra, for arbitrary Coste contexts,
of set-based models and models in Grothendieck toposes given by explicit sites.
Subsections \ref{subsec:comparison-with-Coste-spectra} and \ref{subsec:standard-context} are devoted to
proving the equivalence of our and Coste's original constructions and 
giving criteria for a Coste context so that every model can be recovered as the model of global sections of the structure sheaf.

\subsection{Categories of Modelled Spaces} \label{subsec:introducing-modelled-spaces}
Let $ (T_0,T,\Lambda) $ be a Coste context.
We introduce categories involved in spectral adjunctions.

\begin{Def}
	A sheaf $ P \colon \mathcal{O}(X)^{\op} \to \mathbf{Set} $ on a topological space $ X $ is called a \emph{sheaf of $ \mathcal{L} $-structures on $ X $}
	when each $ PU $ is equipped with data to be an $ \mathcal{L} $-structure and each restriction map $ PU \to PV $ is a homomorphism.
	Such $ P $ is called a \emph{sheaf of $ T_0 $-models on $ X $} when each $ PU $ is a $ T_0 $-model.
	In other words, it is a functor $ P \colon \mathcal{O}(X)^{\op} \to T_0\mhyphen\mathbf{Mod} $ satisfying the sheaf condition.
	
	For sheaves $ P,Q $ of $ T_0 $-models, a morphism $ \eta \colon P \to Q $ of sheaves is called a \emph{morphism of sheaves of $ T_0 $-models} 
	when each component $ \eta_U \colon PU \to QU $ is a homomorphism.
	We will write $ T_0\mhyphen\mathbf{Mod}(\mathbf{Sh}(X)) $ for the category of sheaves of $ T_0 $-models on $ X $.
\end{Def}
A sheaf of $ T_0 $-models on $ X $ (resp.\ a morphism of sheaves of $ T_0 $-models) is the same thing 
as a $ T_0 $-model (resp.\ a homomorphism of $ T_0 $-models) in the topos $ \mathbf{Sh}(X) $ of set-valued sheaves.
For more explanation of sheaves of structures, see \cite{Ara2021}.

Given a continuous map $ f \colon X \to Y $ and a sheaf $ Q $ of $ T_0 $-models on $ Y $,
the inverse image sheaf $ f^*Q $ is canonically a sheaf of $ T_0 $-models on $ X $.
For a sheaf $ P $ of $ T_0 $-models on $ X $, we also have the direct image sheaf $ f_*P $ on $ Y $, which is again a sheaf of $ T_0 $-models on $ Y $.
These are consequences of the fact that both the inverse image functor $ f^* \colon \mathbf{Sh}(Y) \to \mathbf{Sh}(X) $ 
and the direct image functor $ f_* \colon \mathbf{Sh}(X) \to \mathbf{Sh}(Y) $ are cartesian.

\begin{Def}
	A \emph{$ T_0 $-modelled space} is a pair $ (X,P) $ of a topological space $ X $ and a sheaf $ P $ of $ T_0 $-models on $ X $.
	A \emph{morphism of $ T_0 $-modelled spaces} from $ (X,P) $ to $ (Y,Q) $
	is a pair $ (f,f^{\flat}) $ of a continuous map $ f \colon X \to Y $ and a morphism $ f^{\flat} \colon f^*Q \to P $ of sheaves of $ T_0 $-models on $ X $.
	The composite of $ (f,f^{\flat}) \colon (X,P) \to (Y,Q) $ and $ (g,g^{\flat}) \colon (Y,Q) \to (Z,R) $
	is given by the pair $ (gf , (gf)^*R \cong f^*(g^* R) \xrightarrow{f^*(g^{\flat})} f^*Q \xrightarrow{f^{\flat}} P) $.
	Let $ \modsp{T_0}{\mathbf{Sp}} $ denote the category of $ T_0 $-modelled spaces.
	We will often drop $ f^{\flat} $ in notation and simply say ``$ f \colon (X,P) \to (Y,Q) $ is a morphism in $ \modsp{T_0}{\mathbf{Sp}} $.''
\end{Def}
Let $ (X,P),(Y,Q) $ be as above. Suppose that a morphism $ f^{\flat} \colon f^*Q \to P $ in $ \mathbf{Sh}(X) $
corresponds to $ f^{\sharp} \colon Q \to f_*P $ in $ \mathbf{Sh}(Y) $ under the adjunction $ f^* \dashv f_* $.
If one of them is a morphism of sheaves of $ T_0 $-models, then so is the other.
Therefore, the adjunction $ f^* \dashv f_* $ can be extended to an adjunction between 
$ T_0\mhyphen\mathbf{Mod}(\mathbf{Sh}(X)) $ and $ T_0\mhyphen\mathbf{Mod}(\mathbf{Sh}(Y)) $.

The following lemma makes it easy to obtain a morphism of modelled spaces:
\begin{Lem}
	Let $ P,Q $ be sheaves of $ T_0 $-models on a space $ X $.
	Then, giving a morphism $ \eta \colon P \to Q $ of sheaves of $ T_0 $-models is equivalent to 
	giving a family of homomorphisms $ \{\eta_x \colon P_x \to Q_x \}_{x \in X} $ 
	between stalks with the following property:
	for every open subset $ U \subseteq X $ and $ s \in PU $, there exist an open covering $ \{U_i\}_i $ of $ U $ and sections $ t_i \in QU_i $
	such that $ (t_i)_x = \eta_x(s_x) $ for any $ i $ and $ x \in U_i $.
\end{Lem}
\begin{proof}
	Represent $ P,Q $ as \'{e}tal\'{e} spaces over $ X $, and consider the condition for $ \eta $ to be continuous:
	note that $ \eta $ is continuous iff $ \eta \circ s $ is continuous for each $ s \in PU $.
\end{proof}

When we consider the Zariski context, $ T_0 $-modelled spaces are the same as ringed spaces.
Recall that a morphism of locally ringed spaces is not just a morphism of sheaves of rings but stalkwise a local homomorphism.
A morphism $ \eta $ of sheaves of rings is stalkwise local exactly when each component $ \eta_U $ reflects invertible elements.
The following definition is a natural generalization:
\begin{Def}
	We say a morphism $ \eta \colon P \to Q $ in $ T_0\mhyphen\mathbf{Mod}(\mathbf{Sh}(X)) $ is \emph{admissible}
	if each component $ \eta_U \colon PU \to QU $ is admissible in the sense we introduced before \emph{Lemma \ref{lem:adm-map-reflect-T-models}}.
\end{Def}

It is easy to see that admissibility of morphisms is preserved by the inverse image functor 
$ f^* \colon T_0\mhyphen\mathbf{Mod}(\mathbf{Sh}(Y)) \to T_0\mhyphen\mathbf{Mod}(\mathbf{Sh}(X)) $.
For any coherent theory $ T $, the notion of $ T $-models can be generalized to those in any coherent category (in particular, any topos),
and $ f^* $ also preserves $ T $-models.
Thus, we have the following definition:
\begin{Def}
	A $ T_0 $-modelled space $ (X,P) $ is called a \emph{$ T $-modelled space} when $ P $ is a $ T $-model in $ \mathbf{Sh}(X) $ (cf.\ the lemma below).
	Let $ \modsp{\mathbb{A}}{\mathbf{Sp}} $ denote the subcategory of $ \modsp{T_0}{\mathbf{Sp}} $ consisting of $ T $-modelled spaces and
	morphisms $ f \colon (X,P) \to (Y,Q) $ with $ f^{\flat} $ admissible.
\end{Def}

We can reduce these definitions to more elementary conditions and use the latter only.
\begin{Lem} \label{lem:conditions-for-being-in-A-ModSp}
	Let $ T $ be a coherent $ \mathcal{L} $-theory, and $ P,Q $ be sheaves of $ T_0 $-models on a space $ X $.
	\begin{enumerate}
		\item $ P $ is a $ T $-model in $ \mathbf{Sh}(X) $ iff each stalk $ P_x $ is a $ T $-model.
		
		\item A morphism $ \eta \colon P \to Q $ in $ T_0\mhyphen\mathbf{Mod}(\mathbf{Sh}(X)) $ is admissible
		iff $ \eta_x \colon P_x \to Q_x $ is admissible for each point $ x \in X $.
		Consequently, if $ \eta $ is admissible and $ Q $ is a $ T $-model in $ \mathbf{Sh}(X) $, then so is $ P $.
	\end{enumerate}
\end{Lem}
\begin{proof}
	(1) See \cite[Corollary D1.2.14]{Elephant}.
	
	\noindent (2) For any cartesian formula $ \varphi(\bvec{u}) $, 
	the interpretation $ \dbracket{\varphi}_{P} $ is the sheaf on $ X $ sending $ U $ to the solution set $ \dbracket{\varphi}_{PU} \cong \Hom(M_{\varphi},PU) $.
	$ \eta $ is admissible if and only if the following diagram in $ \mathbf{Sh}(X) $ is a pullback for each $ (\varphi,\psi) \in \Lambda $:
	\[ \tikz[auto]{
		\node (UL) at (0,2) {$ \dbracket{\psi}_P $};
		\node (UR) at (2,2) {$ \dbracket{\psi}_Q $};
		\node (DL) at (0,0) {$ \dbracket{\varphi}_P $};
		\node (DR) at (2,0) {$ \dbracket{\varphi}_Q $};
		\draw[->] (UL) to node {$ \eta $} (UR);
		\draw[->] (UL) to (DL);
		\draw[->] (DL) to node {$ \eta $} (DR);
		\draw[->] (UR) to (DR);
	}.  \]
	(cf.\ the description of admissible homomorphisms before \emph{Lemma \ref{lem:adm-map-reflect-T-models}})
	By essentially the same reason as (1), this condition holds exactly when the diagram in $ \mathbf{Set} $
	\[ \tikz[auto]{
		\node (UL) at (0,2) {$ \dbracket{\psi}_{P_x} $};
		\node (UR) at (2,2) {$ \dbracket{\psi}_{Q_x} $};
		\node (DL) at (0,0) {$ \dbracket{\varphi}_{P_x} $};
		\node (DR) at (2,0) {$ \dbracket{\varphi}_{Q_x} $};
		\draw[->] (UL) to node {$ \eta_x $} (UR);
		\draw[->] (UL) to (DL);
		\draw[->] (DL) to node {$ \eta_x $} (DR);
		\draw[->] (UR) to (DR);
	}  \]
	is a pullback for each $ x \in X $ and $ (\varphi,\psi) \in \Lambda $.
	This means that each $ \eta_x $ is admissible.
	The last assertion immediately follows from (1) and \emph{Lemma \ref{lem:adm-map-reflect-T-models}}.
\end{proof}

\begin{Rmk}
	Though $ T_0\mhyphen\mathbf{Mod}(\mathbf{Sh}(X)) $ is not necessarily locally finitely presentable,
	the factorization theorem can be generalized to morphisms in $ T_0\mhyphen\mathbf{Mod}(\mathbf{Sh}(X)) $,
	and such factorizations are stable under inverse image functors.
	This is the key feature of Coste's framework, which enables us to apply the $ 2 $-categorical construction of topos-theoretic spectra by Cole \cite{Cole2016} (cf.\ \cite[pp.~205--209]{JohTT}).
\end{Rmk}

Taking global sections yields the functor $ \Gamma \colon \modsp{T_0}{\mathbf{Sp}} \to T_0\mhyphen\mathbf{Mod}^{\op} $ sending
\[ f \colon (X,P) \to (Y,Q) \qquad \mapsto \qquad (f^{\sharp})_Y \colon QY \to PX. \]
We also write $ \Gamma \colon \modsp{\mathbb{A}}{\mathbf{Sp}} \to T_0\mhyphen\mathbf{Mod}^{\op} $ for the restriction of the above functor.
Our first goal in this section is to construct the remaining part of the adjunction
\[ \tikz{
	\node (L) at (0,0) {$ \modsp{\mathbb{A}}{\mathbf{Sp}} $};
	\node (R) at (4,0) {$ T_0\mhyphen\mathbf{Mod}^{\op} $};
	\draw[bend left=15pt,->] (L) to node[above] {$ \Gamma $} node[midway,name=U] {} (R);
	\draw[bend left=15pt,->] (R) to node[below] {$ \Spec $} node[midway,name=D] {} (L);
	\path (D) to node[sloped] {$ \vdash $} (U);
}.  \]
For a $ T_0 $-model $ A $, $ \Spec(A) $ will be called the \emph{spectrum of $ A $}.
We cannot obtain such an adjunction for arbitrary Coste contexts, but it suffices to impose the following condition:
\begin{Def}[{\cite[4.4.1]{Coste1979}}]
	A Coste context $ (T_0,T,\Lambda) $ is \emph{spatial} if each $ (\varphi(\bvec{u}),\psi(\bvec{u},\bvec{v})) \in \Lambda $ satisfies
	\[ T_0 \models \psi(\bvec{u},\bvec{v}) \land \psi(\bvec{u},\bvec{v'}) \vdash \bvec{v}=\bvec{v'}. \]
	In other words, the morphism $ \{\bvec{uv}.\,\psi \} \to \{\bvec{u}.\,\varphi \} $ in $ \mathcal{C}_{T_0} $ is monic (cf.\ \emph{Proposition \ref{prop:generalities-on-C_T}}),
	or, equivalently, the corresponding $ M_{\varphi} \to M_{\psi} $ in $ \mathcal{V} $ is epic.
\end{Def}

\begin{Expl}
	All the examples of Coste contexts we gave in the previous section are spatial.
	An intended omission there is the real \'{e}tale spectrum of a ring (\cite{CosteRoy1979b}, \cite{CosteRoy1980a}).
	It can be considered as a generic \textit{separably real closed local ring}.
	Several coherent axiomatizations of separably real closed local rings are proposed (\cite{Kock1979a}, \cite{JoyRey1986}).
	Strictly speaking, these theories seemingly do not fall into our setting since these axiomatizations are not in the form of \emph{Definition \ref{def:Coste-context}}.
	Nevertheless, after rephrasing Coste contexts categorically as in \cite[Definition 2.0.2]{Osm2021a}, we can argue real \'{e}tale spectra as in \cite{CosteRoy1979b},
	while spatiality of real \'{e}tale spectra is non-trivial \cite{CosteRoy1980a}.
	The same comments apply to the \'{e}tale spectra of (local) rings by \cite{HakimThesis} (cf.\ \cite{Wra1976}),
	which are generic \textit{separably closed local rings}, although in this case the involved Coste context is not spatial.
\end{Expl}

In the next subsection, we will provide a concrete description of Coste's construction of spectra in the spatial case,
while we will not prove that $ (\Gamma,\Spec) $ constitutes an adjunction.
The proof is postponed to \S\ref{sec:spectra}, where we will extend the spectrum construction to obtain an adjunction
\[ \tikz{
	\node (L) at (0,0) {$ \modsp{\mathbb{A}}{\mathbf{Sp}} $};
	\node (R) at (4,0) {$ \modsp{T_0}{\mathbf{Sp}} $};
	\draw[bend left=15pt,->] (L) to node[above] {forgetful} node[midway,name=U] {} (R);
	\draw[bend left=15pt,->] (R) to node[below] {$ \Spec $} node[midway,name=D] {} (L);
	\path (D) to node[sloped] {$ \vdash $} (U);
}.  \]

\subsection{Spectra of $ T_0 $-Models in $ \mathbf{Set} $} \label{subsec:spectra-of-models}
We fix a spatial Coste context $ (T_0,T,\Lambda) $ and a $ T_0 $-model $ A $.
Let us proceed to the construction of the spectrum $ \Spec(A) $ of $ A $.
Observe that every morphism in $ \mathcal{V} $ is epic, for $ \mathcal{V} $ is generated by epimorphisms $ M_{\varphi} \twoheadrightarrow M_{\psi} $
and the completion process in \emph{Proposition \ref{prop:admissibility-is-invariant}} adds epimorphisms only.
Hence, any $ \lambda \in \mathcal{V}_A $ is epic too, and $ \mathcal{V}_A $ is categorically equivalent to a small preorder.
We showed in \emph{Proposition \ref{prop:V_A-is-cocartesian}} that $ \mathcal{V}_A $ is cocartesian and thus obtain the following:

\begin{Prop*}
	$ \mathcal{V}_A $ is categorically equivalent to a join-semilattice.
\end{Prop*}

Hereafter, we will always regard it as a join-semilattice. Taking pushouts in $ T_0\mhyphen\mathbf{Mod} $ yields the join-operation.
An ideal of a join-semilattice is a non-empty lower subset closed under binary joins.
The poset $ \Idl(\mathcal{V}_A) $ of ideals of $ \mathcal{V}_A $ is an algebraic lattice, and the subsets
\[ D_{\lambda} := \Set{I \in \Idl(\mathcal{V}_A) \smid \lambda \in I } \qquad \text{(for each $ \lambda \in \mathcal{V}_A $)} \]
form a basis of a topology which makes $ \Idl(\mathcal{V}_A) $ a topological meet-semilattice, i.e.\ the meet-operation is continuous (cf.\ \cite[Corollary VI.3.6]{JohSS}).

For each ideal $ I \in \Idl(\mathcal{V}_A)$, let $ A_I $ denote the colimit
\[ \rlim_{\lambda \in I} A_{\lambda} \]
in $ T_0\mhyphen\mathbf{Mod} $.
Since $ I $ is directed, the composite of $ A \xrightarrow{\lambda} A_{\lambda} \to A_I $ does not depend on the choice of $ \lambda $.
Here is the definition of the underlying space of $ \Spec(A) $:

\begin{Def} \label{def:Spec(A)}
	$ \Spec(A) $ is the subspace of $ \Idl(\mathcal{V}_A) $ which consists of the ideals $ I $ with $ A_I $ a $ T $-model.
\end{Def}
We will also write $ D_{\lambda} $ for the basic open set $ \Set{I \in \Spec(A) \smid \lambda \in I } $ if no confusion arises.

Let us give necessary and sufficient conditions for $ A_I $ to be a $ T $-model.
In the statement and proof below, let $ k_i \colon M_{\varphi} \to M_{\psi_i}$ be the morphism corresponding to $ (\varphi,\psi_i) \in \Lambda $.

\begin{Prop} \label{prop:char-of-ideals-in-Spec}
	For an ideal $ I $ of $ \mathcal{V}_A $, the following are equivalent:
	\begin{enumerate}[label=(\roman*)]
		\item $ A_I $ is a $ T $-model.
		
		\item For each additional axiom $ \varphi(\bvec{u}) \vdash \bigvee_i \exists \bvec{v}_i \psi_i(\bvec{u},\bvec{v}_i) \in T $,
		\[ \forall \lambda \in I,\, \forall \bvec{a} \in A_{\lambda},\, \left [ 
		A_{\lambda} \models \varphi(\bvec{a}) \implies \exists \mu \in I,\, \exists i,\, \mu \geq \lambda \;\text{and}\; 
		A_{\mu} \models \exists \bvec{v}_i \psi_i(\bvec{a'},\bvec{v}_i) \right ], \]
		where $ \bvec{a'} \in A_{\mu} $ is the image of $ \bvec{a} $ under the morphism $ A_{\lambda} \to A_{\mu} $.
		
		\item For each additional axiom $ \varphi(\bvec{u}) \vdash \bigvee_i \exists \bvec{v}_i \psi_i(\bvec{u},\bvec{v}_i) \in T $, 
		$ \lambda \in I $ and $ \bvec{a} \colon M_{\varphi} \to A_{\lambda} $,
		if we write $ \mu_i \in \mathcal{V}_A $ for the composite of $ \lambda $ and the pushout of $k_i $ along $ \bvec{a} $,
		then at least one of $ \mu_i $'s belongs to $ I $.
	\end{enumerate}
\end{Prop}
\begin{proof}
	(i)$ \Leftrightarrow $(ii):
	For any cartesian formula $ \varphi(\bvec{u}) $ and $ \bvec{a} \in A_I $ satisfying $ \varphi $, 
	the corresponding morphism $ \bvec{a} \colon M_{\varphi} \to A_I $ factors through some $ A_{\lambda} \to A_I $ by finite presentability, and we can see
	\[ A_I \models \varphi(\bvec{a}) \iff \exists \lambda \in I,\, \exists \bvec{a'} \in A_{\lambda} ,\, A_{\lambda} \models \varphi(\bvec{a'}) \;\text{and $ A_{\lambda} \to A_I $ sends $ \bvec{a'} $ to $ \bvec{a} $.} \]
	Let $ \varphi \vdash \bigvee_i \exists \bvec{v}_i \psi_i(\bvec{v}_i) $ be an axiom of $ T $. Then,
	\[ A_I \models \varphi \vdash \bigvee_i \exists \bvec{v}_i \psi_i(\bvec{v}_i) \iff \text{every $ \bvec{a} \colon M_{\varphi} \to A_I $ factors through some $ k_i $.} \]
	An argument as above shows that the latter condition is equivalent to the condition
	\begin{multline*}
		\forall \lambda \in I,\, \forall \bvec{a} \colon M_{\varphi} \to A_{\lambda},\, \exists \mu \in I,\, \exists i,\, \\
		\mu \geq \lambda \;\text{and the composite of}\; M_{\varphi} \xrightarrow{\bvec{a}} A_{\lambda} \to A_{\mu} \;\text{factors through $ k_i $}.
	\end{multline*}
	This is the same thing as (ii).
	
	\noindent (ii)$ \Rightarrow $(iii):
	Suppose $ I $ satisfies (ii).  For $ \varphi \vdash \bigvee_i \exists \bvec{v}_i \psi_i \in T $, $ \lambda \in I $ and $ \bvec{a} \colon M_{\varphi} \to A_{\lambda} $,
	there exists a commutative diagram
	\[ \tikz[auto]{
		\node (UL) at (0,2) {$ M_{\varphi} $};
		\node (UR) at (2,2) {$ M_{\psi_i} $};
		\node (DL) at (0,0) {$ A_{\lambda} $};
		\node (DR) at (2,0) {$ A_{\lambda'} $};
		\draw[->] (UL) to node {$ k_i $} (UR);
		\draw[->] (UL) to node {$ \bvec{a} $} (DL);
		\draw[->] (DL) to (DR);
		\draw[->] (UR) to (DR);
	},  \]
	for some $ i $ and $ \lambda' \in I $ with $ \lambda' \geq \lambda $.
	By the construction of $ \mu_i $, we have $ \lambda \leq \mu_i \leq \lambda' $, and hence $ \mu_i \in I $.
	
	\noindent (iii)$ \Rightarrow $(ii): Straightforward.
\end{proof}

Notice that the condition (iii) does not imply that the largest ideal $ \mathcal{V}_A $ is in $ \Spec(A) $.
This is because an axiom of $ T $ can be of the form $ \varphi(\bvec{u}) \vdash \bot $, where $ \bot $ is the empty disjunction.
For such an axiom, if there exist $ \lambda $ and a morphism $ M_{\varphi} \to A_{\lambda} $, then $ \mathcal{V}_A \notin \Spec(A) $.
We also remark that the singleton set $ \{\id_A\} $ is in $ \Spec(A) $ if $ A $ itself is a $ T $-model.

\begin{Expl} \label{expl:V_A-and-ideals-in-examples}
	We describe $ \mathcal{V}_A $ and ideals in $ \Spec(A) $ for the Zariski context and the DL context (\emph{Examples \ref{expl:Zariski-context}, \ref{expl:DL-context}}).
	
	\noindent\textbf{Zariski context:}
	In this case, $ \modsp{T_0}{\mathbf{Sp}} $ is the category $ \mathbf{RS} $ of ringed spaces,
	and $ \modsp{\mathbb{A}}{\mathbf{Sp}} $ is the category $ \mathbf{LRS} $ of locally ringed spaces.
	The class $ \mathcal{V} $ is generated by the homomorphism $ \mathbb{Z}[X] \to \mathbb{Z}[X,X^{-1}] $.
	Let $ A $ be a ring. Then $ \mathcal{V}_A $ is (equivalent to) the poset of localizations $ A \to A[a^{-1}] $ for $ a \in A $.
	In other words, $ \mathcal{V}_A $ is the poset reflection of the preorder $ (A,\leq) $, where we define
	\[ a \leq b \iff b \in \sqrt{\langle a \rangle}, \quad \text{i.e.} \quad \exists n \in \mathbb{N},\, \exists c \in A,\, b^n = ac. \]
	Observe that $ a \leq b $ iff there exists a morphism $ A[a^{-1}] \to A[b^{-1}] $ under $ A $.
	An ideal $ S $ of $ \mathcal{V}_A $ is nothing other than a \textit{saturated} multiplicative set of $ A $,
	i.e.\ $ 1 \in S $ and $  \forall a,b \in A,\, ab \in S \iff a \in S $ and $ b \in S $.
	It belongs to $ \Spec(A) $ exactly when it is prime, i.e.\ $ 0 \notin S $ and $ a+b \in S $ implies either $ a \in S $ or $ b \in S $.
	Thus, a saturated prime multiplicative set is the complement of some prime (ring) ideal of $ A $,
	and $ \Spec(A) $ is homeomorphic to $ \Spec_Z(A) $.
	The basic open set $ D_a = \Set{S \in \Spec(A) \smid a \in S} $ is identified with the subset $ \Set{\mathfrak{p} \in \Spec_Z(A) \smid a \notin \mathfrak{p}} $.
	
	Saturated multiplicative sets have appeared in \cite{Tier1976a}.
	There, Tierney constructed Zariski spectra of ringed toposes by considering the internalization of the opposite of the above join-semilattice.
	This is compared with Joyal's method involving a distributive lattice, sometimes called a \textit{reticulation} of a ring, which we will describe below (for an arbitrary spatial Coste context).
	For the constructive perspectives of these constructions, see \cite[\S12.9]{BlechThesis} and the references listed in \cite[p.~222]{JohSS}.
	
	\noindent\textbf{DL context:}
	Let $ M $ be the distributive lattice corresponding to the formula $ \sup(u,u')=1 $, i.e.\ 
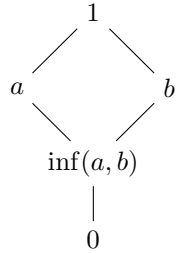
	\[ \tikz[auto]{
		\node (1) at (0,3) {$ 1 $};
		\node (a) at (-1,2) {$ a $};
		\node (b) at (1,2) {$ b $};
		\node (ab) at (0,1) {$ \inf(a,b) $};
		\node (0) at (0,0) {$ 0 $};
		\draw (a) to (1);
		\draw (b) to (1);
		\draw (ab) to (a);
		\draw (ab) to (b);
		\draw (0) to (ab);
	}. \]
	The class $ \mathcal{V} $ is generated by the obvious homomorphism $ M \to \{0 \lneq a \lneq b=1\} $.
	Let $ L $ be a distributive lattice. Then $ \mathcal{V}_L $ is (equivalent to) the poset of the quotients $ L \to L/\langle a \rangle $ for $ a \in L $,
	where $ \langle a \rangle $ is the principal filter generated by $ a $.
	Then there exists a morphism $ L/\langle a \rangle \to L/\langle b \rangle $ under $ L $ exactly when $ b \leq a $ in the order of $ L $.
	Thus, $ \mathcal{V}_L $ is (equivalent to) $ L^{\op} $, and an ideal of $ \mathcal{V}_L $ is a filter of $ L $.
	Under this identification, a filter of $ L $ belongs to $ \Spec(L) $ iff it is prime,
	and $ \Spec(L) $ is homeomorphic to the prime filter spectrum $ \Spec_D(L) $.
	The basic open set $ D_a = \Set{I \in \Spec(L) \smid a \in I} $ is identified with the subset $ D(a)=\Set{\mathfrak{p} \in \Spec_D(L) \smid a \in \mathfrak{p}} $.
	We adopted the convention for $ \Spec_D(L) $ so that it becomes a spectrum for the DL context,
	while the authors of \cite{DSTspectral} define the prime filter spectrum as the \textit{inverse space} of our $ \Spec_D(L) $.
	We can also regard the latter spectrum as a spectrum for the ``coDL context'' by replacing $ \sup $ with $ \inf $ and $ 1 $ with $ 0 $ in \emph{Example \ref{expl:DL-context}}.
\end{Expl}

Conversely, all spectra can be described as prime filter spectra of distributive lattices.
A topological space $ X $ is \emph{spectral} (a.k.a.\ ``coherent'' in the terminology of \cite[{\S}II.3]{JohSS})
if it is a compact (not necessarily Hausdorff) sober space, and its compact open subsets constitute a basis of $ X $ closed under finite intersections.
We will use the facts that the prime filter spectrum $ \Spec_D(L) $ of a distributive lattice $ L $ is spectral
and that any homomorphism of distributive lattices induces a spectral map 
(i.e.\ a continuous map whose inverse image map preserves compact open sets) between their spectra.
The category of spectral spaces and spectral maps is dually equivalent to the category of distributive lattices.
This duality itself is often referred to as Stone duality.
For an extensive exposition of spectral spaces, see \cite{DSTspectral}.
Coste \cite[4.4.3]{Coste1979} mentioned the following fact without proof:

\begin{Prop} \label{prop:spec(A)-is-spectral}
	$ \Spec(A) $ is a spectral space.
\end{Prop}
\begin{proof}
	We would like to find a distributive lattice $ L_A $ such that $ \Spec(A) $ is homeomorphic to $ \Spec_D(L_A) $.
	First, we describe $ \Idl(\mathcal{V}_A) $ as a spectral space.
	Let $ \mathbf{DLat} $ be the category of distributive lattices, and $ \mathrm{j}\mhyphen\mathbf{SLat} $ the category of join-semillatices.
	The forgetful functor from $ \mathbf{DLat} $ to $ \mathrm{j}\mhyphen\mathbf{SLat} $ has a left adjoint sending a join-semilattice $ S $
	to the distributive lattice 
	\[ \tilde{S} := \Set{R \subseteq S \smid \text{$ R $ is finite, and any two distinct elements are incomparable}} \]
	(cf.\ \cite[Lemma I.4.8]{JohSS}),
	where its element $ R = \{s_1,\dots,s_n \} $ is considered as ``a formal meet $ \bigwedge_i s_i $.''
	Then, join-preserving maps $ S \to 2 $ are in bijection with lattice homomorphisms $ \tilde{S} \to 2 $,
	and this correspondence yields a bijection between ideals of $ S $ and prime ideals of $ \tilde{S} $.
	When we put $ S=\mathcal{V}_A $, being careful about the definition of the topology on $ \Idl(\mathcal{V}_A) $,
	we can see that $ \Idl(\mathcal{V}_A) $ and $ \Spec_D(\widetilde{\mathcal{V}_A}^{\op}) $ are homeomorphic.
	
	If $ \Spec(A) $ is a spectral subspace of $ \Idl(\mathcal{V}_A) \cong \Spec_D(\widetilde{\mathcal{V}_A}^{\op}) $,
	it should be identified with $ \Spec_D(L_A) $ for some quotient lattice $ L_A $ of $ \widetilde{\mathcal{V}_A}^{\op} $.
	To obtain such an $ L_A $, we only have to take the quotient of $ \widetilde{\mathcal{V}_A}^{\op} $ by the relations
	\[ \lambda = \bigwedge_i \mu_i \qquad \text{(considered in $ \widetilde{\mathcal{V}_A} $)} \]
	for any $ \lambda \in \mathcal{V}_A $ and $ \mu_i $ as in the condition (iii) of \emph{Proposition \ref{prop:char-of-ideals-in-Spec}}.
\end{proof}

Next, we describe the structure sheaf $ S $ of $ \Spec(A) $.
As opposed to Coste's construction,
we define the structure sheaf such that it directly generalizes the usual construction of Zariski spectra (cf.\ \emph{Lemma \ref{lem:comparison-of-str-sheaves}}).

\begin{Def} \label{def:local-sections-of-str-sheaf}
	Let $ U $ be an open subset of $ \Spec(A) $.
	We define the sheaf $ S $ of $ T_0 $-models by
	\[ SU := \Set{s \in \prod_{I \in U} A_I \smid 
		\begin{aligned}
			& \forall J \in U,\, \exists \lambda \in J \;\text{(i.e.\ $ J \in D_{\lambda} $)},\, \exists a \in A_{\lambda},\,  \\
			& D_{\lambda} \subseteq U,\;\text{and}\;\forall I \in D_{\lambda},\,s_I=a_I.
		\end{aligned}
	}, \]
	where $ a_I $ is the image of $ a $ under $ A_{\lambda} \to A_I $.
\end{Def}
In particular, for $ \lambda \in \mathcal{V}_A $, we have a canonical homomorphism $ A_{\lambda} \to S(D_{\lambda}) $,
which is not necessarily an isomorphism (cf.\ \emph{Corollary \ref{cor:representation-as-local-sections}}).

\begin{Lem} \label{lem:stalk-of-structure-sheaf}
	The stalk $ S_I $ at an ideal $ I \in \Spec(A) $ is isomorphic to $ A_I $.
	In particular, $ (\Spec(A),S) $ is a $ T $-modelled space by \emph{Lemma \ref{lem:conditions-for-being-in-A-ModSp}}.
\end{Lem}
\begin{proof}
	There exists the canonical homomorphism $ S_I \to A_I $ sending the germ of $ s $ at $ I $ to $ s_I $.
	We show that this is a surjective embedding.
	
	Since $ A_I $ is the filtered colimit $ \rlim_{\lambda \in I} A_{\lambda} $,
	any element of $ A_I $ is of the form $ a_I $ for some $ \lambda \in I $ and $ a \in A_{\lambda} $.
	For such $ \lambda $ and $ I $, we have a commutative diagram
	\[ \tikz[auto]{
		\node (UL) at (0,1.5) {$ A_{\lambda} $};
		\node (UC) at (2,1.5) {$ S(D_{\lambda}) $};
		\node (UR) at (4,1.5) {$ S_I $};
		\node (DR) at (4,0) {$ A_I $};
		\draw[bend right=10pt,->] (UL) to (DR);
		\draw[->] (UL) to (UC);
		\draw[->] (UC) to (UR);
		\draw[->] (UC) to (DR);
		\draw[->] (UR) to (DR);
	}.  \]
	Thus, $ S_I \to A_I $ is surjective.
	
	To show that this is an embedding, suppose that $ A_I \models \rho(s_I, s'_I) $ for $ s,s' \in SU $, $ I \in U $, and a binary relation symbol $ \rho $.
	We would like to find a basic open $ D_{\mu} $ such that $ I \in D_{\mu} \subseteq U $ and $ S(D_{\mu}) \models \rho(s|_{D_{\mu}},s'|_{D_{\mu}}) $.
	By the definition of $ SU $, we can find $ \lambda, \lambda' \in I $ and $ a \in A_{\lambda}, a' \in A_{\lambda'} $ such that
	$ D_{\lambda} \cup D_{\lambda'} \subseteq U $, $ \forall J \in D_{\lambda},\, s_J = a_J $, and $ \forall J \in D_{\lambda'},\, s'_J = a'_J $.
	Replacing $ \lambda,\lambda' $ by $ \lambda \vee \lambda' $, we may assume that $ \lambda=\lambda' $.
	Since $ A_I \models \rho(a_I, a'_I) $, there exists $ \mu \in I $ such that $ \mu \geq \lambda $ and $ A_{\mu} \models \rho(\bar{a},\bar{a}') $,
	where $ \bar{a},\bar{a}' $ are respectively the images of $ a,a' $ under $ A_{\lambda} \to A_{\mu} $.
	Then, for any $ J \in D_{\mu} $, $ s_J = \bar{a}_J $, $ s'_J = \bar{a}'_J $, and $ A_J \models \rho(s_J,s'_J) $.
	Thus, we have the desired basic open set $ D_{\mu} $.
	The homomorphism $ S_I \to A_I $ also reflects the interpretations of the other relation symbols and the equality $ = $, and hence it is an embedding.
\end{proof}

We are now ready to show that this construction defines a functor 
\[ \Spec \colon T_0\mhyphen\mathbf{Mod}^{\op} \to \modsp{\mathbb{A}}{\mathbf{Sp}}. \]

\begin{Prop} \label{prop:hom-induces-morph-bw-Spec}
	Any homomorphism $ \alpha \colon A \to B $ induces a morphism $ \tilde{\alpha} \colon \Spec(B) \to \Spec(A) $ in $ \modsp{\mathbb{A}}{\mathbf{Sp}} $.
\end{Prop}
\begin{proof}
	The underlying map of $ \tilde{\alpha} $ is the restriction of the spectral map $ \Idl(\mathcal{V}_B) \to \Idl(\mathcal{V}_A) $
	which is induced by the join-preserving map $ \alpha_* \colon \mathcal{V}_A \to \mathcal{V}_B $ taking pushouts along $ \alpha $.
	Here, using \emph{Proposition \ref{prop:char-of-ideals-in-Spec}}, we can easily see that an ideal $ J \in \Spec(B) $ is sent to the ideal
	\[ \tilde{\alpha}(J) = \Set{\lambda \in \mathcal{V}_A \smid \alpha_*(\lambda) \in J}, \]
	which belongs to $ \Spec(A) $.
	Since $ \alpha_* $ yields a lattice homomorphism $ L_A \to L_B $, and $ \tilde{\alpha} $ can be identified with the induced map $ \Spec_D(L_B) \to \Spec_D(L_A) $,
	we can also see $ \tilde{\alpha} $ is a spectral map.
	
	Let $ S,R $ be the structure sheaves of $ \Spec(A),\Spec(B) $, respectively.
	The sheaf morphism $ \tilde{\alpha}^{\flat} \colon \tilde{\alpha}^*S \to R$ is induced by the homomorphisms between stalks
	\[ (\tilde{\alpha}^*S)_J \cong S_{\tilde{\alpha}(J)} \cong A_{\tilde{\alpha}(J)} \to B_J \cong R_J, \]
	which make the following diagram commutative for each $ \lambda \in \tilde{\alpha}(J) $:
	\[ \tikz[auto]{
		\node (UL) at (0,2) {$ A $};
		\node (UC) at (2,2) {$ A_{\lambda} $};
		\node (UR) at (4,2) {$ A_{\tilde{\alpha}(J)} $};
		\node (DL) at (0,0) {$ B $};
		\node (DC) at (2,0) {$ B_{\alpha_*(\lambda)} $};
		\node (DR) at (4,0) {$ B_J $};
		\draw[->] (UL) to node {$ \alpha $} (DL);
		\draw[->] (UL) to node {$ \lambda $} (UC);
		\draw[->] (DL) to node[swap] {$ \alpha_*(\lambda) $} (DC);
		\draw[->] (UC) to (DC);
		\draw[->] (UC) to (UR);
		\draw[->] (DC) to (DR);
		\draw[dashed,->] (UR) to (DR);
		\node (po) at (1.5,0.5) {$ \ulcorner $};
	}.  \]
	
	Moreover, this map $ A_{\tilde{\alpha}(J)} \to B_J $ is admissible.
	To show this, suppose that we have a commutative diagram
	\[ \tikz[auto]{
		\node (UL) at (0,2) {$ M $};
		\node (UR) at (2,2) {$ A_{\tilde{\alpha}(J)} $};
		\node (DL) at (0,0) {$ N $};
		\node (DR) at (2,0) {$ B_J $};
		\draw[->] (UL) to node {$ m $} (UR);
		\draw[->] (UL) to node {$ k $} (DL);
		\draw[->] (DL) to node {$ n $} (DR);
		\draw[->] (UR) to (DR);
	}  \]
	with $ k \in \mathcal{V} $.
	Then $ m $ factors through $ A_{\lambda} \to A_{\tilde{\alpha}(J)} $ for some $ \lambda \in \tilde{\alpha}(J) $.
	We have yet another $ \mu \in \mathcal{V}_A $ such that the above diagram is factored as
	\[ \tikz[auto]{
		\node (UL) at (0,2) {$ M $};
		\node (UC) at (2,2) {$ A_{\lambda} $};
		\node (UR) at (4,2) {$ A_{\tilde{\alpha}(J)} $};
		\node (DL) at (0,0) {$ N $};
		\node (DC) at (2,0) {$ A_{\mu} $};
		\node (DR) at (4,0) {$ B_J $};
		\draw[->] (UL) to node {$ k $} (DL);
		\draw[->] (UL) to (UC);
		\draw[->] (DL) to (DC);
		\draw[->] (UC) to (DC);
		\draw[->] (UC) to (UR);
		\draw[->] (DC) to (DR);
		\draw[->] (UR) to (DR);
		\node (po) at (1.5,0.5) {$ \ulcorner $};
	}.  \]
	We would like to show $ \mu \in \tilde{\alpha}(J) $ to obtain the desired filler $ N \to A_{\mu} \to A_{\tilde{\alpha}(J)} $.
	The square
	\[ \tikz[auto]{
		\node (UL) at (0,2) {$ M $};
		\node (UC) at (2,2) {$ A_{\lambda} $};
		\node (UR) at (4,2) {$ B_{\alpha_*(\lambda)} $};
		\node (DL) at (0,0) {$ N $};
		\node (DR) at (4,0) {$ B_J $};
		\draw[->] (UL) to node {$ k $} (DL);
		\draw[->] (UL) to (UC);
		\draw[->] (DL) to node {$ n $} (DR);
		\draw[->] (UC) to (UR);
		\draw[->] (UR) to (DR);
	}  \]
	commutes because the composites of
	$ A_{\lambda} \to B_{\alpha_*(\lambda)} \to B_{J} $ and $ A_{\lambda} \to A_{\tilde{\alpha}(J)} \to B_{J} $ coincide.
	By applying \emph{Lemma \ref{lem:useful-lemma}} to this square,
	we obtain the following commutative diagram for some $ \theta \in J $ with $ \alpha_*(\lambda) \leq \theta $:
	\[ \tikz[auto]{
		\node (UL) at (0,2) {$ A_{\lambda} $};
		\node (UR) at (2,2) {$ B_{\alpha_*(\lambda)} $};
		\node (DL) at (0,0) {$ A_{\mu} $};
		\node (DR) at (2,0) {$ B_{\alpha_*(\mu)} $};
		\node (Out) at (3.5,-1.5) {$ B_{\theta} $};
		\draw[->] (UL) to (DL);
		\draw[->] (UL) to (UR);
		\draw[->] (DL) to (DR);
		\draw[->] (UR) to (DR);
		\draw[bend left=15pt,->] (UR) to (Out);
		\draw[bend right=15pt,->] (DL) to (Out);
		\draw[dashed,->] (DR) to node {$ \exists! $} (Out);
		\node (po) at (1.5,0.5) {$ \ulcorner $};
	}.  \]
	This implies that $ \alpha_*(\mu) \leq \theta $ and, hence, $ \mu \in \tilde{\alpha}(J) $.
\end{proof}

The factorization theorem for the spatial case is much simpler:
if $ \alpha \colon A \to B $ is a homomorphism with $ B \models T $,
we have an ideal $ \tilde{\alpha}(\{\id_B\}) \in \Spec(A) $ and a factorization of $ \alpha $
\[ \tikz[auto]{
	\node (L) at (0,1.5) {$ A $};
	\node (C) at (2,0) {$ A_{\tilde{\alpha}(\{\id_B\})} $};
	\node (R) at (4,1.5) {$ B $};
	\draw[->] (L) to node {$ \alpha $} (R);
	\draw[->] (L) to (C);
	\draw[->] (C) to (R);
}. \]

\subsection{Comparison with Coste's Original Construction} \label{subsec:comparison-with-Coste-spectra}
The proofs of the results in the present and next subsections demand more knowledge of topos theory and first-order categorical logic.
We will not use those results after \S\ref{sec:colimits}, and the reader unfamiliar with these subjects can safely skip them.

We will continue to work with a fixed spatial Coste context $ (T_0,T,\Lambda) $ and a $ T_0 $-model $ A $.
Our construction of $ \Spec(A) $ differs from Coste's in two respects.
In this subsection, we will show the equivalence of these constructions.

The first difference is the construction of the underlying spaces (or the underlying toposes).
Coste \cite[4.4.1]{Coste1979} defines the underlying space as
\[ \Spec(A) := \Set{ \alpha \colon A \to B \smid \alpha \text{ is a localization in the sense after \emph{Lemma \ref{lem:adm-map-reflect-T-models}}.}}\!/{\cong}, \]
where $ \cong $ denotes isomorphisms of localizations in $ A \downarrow T_0\mhyphen\mathbf{Mod} $.
Its topology is generated by the basic open sets
\[ D_{\lambda} = \Set{\alpha \in \Spec(A) \smid \alpha \text{ factors through } \lambda} \]
for $ \lambda \in \mathcal{V}_A $.
This space is homeomorphic to our $ \Spec(A) $ since every localization is of the form $ A \to A_I $ and
two ideals $ I,I' $ giving the same localization $ A_I \cong A_{I'} $ turn out to be the same.

Coste also asserts that the sheaf topos $ \mathbf{Sh}(\Spec(A)) $ is categorically equivalent to
the topos $ \mathbf{Sh}(\mathcal{V}_A^{\op},C) $ of sheaves on the site $ (\mathcal{V}_A^{\op},C) $,
where $ C $ is the coverage on the meet-semilattice $ \mathcal{V}_A^{\op} $ we will describe below.
For each $ \lambda \in \mathcal{V}_A $, $ \varphi \vdash \bigvee_i \exists \bvec{v}_i \psi_i \in T $, and $ \bvec{a} \colon M_{\varphi}  \to A_{\lambda} $,
let us consider the family $ \{ \lambda \to \mu_i \}_i $ defined in the condition (iii) of \emph{Proposition \ref{prop:char-of-ideals-in-Spec}}.
Let $ C(\lambda) $ denote the set of $ \{\id_{\lambda}\} $ and such families:
\[ C(\lambda) := \left \{ \{\id_{\lambda}\} \right \} \cup \Set{\{ \lambda \to \mu_i \}_i \quad \smid \quad 
	\varphi \vdash \bigvee_i \exists \bvec{v}_i \psi_i \in T,\quad \bvec{a} \colon M_{\varphi}  \to A_{\lambda} }. \]
If $ \{ \lambda \to \mu_i \}_i \in C(\lambda) $ and $ \nu \geq \lambda $, then $ \{ \nu \to \mu_i \vee \nu \}_i \in C(\nu)$. 
Therefore, we have a coverage $ C $ on $ \mathcal{V}_A^{\op} $ in the sense of \cite[{\P}II.2.11]{JohSS}\footnote{
	A similar construction of a coverage on $ \mathcal{V}_A^{\op} $ works for a non-spatial Coste context,
	and Coste used such a site to give a spectrum of a model as a modelled topos \cite[\S4.1]{Coste1979}.}.
A sheaf on the site $ (\mathcal{V}_A^{\op},C) $ is a functor $ F \colon \mathcal{V}_A \to \mathbf{Set} $ such that, 
for each covering $ \{ \lambda \to \mu_i \}_i \in C(\lambda) $ and a family $ \{ a_i \in F(\mu_i) \}_i $ 
satisfying $ F(\mu_i \to \mu_i \vee \mu_j)(a_i) = F(\mu_j \to \mu_i \vee \mu_j)(a_j) $ as elements of $ F(\mu_i \vee \mu_j) $,
there exists a unique $ a \in F(\lambda) $ with $ a_i = F(\lambda \to \mu_i)(a) $ for all $ i $.

To obtain the desired equivalence between sheaf toposes,
we use the canonical cartesian functor (a meet-semilattice homomorphism)
\[ D_{(-)} \colon \mathcal{V}_A^{\op} \ni \lambda \mapsto D_{\lambda} \in \mathcal{O}(\Spec(A)). \]
By the characterization of ideals in $ \Spec(A) $ (\emph{Proposition \ref{prop:char-of-ideals-in-Spec}}),
this functor sends a covering $ \{ \lambda \to \mu_i \}_i \in C(\lambda) $ to an open covering $ D_{\lambda} = \bigcup_i D_{\mu_i} $.
Thus, it is a morphism of sites.
Once we show the following lemma\footnote{
	Since $ D_{(-)} $ is not necessarily full (cf.\ the discussion at the end of this subsection), we need to use a generalized comparison lemma.},
by an analogous proof of the comparison lemma \cite[Theorem C2.2.3]{Elephant}, we have an adjoint equivalence
\[ \tikz{
	\node (L) at (0,0) {$ \mathbf{Sh}(\Spec(A)) $};
	\node (R) at (4,0) {$ \mathbf{Sh}(\mathcal{V}_A^{\op},C) $};
	\draw[bend left=15pt,->] (L) to node[above] {$ (-) \circ D^{\op} $} node[midway,name=U] {} (R);
	\draw[bend left=15pt,->] (R) to node[below] {$ \RanExt_{D^{\op}} $} node[midway,name=D] {} (L);
	\path (D) to node[sloped] {$ \vdash $} (U);
},  \]
where $ \RanExt_{D^{\op}} $ is a global right Kan extension functor.

\begin{Lem}
	$ D_{(-)} \colon (\mathcal{V}_A^{\op},C) \to \mathcal{O}(\Spec(A)) $ is a \emph{dense morphism of sites} in the sense of \cite[Definition 5.1]{Car2020}.
	That is, it satisfies the following conditions:
	\begin{enumerate}[label=(\roman*)]
		\item Any upper subset $ \mathfrak{u} \subseteq {\uparrow}(\lambda) = \Set{\mu \smid \mu \geq \lambda}$ 
		with $ D_{\lambda} = \bigcup_{\mu \in \mathfrak{u}} D_{\mu} $ contains some covering $ \{ \lambda \to \mu_i \}_i \in C(\lambda) $,
		i.e.\ $ \mu_i \in \mathfrak{u} $.
		
		\item Every open $ U \in \mathcal{O}(\Spec(A)) $ is covered by basic open sets $ D_{\lambda} $.
		
		\item If $ D_{\lambda} \subseteq D_{\nu} $, then there exists a covering $ \{ \lambda \to \mu_i \}_i \in C(\lambda) $ such that $ \nu \leq \mu_i $ for all $ i $.
	\end{enumerate}
\end{Lem}
\begin{proof}
	\noindent (i) By the proof of \emph{Proposition \ref{prop:spec(A)-is-spectral}}, $ D_{\lambda} $ is compact,
	and $ L_A $ is isomorphic to the distributive lattice of finite unions $ D_{\nu_1} \cup \dots \cup D_{\nu_k} $ of basic open sets.
	Thus, if $ D_{\lambda} = \bigcup_{\mu \in \mathfrak{u}} D_{\mu} $,
	there exists a finite subset $ \{\nu_1,\dots,\nu_k\} \subseteq \mathfrak{u} $ with $ D_{\lambda} = D_{\nu_1} \cup \dots \cup D_{\nu_k} $.
	Since $ L_A $ is the quotient lattice of $ \widetilde{\mathcal{V}_A}^{\op} $ by the relations $ \lambda = \bigwedge_i \mu_i $
	for coverings $ \{ \lambda \to \mu_i \}_i \in C(\lambda) $,
	the equation $ D_{\lambda} = D_{\nu_1} \cup \dots \cup D_{\nu_k} $ must be witnessed by some covering $ \{ \lambda \to \mu_i \}_i \in C(\lambda) $
	such that $ \forall i ,\, \exists j,\, \mu_i \geq \nu_j $.
	We also have $ \mu_i \in \mathfrak{u} $, as desired, because $ \mathfrak{u} $ is an upper set.
	
	\noindent (ii) Obvious from the definition of the topology.
	
	\noindent (iii) By a similar argument as (i), the equation $ D_{\lambda} = D_{\lambda} \cap D_{\nu} = D_{\lambda \vee \nu} $
	must be witnessed by some covering $ \{ \lambda \to \mu_i \}_i \in C(\lambda) $ with $ \mu_i \geq \lambda \vee \nu $.
\end{proof}

\begin{Rmk}
	Instead of showing denseness of $ D_{(-)} $, we can also prove the above equivalence 
	by using the frame of $ C $-ideals on $ \mathcal{V}_A^{\op} $ \cite[{\P}II.2.11]{JohSS}.
	A subset $ \mathfrak{f} \subseteq \mathcal{V}_A $ is a $ C $-ideal if it is an upper set and satisfies the following condition
	\[ \forall \{ \lambda \to \mu_i \}_i \in C(\lambda),\, [\forall i,\, \mu_i \in \mathfrak{f} \implies \lambda \in \mathfrak{f}]. \]
	The poset $ C\mhyphen\!\Idl(\mathcal{V}_A^{\op}) $ of $ C $-ideals is a frame having the following universal property:
	for any frame $ L $, if a meet-semilattice homomorphism $ h \colon \mathcal{V}_A^{\op} \to L $ 
	sends a covering $ \{ \lambda \to \mu_i \}_i \in C(\lambda) $ to a covering $ h(\lambda) = \bigvee_i h(\mu_i) $,
	then it is uniquely extended to a frame homomorphism $ C\mhyphen\!\Idl(\mathcal{V}_A^{\op}) \to L $.
	In particular, $ D_{(-)} $ is extended to $ C\mhyphen\!\Idl(\mathcal{V}_A^{\op}) \to \mathcal{O}(\Spec(A)) $.
	This map sends a $ C $-ideal $ \mathfrak{f} $ to the open set $ \bigcup_{\lambda \in \mathfrak{f}} D_{\lambda} $.
	We can show that this is an isomorphism by observing that the map
	\[ \mathcal{O}(\Spec(A)) \ni U \quad \mapsto \quad \Set{\lambda \smid D_{\lambda} \subseteq U } \in C\mhyphen\!\Idl(\mathcal{V}_A^{\op}) \]
	is its inverse.
	By \cite[Lemma V.1.7]{JohSS}, $ \mathbf{Sh}(C\mhyphen\!\Idl(\mathcal{V}_A^{\op})) $ and $ \mathbf{Sh}(\mathcal{V}_A^{\op},C) $ are equivalent,
	and hence so are $ \mathbf{Sh}(\Spec(A)) $ and $ \mathbf{Sh}(\mathcal{V}_A^{\op},C) $.
	In any case, our proofs exploit the construction of the reticulation $ L_A $ and the identification between $ L_A $ and the lattice of compact open sets of $ \Spec(A) $.
	We could not come up with direct proofs of compactness of $ D_{\lambda} $ and the conditions (i), (iii).
\end{Rmk}

The second difference between our and Coste's constructions of spectra lies in the definitions of their structure sheaves.
Coste defines the structure sheaf on the site $ (\mathcal{V}_A^{\op},C) $ as the associated sheaf $ \tilde{P} $ of the presheaf
\[ P \colon \mathcal{V}_A \ni \lambda \quad \mapsto \quad A_{\lambda} \in T_0\mhyphen\mathbf{Mod}. \]

\begin{Lem} \label{lem:comparison-of-str-sheaves}
	Under the equivalence $ \mathbf{Sh}(\mathcal{V}_A^{\op},C) \simeq \mathbf{Sh}(\Spec(A)) $, $ \tilde{P} $ corresponds to our structure sheaf $ S $ on $ \Spec(A) $.
\end{Lem}
\begin{proof}
	The morphism $ D_{(-)} $ of sites induces an adjunction
	\[ \tikz{
		\node (L) at (0,0) {$ \mathbf{Sh}(\Spec(A)) $};
		\node (R) at (4,0) {$ \mathbf{Sh}(\mathcal{V}_A^{\op},C) $};
		\draw[bend left=15pt,->] (L) to node[above] {$ (-) \circ D^{\op} $} node[midway,name=U] {} (R);
		\draw[bend left=15pt,->] (R) to node[below] {$ \mathbf{a} \circ \LanExt_{D^{\op}} $} node[midway,name=D] {} (L);
		\path (D) to node[sloped] {$ \dashv $} (U);
	},  \]
	where $ \LanExt_{D^{\op}} $ is a global left Kan extension functor,
	and $ \mathbf{a} \colon \mathbf{Set}^{\mathcal{O}(\Spec(A))^{\op}} \to \mathbf{Sh}(\Spec(A)) $ is the associated sheaf functor.
	We have already seen that $ (-) \circ D^{\op} $ is part of an equivalence,
	and thus $ \RanExt_{D^{\op}} \cong \mathbf{a} \circ \LanExt_{D^{\op}} $.
	Since the left square below is commutative, the right square commutes up to natural isomorphism.
	\[ \tikz[auto]{
		\node (UL) at (0,2) {$ \mathbf{Set}^{\mathcal{V}_A} $};
		\node (DL) at (0,0) {$ \mathbf{Sh}(\mathcal{V}_A^{\op},C) $};
		\node (UR) at (4,2) {$ \mathbf{Set}^{\mathcal{O}(\Spec(A))^{\op}} $};
		\node (DR) at (4,0) {$ \mathbf{Sh}(\Spec(A)) $};
		
		\draw[<-] (UL) to (DL);
		\draw[<-] (UR) to (DR);
		\draw[<-] (DL) to node {$ (-) \circ D^{\op} $}(DR);
		\draw[<-] (UL) to node {$ (-) \circ D^{\op} $} (UR);
	} \quad
	\tikz[auto]{
		\node (UL) at (0,2) {$ \mathbf{Set}^{\mathcal{V}_A} $};
		\node (DL) at (0,0) {$ \mathbf{Sh}(\mathcal{V}_A^{\op},C) $};
		\node (UR) at (4,2) {$ \mathbf{Set}^{\mathcal{O}(\Spec(A))^{\op}} $};
		\node (DR) at (4,0) {$ \mathbf{Sh}(\Spec(A)) $};
		
		\draw[->] (UL) to node {$ \tilde{(-)} $} (DL);
		\draw[->] (UR) to node {$ \mathbf{a} $} (DR);
		\draw[->] (DL) to node {$ \mathbf{a} \circ \LanExt_{D^{\op}} $}(DR);
		\draw[->] (UL) to node {$ \LanExt_{D^{\op}} $} (UR);
	}. \]
	Hence, we have isomorphisms of sheaves
	\[ \RanExt_{D^{\op}}(\tilde{P}) \cong \mathbf{a}(\LanExt_{D^{\op}}(\tilde{P})) \cong \mathbf{a}(\LanExt_{D^{\op}}(P)). \]
	Since the sheafification process does not change stalks, it suffices to show that the stalk of $ \LanExt_{D^{\op}}(P) $ at $ I \in \Spec(A) $ 
	is isomorphic to $ A_I $ (the stalk of $ S $ at $ I $).
	We have a canonical homomorphism $ A_I \to (\LanExt_{D^{\op}}(P))_I $ making the following square commutative for each $ \lambda \in I $:
	\[ \tikz[auto]{
		\node (UL) at (0,1.5) {$ A_I $};
		\node (DL) at (0,0) {$ A_{\lambda} $};
		\node (UR) at (3,1.5) {$ (\LanExt_{D^{\op}}(P))_I $};
		\node (DR) at (3,0) {$ (\LanExt_{D^{\op}}(P))(D_{\lambda}) $};
		\node at (7,0) {$ = \rlim\left ((D_{\lambda} \downarrow D)^{\op} \to \mathcal{V}_A \xrightarrow{P} \mathbf{Set} \right ) $};
		\draw[<-] (UL) to (DL);
		\draw[<-] (UR) to (DR);
		\draw[->] (DL) to node {coproj.} (DR);
		\draw[->] (UL) to (UR);
	} \]
	From this description, it is easy to check that $  A_I \cong (\LanExt_{D^{\op}}(P))_I $.
\end{proof}

We have finished the proof of the following theorem:

\begin{Thm*} \label{thm:equiv-of-two-modelled-toposes}
	$ (\mathbf{Sh}(\mathcal{V}_A^{\op},C), \tilde{P}) $ and $ (\mathbf{Sh}(\Spec(A)), S) $ are equivalent as $ T_0 $-modelled toposes.
\end{Thm*}

Before leaving this subsection, we make a few comments on an interesting property of $ D_{(-)} $ which does not hold in general.
We say that \emph{the prime ideal theorem (PIT) holds for a $ T_0 $-model $ A $}
when the functor $ D_{(-)} \colon \mathcal{V}_A^{\op} \to \mathcal{O}(\Spec(A)) $ is full.
In other words, if $ \lambda \nleq \mu $ in $ \mathcal{V}_A $, then there exists $ I \in \Spec(A) $ such that $ \mu \in I $ and $ \lambda \notin I $.
In the case when $ T $ contains an additional axiom $ \varphi \vdash \exists y \psi $,
we may have a singleton covering $ \{ \lambda \to \mu \} $, but we cannot distinguish $ \lambda $ and $ \mu $ by ideals in $ \Spec(A) $.
Thus, $ D_{(-)} $ need not be full.

From the description in \emph{Example \ref{expl:V_A-and-ideals-in-examples}}, concerning the DL context, 
PIT for a distributive lattice $ L $ is the well-known prime filter theorem.
Concerning the Zariski context, PIT holds for any ring $ A $:
\[ b \notin \sqrt{\langle a \rangle} \implies \exists \mathfrak{p} \in \Spec_Z(A),\, a \in \mathfrak{p} \;\text{and}\; b \notin \mathfrak{p}. \]
This amounts to the fact that $ \sqrt{\langle a \rangle} = \bigcap \Set{\mathfrak{p} \in \Spec_Z(A) \smid a \in \mathfrak{p}} $.

\subsection{Standard Coste Contexts} \label{subsec:standard-context}
It is natural to ask when a model can be recovered as the model of global sections of the structure sheaf.
Coste gave some necessary and sufficient conditions for a Coste context so that every model can be recovered.
To state them, we need to recall some constructions from categorical logic.
For any coherent theory $ T $, we can construct \emph{the coherent syntactic category $ \mathcal{C}^{\mathrm{coh}}_{T} $ of $ T $}
as we did in \emph{Proposition \ref{prop:cart-cat-vs-cart-th}}(1) for the cartesian case.
Its objects are coherent formulas, and its morphisms are equivalence classes of $ T $-functional coherent formulas.
All the details are described in \cite[Chapter D1]{Elephant}.
In the rest of this section, we will write $ \mathcal{C}^{\mathrm{cart}}_{T_0} $ for the (cartesian) syntactic category of $ T_0 $
to distinguish cartesian and coherent syntactic categories.

For a Coste context $ (T_0,T,\Lambda) $, 
we have the coverage $ J^{T}_{T_0} $ on $ \mathcal{C}^{\mathrm{cart}}_{T_0} $ generated by the families
\[ \Set{ \{\bvec{uv}_i.\,\psi_i \} \to \{\bvec{u}.\,\varphi\}}_i \]
for an additional axiom $ \varphi \vdash \bigvee_i \exists \bvec{v}_i \psi_i \in T $.
We also have the coherent coverage $ J_T $ on $ \mathcal{C}^{\mathrm{coh}}_{T} $, which is always subcanonical.
These sites $ (\mathcal{C}^{\mathrm{cart}}_{T_0},J^{T}_{T_0}) $ and $ (\mathcal{C}^{\mathrm{coh}}_{T},J_T) $
have an equivalent topos of sheaves: it is the classifying topos of $ T $ (\cite[Proposition D3.1.10]{Elephant}).

\begin{Thm}[{\cite[Theorem 4.5.1 and Proposition 4.5.2]{Coste1979}}]
	For a spatial Coste context $ (T_0,T,\Lambda) $, the following are equivalent:
	\begin{enumerate}[label=(\roman*)]
		\item The canonical functor $ F \colon \mathcal{C}^{\mathrm{cart}}_{T_0} \to \mathcal{C}^{\mathrm{coh}}_{T} $ is fully faithful.
		\item The site $ (\mathcal{C}^{\mathrm{cart}}_{T_0},J^{T}_{T_0}) $ is subcanonical.
		\item For each $ T_0 $-model $ A $, the functor $ P_A \colon \mathcal{V}_A \ni \lambda \mapsto A_{\lambda} \in \mathbf{Set} $ 
		is a sheaf for the site $ (\mathcal{V}_A^{\op},C) $. 
		That is, for each $ \lambda \in \mathcal{V}_A$, a covering $ \{ \lambda \to \mu_i \}_i \in C(\lambda) $, and a family $ \{ a_i\}_i  \in \prod_i A_{\mu_i} $
		satisfying $ (A_{\mu_i} \to A_{\mu_i \vee \mu_j})(a_i) = (A_{\mu_j} \to A_{\mu_i \vee \mu_j})(a_j) $ as elements of $ A_{\mu_i \vee \mu_j} $,
		there exists a unique $ a \in A_{\lambda} $ with $ a_i = (A_{\lambda} \to A_{\mu_i})(a) $ for all $ i $.
		
		\item For each $ T_0 $-model $ A $, the unit component $ \eta_A \colon A \to \Gamma(\Spec(A),S) $ of the adjunction $ \Gamma \dashv \Spec $
		is an isomorphism. Equivalently, $ \Spec \colon T_0\mhyphen\mathbf{Mod}^{\op} \to \modsp{\mathbb{A}}{\mathbf{Sp}} $ is fully faithful.
	\end{enumerate}
\end{Thm}
\begin{proof}
	(i)$ \Leftrightarrow $(ii):
	Let us write $ \mathbf{a},\mathbf{y} $ for the associated sheaf functor and the Yoneda embedding, respectively,
	w.r.t.\ the site $ (\mathcal{C}^{\mathrm{cart}}_{T_0},J^{T}_{T_0}) $.
	Similarly, $ \mathbf{a'},\mathbf{y'} $ w.r.t.\ $ (\mathcal{C}^{\mathrm{coh}}_{T},J_T) $.
	By an argument similar to the proof of \emph{Lemma \ref{lem:comparison-of-str-sheaves}},
	the square below commutes up to natural isomorphism
	\[ \tikz[auto]{
		\node (UL) at (0,2) {$ \mathcal{C}^{\mathrm{cart}}_{T_0} $};
		\node (DL) at (0,0) {$ \mathbf{Sh}(\mathcal{C}^{\mathrm{cart}}_{T_0},J^{T}_{T_0}) $};
		\node (UR) at (4,2) {$ \mathcal{C}^{\mathrm{coh}}_{T} $};
		\node (DR) at (4,0) {$ \mathbf{Sh}(\mathcal{C}^{\mathrm{coh}}_{T},J_T) $};
		
		\draw[->] (UL) to node {$ \mathbf{ay} $} (DL);
		\draw[->] (UR) to node {$ \mathbf{y'} $} (DR);
		\draw[->] (DL) to node {$ \mathbf{a'} \circ \LanExt_{F^{\op}} $}(DR);
		\draw[->] (UL) to node {$ F $} (UR);
	}. \]
	Here, $ \mathbf{a'} \circ \LanExt_{F^{\op}} $ is part of an equivalence.
	If $ (\mathcal{C}^{\mathrm{cart}}_{T_0},J^{T}_{T_0}) $ is subcanonical, then $ \mathbf{ay} $ is fully faithful, and so is $ F $.
	Conversely, suppose that $ F $ is fully faithful.
	Then we can easily see that the composite of the functors
	\[ \mathcal{C}^{\mathrm{cart}}_{T_0} \xrightarrow{F} \mathcal{C}^{\mathrm{coh}}_{T} \xrightarrow{\mathbf{y'}}
	\mathbf{Sh}(\mathcal{C}^{\mathrm{coh}}_{T},J_T) \xrightarrow{(-) \circ F^{\op}} \mathbf{Sh}(\mathcal{C}^{\mathrm{cart}}_{T_0},J^{T}_{T_0}) 
	\hookrightarrow \mathbf{Set}^{(\mathcal{C}^{\mathrm{cart}}_{T_0})^{\op}} \]
	is naturally isomorphic to the Yoneda embedding $ \mathbf{y} $,
	and thus $(\mathcal{C}^{\mathrm{cart}}_{T_0},J^{T}_{T_0}) $ is subcanonical.
	
	\noindent (ii)$ \Rightarrow $(iii):
	Suppose that the representable presheaf $ \Hom(-,\theta) \colon (\mathcal{C}^{\mathrm{cart}}_{T_0})^{\op} \to \mathbf{Set} $ 
	is a $ J^{T}_{T_0} $-sheaf for each $ T_0 $-cartesian formula $ \theta $.
	This implies that, for each $ \varphi \vdash \bigvee_i \exists \bvec{v}_i \psi_i \in T $, $ \{\bvec{u}.\,\varphi\} $ is a colimit of the finite diagram
	\[ \Set{ \{\bvec{uv}_i.\,\psi_i \} \leftarrow \{\bvec{u}\bvec{v}_i\bvec{v}_j.\,\psi_i \land \psi_j \} \to \{\bvec{uv}_j.\,\psi_j \} }_{i,j}. \]
	For simplicity, we will consider a covering $ \{\mu \leftarrow \lambda \to \mu' \} \in C(\lambda) $
	induced by an axiom $ \varphi(\bvec{u}) \vdash \exists \bvec{v}\psi(\bvec{u},\bvec{v}) \lor \exists\bvec{v'}\psi'(\bvec{u},\bvec{v'}) $ in $ T $,
	where the tuples $ \bvec{u},\bvec{v},\bvec{v'} $ are assumed to be mutually disjoint.
	Then the following square in $ \mathcal{C}^{\mathrm{cart}}_{T_0} $ is a pullback
	\[ \tikz[auto]{
		\node (UL) at (0,1.5) {$ \{\bvec{uvv'}.\,\psi \land \psi' \} $};
		\node (UR) at (3,1.5) {$ \{\bvec{uv'}.\,\psi'\} $};
		\node (DL) at (0,0) {$ \{\bvec{uv}.\,\psi\} $};
		\node (DR) at (3,0) {$ \{\bvec{u}.\,\varphi\} $};
		\draw[->] (UL) to (UR);
		\draw[->] (UL) to (DL);
		\draw[->] (DL) to (DR);
		\draw[->] (UR) to (DR);
	}.  \]
	By the assumption, it is also a pushout.
	Put $ \theta \equiv \psi \land \psi' $.
	Consider the following cube whose faces are pushouts:
	\[ \tikzset{cross/.style={preaction={-,draw=white,line width=6pt}}}
	\tikz[auto]{
		\node (UL1) at (0,2,0) {$ M_{\varphi} $};
		\node (DL1) at (3,2,3) {$ M_{\psi} $};
		\node (UR1) at (3,2,0) {$ M_{\psi'}$};
		\node (DR1) at (6,2,3) {$ M_{\theta}$};
		
		\node (UL2) at (0,0,0) {$ A_{\lambda}$};
		\node (DL2) at (3,0,3) {$ A_{\mu} $};
		\node (UR2) at (3,0,0) {$ A_{\mu'}$};
		\node (DR2) at (6,0,3) {$ A_{\mu\vee\mu'}$};
		
		\draw[->] (UL1) to (UR1);
		\draw[->] (UR1) to node[right,pos=0.7] {} (UR2);
		\draw[->] (UL1) to node[left] {} (UL2);
		\draw[->] (UL2) to (UR2);
		\draw[->,cross] (DR1) to (DR2);
		
		\draw[->] (UL1) to (DL1);
		\draw[->] (UL2) to (DL2);
		\draw[->,cross] (DL1) to (DL2);
		\draw[->,cross] (DL1) to (DR1);
		\draw[->] (DL2) to (DR2);
		\draw[->] (UR1) to (DR1);
		\draw[->] (UR2) to (DR2);
	}. \]
	The top face is also a pullback in $ T_0\mhyphen\mathbf{Mod}_{\fp} \simeq (\mathcal{C}^{\mathrm{cart}}_{T_0})^{\op} $.
	We will show the bottom face is a pullback in $ T_0\mhyphen\mathbf{Mod} $, from which the sheaf condition for $ P_A $ follows.
	
	Suppose that 
	\[ \tikz[auto]{
		\node (UL) at (0,1.5) {$ M $};
		\node (UR) at (1.5,1.5) {$ A_{\mu'} $};
		\node (DL) at (0,0) {$ A_{\mu} $};
		\node (DR) at (1.5,0) {$ A_{\mu \vee \mu'} $};
		\draw[->] (UL) to node {$ l $} (UR);
		\draw[->] (UL) to node {$ k $} (DL);
		\draw[->] (DL) to node {} (DR);
		\draw[->] (UR) to (DR);
	} \]
	is commutative.
	To obtain the desired homomorphism $ M \to A_{\lambda} $, we may assume that $ M $ is finitely presentable, 
	since an arbitrary $ T_0 $-model is a filtered colimit of finitely presentable ones.
	Like the proof of the lifting lemma, we can represent the bottom face of the cube as a filtered colimit of pushout squares 
	``between the top and bottom faces'' with finitely presentable vertices.
	Then, after replacing the top face by its pushout along some homomorphism, we can lift $ k,l $ so that the following diagram commutes
	\[ \tikz[auto]{
		\node (UL) at (0,1.5) {$ M $};
		\node (UR) at (1.5,1.5) {$ M_{\psi'} $};
		\node (DL) at (0,0) {$ M_{\psi} $};
		\node (DR) at (1.5,0) {$ M_{\theta} $};
		\draw[->] (UL) to node {$ l $} (UR);
		\draw[->] (UL) to node {$ k $} (DL);
		\draw[->] (DL) to node {} (DR);
		\draw[->] (UR) to (DR);
	}. \]
	Since $ J^{T}_{T_0} $-covers are stable under pullbacks, the new top face is a pullback in $ T_0\mhyphen\mathbf{Mod}_{\fp} $,
	and we obtain a homomorphism $ M \to M_{\varphi} $ as desired.
	
	\noindent (iii)$ \Rightarrow $(iv):
	If we assume (iii), then we can identify $ P_A $ with the structure sheaf of $ \Spec(A) $ by \emph{Theorem \ref{thm:equiv-of-two-modelled-toposes}}.
	The global section model is given by $ P_A(\id_A) = A $, and thus we have a sheaf representation of $ A $.

	\noindent (iv)$ \Rightarrow $(i):
	To show that $ F $ is faithful, 
	it suffices to show $ T_0 \models \varphi \vdash \psi $ when $ \varphi \vdash \psi $ is a $ T_0 $-cartesian sequent such that $ T \models \varphi \vdash \psi $.
	For each $ T_0 $-model $ A $, the structure sheaf $ S $ of $ \Spec(A) $ is a $ T $-model in $ \mathbf{Sh}(\Spec(A)) $, and hence $ S \models \varphi \vdash \psi $.
	Since $ A \cong \Gamma(\Spec(A),S) $ and the global section functor $ \Gamma \colon \mathbf{Sh}(\Spec(A)) \to \mathbf{Set} $ preserves validity of cartesian sequents,
	we have $ A \models \varphi \vdash \psi $ as desired.
	
	We now prove fullness of $ F $.
	Suppose that $ \varphi,\psi $ are $ T_0 $-cartesian formulas and 
	$ [\theta] \colon \{\bvec{u}.\,\varphi\} \to \{\bvec{v}.\,\psi\} $ is a morphism in $ \mathcal{C}^{\mathrm{coh}}_{T} $.
	We would like to find a $ T_0 $-cartesian functional formula $ \chi(\bvec{u},\bvec{v}) $ defining a morphism $ \{\bvec{u}.\,\varphi\} \to \{\bvec{v}.\,\psi\} $ 
	in $ \mathcal{C}^{\mathrm{cart}}_{T_0} $ such that 
	\[ T \models \chi \vdash \theta. \]
	If we can find such $ \chi $, since $ \chi,\theta $ are $ T $-functional formulas with the same domain and codomain, we also have $ T \models \theta \vdash \chi $.
	
	For a fixed $ T_0 $-model $ A $, again because $ S $ is a $ T $-model in $ \mathbf{Sh}(\Spec(A)) $,
	we have definable subsheaves $ \dbracket{\theta}_S \subseteq \dbracket{\varphi \land \psi}_S \cong \dbracket{\varphi}_S \times \dbracket{\psi}_S $ of a finite power of $ S $.
	Moreover, $ \dbracket{\theta}_S $ is a retract of $ \dbracket{\varphi \land \psi}_S $ since the projection $ \dbracket{\theta}_S \to \dbracket{\varphi}_S $ is an isomorphism.
	We will take $ \chi $ as ``a formula defining the functor $ A \mapsto \Gamma_A(\dbracket{\theta}_S) $,''
	where $ \Gamma_A \colon \mathbf{Sh}(\Spec(A)) \to \mathbf{Set} $ is the global section functor.
	
	Let $ \alpha \colon A \to B$ be a homomorphism of $ T_0 $-models, 
	and $ \tilde{\alpha} \colon (\Spec(B),R) \to (\Spec(A),S) $ be the induced morphism in $ \modsp{\mathbb{A}}{\mathbf{Sp}} $ by \emph{Proposition \ref{prop:hom-induces-morph-bw-Spec}}.
	Since $ \tilde{\alpha}^{\flat} \colon \tilde{\alpha}^*S \to R$ is a homomorphism of $ T $-models, we have the left diagram below.
	The inverse image functor $ \tilde{\alpha}^* \colon \mathbf{Sh}(\Spec(A)) \to \mathbf{Sh}(\Spec(B)) $ preserves interpretations of coherent formulas,
	and hence transposition w.r.t.\ the adjunction $ \tilde{\alpha}^* \dashv \tilde{\alpha}_* $ yields the right one:
	\[ \tikzset{cross/.style={preaction={-,draw=white,line width=6pt}}}
	\tikz[auto]{
		\node (DL1) at (2,2,3) {$ \dbracket{\theta}_{\tilde{\alpha}^*S} $};
		\node (UR1) at (3,2,0) {$ \dbracket{\varphi \land \psi}_{\tilde{\alpha}^*S} $};
		\node (DR1) at (6,2,3) {$ \dbracket{\theta}_{\tilde{\alpha}^*S} $};
		
		\node (DL2) at (2,0,3) {$ \dbracket{\theta}_{R} $};
		\node (UR2) at (3,0,0) {$ \dbracket{\varphi \land \psi}_{R} $};
		\node (DR2) at (6,0,3) {$ \dbracket{\theta}_{R} $};
		
		\draw[->] (DL1) to (UR1);
		\draw[->] (DL2) to (UR2);
		\draw[->] (UR1) to node[right,pos=0.7] {} (UR2);
		\draw[->,cross] (DR1) to (DR2);
		
		\draw[->,cross] (DL1) to (DL2);
		\draw[=,double distance=3pt,cross] (DL1) to (DR1);
		\draw[=,double distance=3pt,] (DL2) to (DR2);
		\draw[->] (UR1) to (DR1);
		\draw[->] (UR2) to (DR2);
	}
	\qquad
	\tikz[auto]{
		\node (DL1) at (2,2,3) {$ \dbracket{\theta}_{S} $};
		\node (UR1) at (3,2,0) {$ \dbracket{\varphi \land \psi}_{S} $};
		\node (DR1) at (6,2,3) {$ \dbracket{\theta}_{S} $};
		
		\node (DL2) at (2,0,3) {$ \tilde{\alpha}_*(\dbracket{\theta}_{R}) $};
		\node (UR2) at (3,0,0) {$ \tilde{\alpha}_*(\dbracket{\varphi \land \psi}_{R}) $};
		\node (DR2) at (6,0,3) {$ \tilde{\alpha}_*(\dbracket{\theta}_{R}) $};
		
		\draw[->] (DL1) to (UR1);
		\draw[->] (DL2) to (UR2);
		\draw[->] (UR1) to node[right,pos=0.7] {} (UR2);
		\draw[->,cross] (DR1) to (DR2);
		
		\draw[->,cross] (DL1) to (DL2);
		\draw[=,double distance=3pt,cross] (DL1) to (DR1);
		\draw[=,double distance=3pt,] (DL2) to (DR2);
		\draw[->] (UR1) to (DR1);
		\draw[->] (UR2) to (DR2);
	}. \]
	Taking global sections, we obtain the following diagram:
	\[\tikzset{cross/.style={preaction={-,draw=white,line width=6pt}}}
	\tikz[auto]{
		\node (DL1) at (2,2,3) {$ \Gamma_A(\dbracket{\theta}_{S}) $};
		\node (UR1) at (5,2,0) {$ \Gamma_A(\dbracket{\varphi \land \psi}_{S}) \cong \dbracket{\varphi \land \psi}_{A} $};
		\node (DR1) at (10,2,3) {$ \Gamma_A(\dbracket{\theta}_{S}) $};
		
		\node (DL2) at (2,0,3) {$ \Gamma_B(\dbracket{\theta}_{R}) $};
		\node (UR2) at (5,0,0) {$ \Gamma_B(\dbracket{\varphi \land \psi}_{R}) \cong \dbracket{\varphi \land \psi}_{B} $};
		\node (DR2) at (10,0,3) {$ \Gamma_B(\dbracket{\theta}_{R}) $};
		
		\draw[->] (node cs:name=DL1,anchor=north east) to (node cs:name=UR1,anchor=south west);
		\draw[->] (node cs:name=DL2,anchor=north east) to (node cs:name=UR2,anchor=south west);
		\draw[->] (UR1) to (UR2);
		\draw[->,cross] (DR1) to (DR2);
		
		\draw[->,cross] (DL1) to (DL2);
		\draw[=,double distance=3pt,cross] (DL1) to (DR1);
		\draw[=,double distance=3pt,] (DL2) to (DR2);
		\draw[->] (node cs:name=UR1,anchor=south east) to (node cs:name=DR1,anchor=north west);
		\draw[->] (node cs:name=UR2,anchor=south east) to (node cs:name=DR2,anchor=north west);
	}. \]
	Therefore, the functor
	\[ T_0\mhyphen\mathbf{Mod} \ni A \quad \mapsto \quad \Gamma_A(\dbracket{\theta}_S) \in \mathbf{Set} \]
	is a retract of the functor $ A \mapsto \dbracket{\varphi \land \psi}_A \cong \Hom(M_{\varphi \land \psi},A) $.
	Since any retract of a representable functor in $ [T_0\mhyphen\mathbf{Mod}_{\fp},\mathbf{Set}] $ (or in $ [T_0\mhyphen\mathbf{Mod},\mathbf{Set}] $) is again representable,
	we can find a $ T_0 $-cartesian formula $ \chi(\bvec{u},\bvec{v}) $ such that it is a retract of $ \{\bvec{uv}.\,\varphi \land \psi \} $ in $ \mathcal{C}^{\mathrm{cart}}_{T_0} $
	and, for each $ T_0 $-model $ A $, $ \Gamma_A(\dbracket{\theta}_S) = \dbracket{\chi}_A $ as subobjects of $ \dbracket{\varphi \land \psi}_A $.
	
	The obtained formula $ \chi $ defines a morphism $ \{\bvec{u}.\,\varphi\} \to \{\bvec{v}.\,\psi\} $ in $ \mathcal{C}^{\mathrm{cart}}_{T_0} $.
	To show that $ T \models \chi \vdash \theta $, let us suppose that $ A \models \chi(\bvec{a},\bvec{b}) $ for a $ T $-model $ A $.
	If we regard $ \bvec{a},\bvec{b} $ as global sections of $ S $,
	the tuple $ (\bvec{a},\bvec{b}) $ is a global section of the subsheaf $ \dbracket{\theta}_S $.
	Since $ \theta $ is a coherent formula, the germ $ (\bvec{a}_I,\bvec{b}_I) \in A_I $ satisfies $ \theta $ for each $ I \in \Spec(A) $.
	But $ A \models T $ implies $ \{\id_A\} \in \Spec(A) $, and hence $ A \models \theta(\bvec{a},\bvec{b}) $.
	This completes the proof of fullness of $ F $.
\end{proof}

In fact, after suitable modifications, the above proof can be generalized to arbitrary Coste contexts.
We will say a (spatial) Coste context is \emph{standard} if it satisfies any one of the above conditions.
Notice that the conditions (i) and (ii) do not depend on the choice of $ \Lambda $,
and thus any extension $ (T_0,T,\Lambda') $ of a standard context $ (T_0,T,\Lambda) $, i.e.\ $ \Lambda' \supseteq \Lambda $, is standard.

\begin{Cor} \label{cor:representation-as-local-sections}
	Let $ (T_0,T,\Lambda) $ be a standard spatial Coste context, and $ A $ be a $ T_0 $-model.
	Then, for each $ \lambda \in \mathcal{V}_A $, there exists an isomorphism $ A_{\lambda} \cong S(D_{\lambda}) $.
\end{Cor}
\begin{proof}
	In the proof of the implication (ii)$ \Rightarrow $(iii),
	we have seen that, for each covering $ \{\lambda \to \mu_i \}_i \in C(\lambda) $,
	$ A_{\lambda} $ is a limit in $ T_0\mhyphen\mathbf{Mod} $ of the finite diagram $ \{ A_{\mu_i} \to A_{\mu_i \vee \mu_j} \leftarrow A_{\mu_j} \}_{i,j} $.
	Thus, $ \lambda $ is a limit in $ \mathcal{V}_A $ of the finite diagram $ \{\mu_i \to \mu_i \vee \mu_j \leftarrow \mu_j \} $,
	and every representable functor $ \Hom(\lambda,-) \colon \mathcal{V}_A \to \mathbf{Set} $ is a sheaf for the site $ (\mathcal{V}_A^{\op},C) $.
	Under the equivalence $ \mathbf{Sh}(\mathcal{V}_A^{\op},C) \simeq \mathbf{Sh}(\Spec(A)) $,
	$ \Hom(\lambda,-) $ corresponds to the subterminal $ D_{\lambda} $. Therefore, 
	\[ A_{\lambda} \cong \Hom_{\mathbf{Sh}(\mathcal{V}_A^{\op},C)}(\Hom(\lambda,-),P_A) \cong \Hom_{\mathbf{Sh}(\Spec(A))}(D_{\lambda},S) \cong S(D_{\lambda}). \qedhere \]
\end{proof}

Finally, we give a characterization of faithfulness of $ F $.

\begin{Prop}[{\cite[Proposition 4.5.3]{Coste1979}}]
	For a Coste context $ (T_0,T,\Lambda) $, the following are equivalent:
	\begin{enumerate}[label=(\roman*)]
		\item $ F \colon \mathcal{C}^{\mathrm{cart}}_{T_0} \to \mathcal{C}^{\mathrm{coh}}_{T} $ is faithful.
		\item For each $ T_0 $-model $ A $, a Horn formula $ \varphi(\bvec{u},\bvec{v}) $, and a tuple $ \bvec{a} \in A^{\bvec{u}} $,
		if $ B \models \exists!\bvec{v} \varphi(\alpha(\bvec{a}),\bvec{v}) $ for any homomorphism $ \alpha \colon A \to B $ with $ B \models T $,
		then $ A \models \exists!\bvec{v} \varphi(\bvec{a},\bvec{v}) $.
	\end{enumerate}
\end{Prop}
\begin{proof}
	(i)$ \Rightarrow $(ii):
	For a $ T_0 $-model $ A $, let $ \mathcal{L}_A $ be the language obtained by adding to $ \mathcal{L} $ a new constant symbol for each element $ a \in A $.
	We define an $ \mathcal{L}_A $-theory 
	\[ \Diag(A) := \Set{\sigma(\bvec{a}) \smid \sigma(\bvec{u})\;\text{is an atomic  formula, $ \bvec{a} \in A^{\bvec{u}} $, and}\; A \models \sigma(\bvec{a})}. \]
	Then $ T \cup \Diag(A) \models \psi(\bvec{a},\bvec{v}) \vdash \theta(\bvec{a},\bvec{v}) $ 
	implies $ T_0 \cup \Diag(A) \models \psi(\bvec{a},\bvec{v}) \vdash \theta(\bvec{a},\bvec{v}) $ 
	when $ \psi(\bvec{u},\bvec{v}) \vdash \theta(\bvec{u},\bvec{v}) $ is a $ T_0 $-cartesian sequent and $ \bvec{a} \in A^{\bvec{u}} $.
	Indeed, if $ T \cup \Diag(A) \models \psi(\bvec{a},\bvec{v}) \vdash \theta(\bvec{a},\bvec{v}) $, by compactness,
	there exist finitely many $ \sigma_1(\bvec{a}_1),\dots, \sigma_k(\bvec{a}_k) \in \Diag(A) $ such that 
	$ T \cup \{\sigma_1(\bvec{a}_1),\dots, \sigma_k(\bvec{a}_k)\} \models \psi(\bvec{a},\bvec{v}) \vdash \theta(\bvec{a},\bvec{v}) $. 
	Then 
	\begin{align*}
		& T \models \sigma_1(\bvec{a}_1) \land \dots \land \sigma_k(\bvec{a}_k) \land \psi(\bvec{a},\bvec{v}) \vdash \theta(\bvec{a},\bvec{v}), \\
		\implies & T_0 \models \sigma_1(\bvec{a}_1) \land \dots \land \sigma_k(\bvec{a}_k) \land \psi(\bvec{a},\bvec{v}) \vdash \theta(\bvec{a},\bvec{v}),
	\end{align*}
	by the assumption.
	Thus, $ T_0 \cup \Diag(A) \models \psi(\bvec{a},\bvec{v}) \vdash \theta(\bvec{a},\bvec{v}) $.
	
	Giving a homomorphism $ \alpha \colon A \to B $ with $ B \models T$ is equivalent to giving a model of $ T \cup \Diag(A) $.
	Therefore, for given $ A,\varphi,\bvec{a} $, the hypothesis of (ii) implies $ T \cup \Diag(A) \models \exists!\bvec{v} \varphi(\bvec{a},\bvec{v}) $.
	The above observation yields $ T_0 \cup \Diag(A) \models \exists!\bvec{v} \varphi(\bvec{a},\bvec{v}) $ as desired.
	
	\noindent (ii)$ \Rightarrow $(i):
	By \cite[Proposition D1.3.10(ii)]{Elephant}, instead of taking account of arbitrary $ T_0 $-cartesian sequents,
	we only have to consider a sequent of the form $ \psi(\bvec{u}) \vdash \exists \bvec{v}\theta(\bvec{u},\bvec{v}) $ for Horn formulas $ \psi,\theta $.
	Suppose that $ T \models \psi \vdash \exists \bvec{v} \theta $.
	Applying the assumption to the case when $ A=M_{\psi} $, $\varphi \equiv \theta $, and $ \bvec{u} \in M_{\psi} $,
	we obtain $ M_{\psi} \models \exists!\bvec{v}\theta(\bvec{u},\bvec{v}) $, which implies $ T_0 \models \psi \vdash \exists\bvec{v}\theta $.
\end{proof}

\section{Colimits of Modelled Spaces} \label{sec:colimits}
The category of locally ringed spaces is known to be cocomplete \cite[Chap.\ I, Proposition 1.6]{DemaGab1980}.
In this section, we will generalize this by proving that $ \modsp{\mathbb{A}}{\mathbf{Sp}} $ has all small colimits for any spatial Coste context $ (T_0,T,\Lambda) $.
It suffices to construct coproducts and coequalizers.
These colimits will be preserved by the forgetful functor $ \modsp{\mathbb{A}}{\mathbf{Sp}} \to \modsp{T_0}{\mathbf{Sp}} $.
If we consider the trivial Coste context $ (T_0,T_0,\varnothing) $, the following proof also shows that $ \modsp{T_0}{\mathbf{Sp}} $ is cocomplete.
We will also provide
\begin{itemize}
	\item an example showing that the full subcategory of $ \modsp{T_0}{\mathbf{Sp}} $ spanned by $ T $-modelled spaces is not closed under coequalizers, and
	\item a non-spatial Coste context for which $ \modsp{\mathbb{A}}{\mathbf{Sp}} $ is not closed under coequalizers in $ \modsp{T_0}{\mathbf{Sp}} $.
\end{itemize}

Notice that the forgetful functor $ \modsp{T_0}{\mathbf{Sp}} \to \mathbf{Sp}$ has both left and right adjoints.
The left adjoint sends a space $ X $ to $ (X,\Delta_X) $, where $ \Delta_X $ is the constant sheaf of the trivial (terminal) $ T_0 $-model.
The right adjoint sends a space $ X $ to $ (X,\nabla_X) $, where $ \nabla_X $ is the locally constant sheaf of the initial $ T_0 $-model.
Thus, the underlying space of a limit (resp.\ a colimit) in $ \modsp{T_0}{\mathbf{Sp}} $ must be a limit (resp.\ a colimit) of the same type in $ \mathbf{Sp} $.

\paragraph{Coproducts}
Let $ F := \{*\} $ be the trivial (terminal) $ T_0 $-model.
The initial object of $ \modsp{\mathbb{A}}{\mathbf{Sp}} $ is the empty space $ \varnothing $ equipped with the trivial sheaf $ F $.
Indeed, $ \mathbf{Sh}(\varnothing) \simeq \mathbf{1}$ is the degenerated topos, and its internal logic makes everything valid.
Thus, $ (\varnothing,F) $ is a $ T $-modelled space.

Let $ \{(X_i,P_i)\}_{i \in I} $ be a family of $ T $-modelled spaces.
We can construct a coproduct $ (X,P) = \coprod_i (X_i,P_i) $ as follows:
$ X $ is the coproduct space $ \coprod_i X_i $. We write $ \coprod_i U_i $ for the open subset of $ X $ determined by open sets $ U_i \subseteq X_i $.
Then the $ T_0 $-model $ P(\coprod_i U_i) $ is defined to be the product $ \prod_i P_i U_i $.
The reader can easily verify that it is a $ T $-modelled space and has the desired properties as a coproduct in $ \modsp{\mathbb{A}}{\mathbf{Sp}} $.

\paragraph{Coequalizers}
Let $ f,g \colon (X,P) \rightrightarrows (Y,Q) $ be morphisms in $ \modsp{\mathbb{A}}{\mathbf{Sp}} $.
We would like to construct a coequalizer in $ \modsp{\mathbb{A}}{\mathbf{Sp}} $
\[ \tikz[auto]{
	\node (L) at (0,0) {$ (X,P) $};
	\node (C) at (2,0) {$ (Y,Q) $};
	\node (R) at (4,0) {$ (Z,R) $};
	\draw[transform canvas={yshift=3pt},->] (L) to node {$ f $} (C);
	\draw[transform canvas={yshift=-3pt},->] (L) to node[swap] {$ g$} (C);
	\draw[->] (C) to node {$ p $} (R);
}.  \]
The underlying map $ p \colon Y \to Z $ is the coequalizer of $ f,g \colon X \rightrightarrows Y $ in $ \mathbf{Sp} $.
Put $ h := pf = pg $.
For any open subset $ W \subseteq Z $ (i.e.\ $ p^{-1}W \subseteq Y $ is open), the $ T_0 $-model $ RW $ of sections of the structure sheaf $ R $
is defined to be the equalizer in $ T_0\mhyphen\mathbf{Mod} $
\[ \tikz[auto]{
	\node (L) at (0,0) {$ RW $};
	\node (C) at (3,0) {$ Q(p^{-1}W) $};
	\node (R) at (6,0) {$ P(h^{-1}W) $};
	\draw[->] (L) to node {$ p^{\sharp}_W $} (C);
	\draw[transform canvas={yshift=3pt},->] (C) to node {$ (f^{\sharp})_{p^{-1}W} $} (R);
	\draw[transform canvas={yshift=-3pt},->] (C) to node[swap] {$ (g^{\sharp})_{p^{-1}W} $} (R);
}.  \]
Then $ R $ is indeed a sheaf of $ T_0 $-models, and we have a morphism $ p \colon (Y,Q) \to (Z,R) $ of $ T_0 $-modelled spaces.

\begin{Claim}
	$ p^{\flat} \colon p^*R \to Q $ is admissible.
\end{Claim}
\begin{proof}
	We would like to show that, for any $ (\varphi(\bvec{u}),\psi(\bvec{u},\bvec{v})) \in \Lambda $ and any point $ y \in Y $ with $ z := py $, the diagram
	\[ \tikz[auto]{
		\node (UL) at (0,2) {$ \dbracket{\psi}_{R_z} $};
		\node (UR) at (2,2) {$ \dbracket{\psi}_{Q_y} $};
		\node (DL) at (0,0) {$ \dbracket{\varphi}_{R_z} $};
		\node (DR) at (2,0) {$ \dbracket{\varphi}_{Q_y} $};
		\draw[->] (UL) to node {$ p^{\flat}_y $} (UR);
		\draw[->] (UL) to (DL);
		\draw[->] (DL) to node {$ p^{\flat}_y $} (DR);
		\draw[->] (UR) to (DR);
	}  \]
	is a pullback.
	Put $ \theta \equiv \exists \bvec{v} \psi $.
	By the spatiality assumption for $ (T_0,T,\Lambda) $,
	the morphism $ \{\bvec{uv}.\,\psi\} \to \{\bvec{u}.\,\varphi\} $ in $ \mathcal{C}_{T_0} $ is monic,
	and it is isomorphic to the subobject $ \{\bvec{u}.\, \theta \} \rightarrowtail \{\bvec{u}.\,\varphi\} $.
	Therefore, assuming that $ R_z \models \varphi(\bvec{s}_z) $ and $ Q_y \models \theta(p^{\flat}_y(\bvec{s}_z)) $
	for an open set $ W \ni z $ and local sections $ \bvec{s} \in RW $,
	it suffices to show that $ R_z \models \theta(\bvec{s}_z) $.
	
	Let us consider the set
	\[ V := \left \{ y' \in p^{-1}W \smid Q_{y'} \models \theta(p^{\sharp}_W(\bvec{s})_{y'}) \right \}. \]
	By $ p^{\flat}_y(\bvec{s}_z) = p^{\sharp}_W(\bvec{s})_{y} \;$, $ V $ contains $ y $.
	It is open, and we have $ QV \models \theta(p^{\sharp}_W(\bvec{s})|_V) $
	since forcing values and discrete values in the sense of \cite{Ara2021} coincide for any cartesian formula (or, more generally, any coherent formula).
	Moreover, $ V $ is closed under the equivalence relation on $ Y $ used to define the quotient $ Z $.
	Indeed, for $ x \in h^{-1}W $,
	\begin{align*}
		fx \in V & \iff Q_{fx} \models \theta(p^{\sharp}_W(\bvec{s})_{fx}), \\
		& \iff P_{x} \models \theta(h^{\sharp}_W(\bvec{s})_{x}), \\
		& \iff Q_{gx} \models \theta(p^{\sharp}_W(\bvec{s})_{gx}), \\
		& \iff gx \in V.
	\end{align*}
	The second and third equivalences follow from admissibility of $ f^{\flat}_x \colon Q_{fx} \to P_x $ and $ g^{\flat}_x \colon Q_{gx} \to P_x $.
	Thus, $ W':=pV \subseteq W $ is open and contains $ z $.
	
	Now we have $ QV \models \exists \bvec{v}\psi(p^{\sharp}_W(\bvec{s})|_{V}, \bvec{v}) $.
	Take local sections $ \bvec{t} \in QV $ so that $ QV \models \psi(p^{\sharp}_W(\bvec{s})|_{V}, \bvec{t}) $.
	Then
	\[ P(h^{-1}W') \models \psi(h^{\sharp}_W(\bvec{s})|_{h^{-1}W'},f^{\sharp}_V(\bvec{t})) \land \psi(h^{\sharp}_W(\bvec{s})|_{h^{-1}W'},g^{\sharp}_V(\bvec{t})), \]
	and we obtain $ f^{\sharp}_V(\bvec{t}) = g^{\sharp}_V(\bvec{t}) $ by the spatiality assumption.
	Since $ p^{\sharp}_{W'} $ is an equalizer of $ f^{\sharp}_V $ and $ g^{\sharp}_V $,
	there exists a (unique) tuple $ \bvec{t'} \in RW' $ such that $ p^{\sharp}_{W'}(\bvec{t'}) = \bvec{t} $.
	Then $ QV \models \psi(p^{\sharp}_{W'}(\bvec{s}|_{W'}), p^{\sharp}_{W'}(\bvec{t'})) $, and $ RW' \models \psi(\bvec{s}|_{W'},\bvec{t'}) $.
	Therefore, $ R_z \models \psi(\bvec{s}_z,\bvec{t'}_z) $ as desired.
\end{proof}

By \emph{Lemma \ref{lem:conditions-for-being-in-A-ModSp}}, $ p^*R $ is a $ T $-model in $ \mathbf{Sh}(Y) $.
Then $ R $ is a $ T $-model in $ \mathbf{Sh}(Z) $ since $ p $ is surjective.
We have shown that $ (Z,R) $ is a $ T $-modelled space and that $ p \colon (Y,Q) \to (Z,R) $ is a morphism in $ \modsp{\mathbb{A}}{\mathbf{Sp}} $.
The proof of the universality of a coequalizer is straightforward.
We have finished the construction of the coequalizer.

\begin{Thm*}
	$ \modsp{\mathbb{A}}{\mathbf{Sp}} $ has all small colimits.
\end{Thm*}

We now exhibit examples showing that spatiality and admissibility are indispensable for creation of coequalizers.

\begin{Expl}
	Regarding the Zariski context, let us consider the following non-admissible morphisms of locally ringed spaces:
	\[ \tikz[auto]{
		\node (L) at (0,0) {$ (\{*\},\mathbb{Q}) $};
		\node (C) at (4,0) {$ (2,\mathbb{Z}_{(p)} \times \mathbb{Z}_{(q)}) $};
		\draw[transform canvas={yshift=3pt},->] (L) to node {$ f $} (C);
		\draw[transform canvas={yshift=-3pt},->] (L) to node[swap] {$ g$} (C);
	},  \]
	where $ 2 $ is the discrete two-point space, $ p,q $ are distinct prime numbers, 
	$ \mathbb{Z}_{(p)} $ (resp.\ $ \mathbb{Z}_{(q)} $) is the localization of $ \mathbb{Z} $ by the prime ideal $ p\mathbb{Z} $ (resp.\ $ q\mathbb{Z} $),
	$ f(*)=0 \in 2 $, $ g(*)=1 \in 2 $, and $ f^{\flat}_* \colon \mathbb{Z}_{(p)} \to \mathbb{Q} , g^{\flat}_* \colon \mathbb{Z}_{(q)} \to \mathbb{Q} $ are the inclusions.
	Then the coequalizer in the category of ringed spaces is (isomorphic to) the single-point space equipped with the non-local ring $ \mathbb{Z}_{(p)} \cap \mathbb{Z}_{(q)} \subseteq \mathbb{Q} $.
	Therefore, the category of locally ringed spaces and (not necessarily admissible) morphisms of ringed spaces is not closed under coequalizers of ringed spaces.
\end{Expl}

\begin{Expl}
	We consider the non-spatial Coste context $ (T_0,T_0, \Lambda) $ where $ T_0 $ is the theory of commutative rings, and $ \Lambda = \{(u=u,uv=0)\} $.
	A homomorphism $ \alpha \colon A \to B $ is admissible if it satisfies the condition
	\[ \forall a \in A,\,\forall b \in B,\, [\alpha(a)b=0 \implies \exists! a' \in A,\, aa'=0 \;\text{and}\; \alpha(a')=b]. \]
	Putting $ a=0 $, we can see that admissible homomorphisms are exactly isomorphisms.
	We exhibit a parallel pair in $ \modsp{\mathbb{A}}{\mathbf{Sp}} $ whose coequalizer as ringed spaces does not belong to that category.
	
	Let $ A $ be the quotient ring $ \mathbb{Z}[X]/\langle X^2 -X \rangle $ of the single-variable polynomial ring, and $ (Y,Q) $ be the Pierce spectrum of $ A $
	(cf.\ \emph{Example \ref{expl:contexts-for-rings}}).
	Concretely, $ Y $ consists of three points $ f(*)$, $g(*)$, and $ \star $ where $ f(*),g(*) $ will be the images of maps $ f,g \colon \{*\} \rightrightarrows Y $, respectively.
	The topology on $ Y $ is $ \{\varnothing,Y,\{f(*)\},\{g(*)\},\{f(*),g(*)\}\} $, i.e.\ the finest topology so that $ Y $ is the only open set containing $ \star $.
	The sheaf $ Q $ on $ Y $ has the stalks
	\[ Q_{f(*)} \cong Q(\{f(*)\}) \cong A/ \langle \bar{X} \rangle, \quad Q_{g(*)} \cong Q(\{g(*)\}) \cong A/ \langle \overline{1-X} \rangle, \quad Q_{\star} \cong QY \cong A. \]
	We define morphisms $ f,g \colon (\{*\},\mathbb{Z}) \rightrightarrows (Y,Q) $ as follows:
	$ f^{\sharp}_Y, g^{\sharp}_Y \colon A \rightrightarrows \mathbb{Z} $ are the homomorphisms substituting $ 0 $ and $ 1 $, respectively, for $ X $.
	Then the admissible isomorphism $ f^{\flat}_* \colon A/ \langle \bar{X} \rangle \isoarrow \mathbb{Z} $ is induced. Similarly, for $ g^{\flat}_* \colon A/\langle \overline{1-X} \rangle \isoarrow \mathbb{Z}$.
	
	When we take its coequalizer $ p \colon (Y,Q) \to (Z,R) $ in the category $ \mathbf{RS} $ of ringed spaces,
	we can see that $ Z $ is the Sierpinski space $ \{\diamondsuit,\heartsuit \} $ whose only non-trivial open set is $ \{\diamondsuit\} $.
	The underlying map of $ p $ is determined by the equations $ p(f(*))=p(g(*)) = \diamondsuit,\, p(\star) = \heartsuit $.
	$ R(Z)$ is an equalizer of $ f^{\sharp}_Y, g^{\sharp}_Y $, and thus it is isomorphic to $ \mathbb{Z} $.
	However, $ p^{\flat}_{\star} \colon R_{\heartsuit} \cong R(Z) \cong \mathbb{Z} \rightarrowtail A $ is not admissible!
\end{Expl}

\section{Spectra of Modelled Spaces} \label{sec:spectra}
Our goal in this section is the following theorem:
\begin{Thm*} \label{thm:spectra-of-modelled-spaces}
	Let $ (T_0,T,\Lambda) $ be a spatial Coste context.
	For any $ T $-modelled space $ (X,P) $, the forgetful functor between the slice categories
	\[ \modsp{\mathbb{A}}{\mathbf{Sp}}/(X,P) \to \modsp{T_0}{\mathbf{Sp}}/(X,P) \]
	has a right adjoint $ \Spec $.
\end{Thm*}
Gillam \cite{Gillam2011} and \mbox{Brandenburg} \cite{BrandLimits} discuss Zariski spectra of ringed spaces in relation to limits 
in the category $ \mathbf{LRS} $ of locally ringed spaces.
Our arguments are largely inspired by \cite{BrandLimits} but rather involved.
We point out that Blechschmidt \cite[Chapter 12]{BlechThesis} constructs relative Zariski spectra of ringed locales from an internal point of view
and compares his construction with Gillam's and Cole's.
In the next section, relative spectra of $ T_0 $-modelled spaces will be used to construct limits in $ \modsp{\mathbb{A}}{\mathbf{Sp}} $.

Osmond \cite{Osm2020b} discusses spectra of modelled spaces for \textit{Diers contexts}.
His construction should be tightly related to our results, but it is unclear to the author how exact the relationship is.
Notice that he uses the algebraic convention for morphisms of modelled spaces while we use the geometric one.

Whenever no confusion arises, for a morphism $ f \colon (Y,Q) \to (X,P) $ in $ \modsp{T_0}{\mathbf{Sp}} $ with $ (X,P) \in \modsp{\mathbb{A}}{\mathbf{Sp}} $,
we will write $ \Spec(Y,Q) $ for the domain of $ \Spec(f) $.
We will give the construction of $ \Spec(Y,Q) \to (X,P) $. 

\paragraph{The space $ \Spec(Y,Q) $}
For each point $ y \in Y $, the homomorphism $ f^{\flat}_y \colon P_{fy} \to Q_y $ induces $ \tilde{f^{\flat}_y} \colon \Spec(Q_y) \to \Spec(P_{fy}) $ as we described in \emph{Proposition \ref{prop:hom-induces-morph-bw-Spec}}.
Since $ P_{fy} \models T $, we have the distinguished ideal $ \{\id_{P_{fy}}\} \in \Spec(P_{fy}) $.
The underlying set of $ \Spec(Y,Q) $ is defined to be the set
\[ \Set{(y,I) \smid y \in Y,\, I \in \Spec(Q_y) \;\text{such that}\; \tilde{f^{\flat}_y}(I)=\{\id_{P_{fy}}\} }. \]
The condition $ \tilde{f^{\flat}_y}(I)=\{\id_{P_{fy}}\} $ can be rephrased as
\[ \forall \lambda \in \mathcal{V}_{P_{fy}},\, \left [ (f^{\flat}_y)_*(\lambda) \in I \implies \lambda = \id_{P_{fy}} \right ]. \]
For each $ V \in \mathcal{O}(Y) $ and $ \lambda \in \mathcal{V}_{QV} $, we define a subset
\[ D(V,\lambda) := \Set{(y,I) \in \Spec(Y,Q) \smid y \in V \;\text{and}\; (QV \to Q_y)_*(\lambda) \in I }. \]

\begin{Lem} \label{lem:D(V,lam)-form-a-basis}
	The subsets $ D(V,\lambda) $'s form a basis for $ \Spec(Y,Q) $.
\end{Lem}
\begin{proof}
	Every $ (y,I) \in \Spec(Y,Q) $ is contained in $ D(V,\id_{QV}) $ for any open $ V \ni y $.
	Moreover, for any $ V,V' \in \mathcal{O}(Y),\lambda \in \mathcal{V}_{QV}, $ and $ \lambda' \in \mathcal{V}_{QV'} $, 
	if we write $ \mu \in \mathcal{V}_{Q(V \cap V')}$ (resp.\ $ \mu' $) for the pushout of $ \lambda $ (resp.\ $ \lambda' $) along the restriction map $ QV \to Q(V \cap V') $ (resp.\ $ QV' \to Q(V \cap V') $),
	then we can show
	\[ D(V,\lambda) \cap D(V',\lambda') = D(V \cap V', \mu \vee \mu'). \]
	Indeed, for any $ y \in V \cap V' $, we have a rectangular diagram
	\[ \tikzset{cross/.style={preaction={-,draw=white,line width=6pt}}}
	\tikz{
		\node (QVV') at (0,0,0) {$ Q(V \cap V') $};
		\node (Qy) at (4,0,0) {$ Q_y $};
		\path (QVV') ++(-1,-2,1) node[circle,name=mu] {$ \bullet $};
		\path (QVV') ++(1,-2,-1) node[circle,name=mu'] {$ \bullet $};
		\path (QVV') ++(0,-4,0) node[circle,name=mumu'] {$ \bullet $};
		\path (Qy) ++(-1,-2,1) node[circle,name=nu] {$ \bullet $};
		\path (Qy) ++(1,-2,-1) node[circle,name=nu'] {$ \bullet $};
		\path (Qy) ++(0,-4,0) node[circle,name=nunu'] {$ \bullet $};
		\draw (QVV') to node[above left] {$ \mu $} (mu);
		\draw[->] (QVV') to node[above right] {$ \mu' $} (mu');
		\draw[dashed,->] (QVV') to node[fill=white,pos=0.4] {$ \mu \vee \mu' $} (mumu');
		\draw[->] (mu) --(mumu');
		\draw[->] (mu') to (mumu');
		
		\draw[->] (QVV') to (Qy);
		\draw[->,cross] (mu) to (nu);
		\draw[->] (mu') to (nu');
		\draw[->] (mumu') to (nunu');
		
		\draw[->,cross] (Qy) to node[above left] {$ \nu $} (nu);
		\draw[->] (Qy) to node[above right] {$ \nu' $} (nu');
		\draw[dashed,->,cross] (Qy) to node[fill=white] {$ \nu \vee \nu' $} (nunu');
		\draw[->] (nu) to (nunu');
		\draw[->] (nu') to (nunu');
	}, \]
	where each face is a pushout.
	Then $ (y,I) \in D(V,\lambda) $ iff $ \nu \in I $, and similarly for $ D(V',\lambda') $ and $ \nu' $,
	while $ (y,I) \in D(V \cap V', \mu \vee \mu') $ iff $ \nu \vee \nu' \in I $.
\end{proof}

\paragraph{The structure sheaf on $ \Spec(Y,Q) $}
We define a sheaf $ S $ of $ T_0 $-models on $ \Spec(Y,Q) $ as follows.
For any open set $ W $ of $ \Spec(Y,Q) $, the $ T_0 $-model of local sections on $ W $ is defined to be
\[ S(W) := \Set{s \in \prod_{(z,J) \in W} (Q_z)_J \smid 
	\begin{aligned}
		\forall (y,I) \in W,\, \exists D(V,\lambda) \;\text{with}\; (y,I) \in D(V,\lambda) \subseteq W, \\
		\exists t \in (QV)_{\lambda},\, \forall (z,J) \in D(V,\lambda),\, s_{(z,J)} = t_{(z,J)}
	\end{aligned}
}, \]
where $ s_{(z,J)} $ is the $ (z,J) $-th component of $ s $, and $ t_{(z,J)} $ is the image of $ t $ under the composite of the bottom line below:
\[ \tikz[auto]{
	\node (UL) at (0,2) {$ QV $};
	\node (UR) at (2,2) {$ Q_z $};
	\node (DL) at (0,0) {$ (QV)_{\lambda} $};
	\node (DR) at (2,0) {$ \bullet $};
	\node (Out) at (4,0) {$ (Q_z)_J $};
	\draw[->] (UL) to node {$ \lambda $} (DL);
	\draw[->] (UL) to (UR);
	\draw[->] (DL) to (DR);
	\draw[->] (UR) to node[right] {$ \in J $} (DR);
	\draw[bend left=10pt,->] (UR) to (Out);
	\draw[->] (DR) to (Out);
	\node (po) at (1.5,0.5) {$ \ulcorner $};
}.  \]
Since $ S(W) $ is defined locally, this is indeed a sheaf.
Hereafter, we will often use the symbol $ \bigstar $ in diagrams to denote unnamed finitely presentable models.

\begin{Prop}
	The stalk $ S_{(y,I)} $ is isomorphic to the $ T $-model $ (Q_y)_I $.
	Thus, $ (\Spec(Y,Q),S) $ is a $ T $-modelled space.
\end{Prop}
\begin{proof}
	We have the canonical homomorphism $ S_{(y,I)} \to (Q_y)_I $ sending the germ of $ s $ at $ (y,I) $ to $ s_{(y,I)} $.
	We will show that this homomorphism is a surjective embedding.
	The proof goes analogously to \emph{Lemma \ref{lem:stalk-of-structure-sheaf}} but involves many modifications.
	
	\vspace{5pt}
	\noindent\emph{Surjectivity}:
	For $ \mu \in I $ and $ r \in (Q_y)_{\mu} $, we write $ r_I $ for the image of $ r $ under $ (Q_y)_{\mu} \to (Q_y)_I $.
	Since $ (Q_y)_I $ is the filtered colimit $ \rlim_{\mu \in I} (Q_y)_{\mu} $,
	any element of $ (Q_y)_I $ is of the form $ r_I $ for some $ \mu \in I $ and $ r \in (Q_y)_{\mu} $.
	Using the lifting lemma to the homomorphism $ r \colon M_{(u=u)} \to (Q_y)_{\mu} $, we can take a diagram witnessing $ \mu \in \mathcal{V}_{Q_y} $
	\[ \tikz[auto]{
		\node (UL) at (0,2) {$ \bigstar $};
		\node (UR) at (2,2) {$ Q_y $};
		\node (DL) at (0,0) {$ \bigstar $};
		\node (DR) at (2,0) {$ (Q_y)_{\mu} $};
		\draw[->] (UL) to (UR);
		\draw[->] (UL) to node[left] {$ \mathcal{V} \ni k $} (DL);
		\draw[->] (DL) to node {$ n $} (DR);
		\draw[->] (UR) to  node {$ \mu $} (DR);
		\node (po) at (1.5,0.5) {$ \ulcorner $};
	}  \]
	such that $ r $ is in the image of $ n $.
	Since $ Q_y $ is the filtered colimit $ \rlim_{V \ni y} QV $, the above square can be factored into
	\[ \tikz[auto]{
		\node (UL) at (0,2) {$ \bigstar $};
		\node (UC) at (2,2) {$ QV $};
		\node (UR) at (4,2) {$ Q_y $};
		\node (DL) at (0,0) {$ \bigstar $};
		\node (DC) at (2,0) {$ (QV)_{\lambda} $};
		\node (DR) at (4,0) {$ (Q_y)_{\mu} $};
		\draw[->] (UL) to node {$ k $} (DL);
		\draw[->] (UL) to (UC);
		\draw[->] (DL) to (DC);
		\draw[->] (UC) to node {$ \lambda $} (DC);
		\draw[->] (UC) to (UR);
		\draw[->] (DC) to (DR);
		\draw[->] (UR) to node {$ \mu $} (DR);
		\node (po1) at (1.5,0.5) {$ \ulcorner $};
		\node (po2) at (3.5,0.5) {$ \ulcorner $};
	}.  \]
	By the choice of $ n $, there exists $ t \in (QV)_{\lambda} $ such that $ t \mapsto r $ under $ (QV)_{\lambda} \to (Q_y)_{\mu} $.
	Then $ (y,I) \in D(V,\lambda) $, and $ t $ (considered as an element of $S(D(V,\lambda)) $) has the germ $ r_I $ at $ (y,I) $.
	This shows the homomorphism $ S_{(y,I)} \to (Q_y)_I $ is surjective.
	
	\vspace{5pt}
	\noindent\emph{Embedding}:
	We will show that the homomorphism $ S_{(y,I)} \to (Q_y)_I $ reflects the interpretations of a binary relation symbol $ \rho $.
	The proof for the other relation symbols and the equality $ = $ goes similarly.
	Suppose that $ (Q_y)_I \models \rho(s_{(y,I)}, s'_{(y,I)}) $ for $ s,s' \in S(W) $ and $(y,I) \in W $.
	We would like to find a basic open $ D(V_0,\lambda_0) $ such that $ (y,I) \in D(V_0,\lambda_0) \subseteq W $ and 
	$ S(D(V_0,\lambda_0)) \models \rho(s|_{D(V_0,\lambda_0)},s'|_{D(V_0,\lambda_0)}) $.
	By the definition of $ S(W) $, we can find basic open sets $ D(V,\lambda), D(V',\lambda') $ and elements $ t \in (QV)_{\lambda}, t' \in (QV')_{\lambda'} $ such that
	\[ \begin{cases}
		(y,I) \in D(V,\lambda) \cap D(V',\lambda'),\, D(V,\lambda) \cup D(V',\lambda') \subseteq W, \\
		\forall (z,J) \in D(V,\lambda),\, s_{(z,J)} = t_{(z,J)}, \\
		\forall (z,J) \in D(V',\lambda'),\, s'_{(z,J)} = t'_{(z,J)}.
	\end{cases} \]
	By the argument in the proof of \emph{Lemma \ref{lem:D(V,lam)-form-a-basis}},
	we may assume that $ V=V' $ and $ \lambda = \lambda' $.
	Let us consider the diagram
	\[ \tikz{
		\node (UL) at (0,2) {$ QV $};
		\node (UR) at (2,2) {$ Q_y $};
		\node (DL) at (0,0) {$ (QV)_{\lambda} $};
		\node (DR) at (2,0) {$ (Q_y)_{\mu} $};
		\node (Out) at (4,0) {$ (Q_y)_I $};
		\draw[->] (UL) to node[right] {$ \lambda $} (DL);
		\draw[->] (UL) to (UR);
		\draw[->] (DL) to (DR);
		\draw[->] (UR) to node[right] {$ \mu \in I $} (DR);
		\draw[bend left=10pt,->] (UR) to (Out);
		\draw[->] (DR) to (Out);
		\node (po) at (1.5,0.5) {$ \ulcorner $};
		
		\node (BL) at (0,-1) {$ t,t' $};
		\node (BR) at (2,-1) {$ t_{\mu},t'_{\mu} $};
		\node (BOut) at (4,-1) {$ t_{(y,I)},t'_{(y,I)} $};
		\draw[|->] (BL) to (BR);
		\draw[|->] (BR) to (BOut);
		\path (BL) to node[sloped] {$ \in $} (DL);
		\path (BR) to node[sloped] {$ \in $} (DR);
		\path (BOut) to node[sloped] {$ \in $} (Out);
	}.  \]
	Since $ (Q_y)_I \models \rho(t_{(y,I)}, t'_{(y,I)}) $, there exists $ \mu_0 \in I $ such that $ \mu_0 \geq \mu $ and $ (Q_y)_{\mu_0} \models \rho(r,r') $,
	where $ r,r' $ are respectively the images of $ t_{\mu},t'_{\mu} $ under the homomorphism $ (Q_y)_{\mu} \to (Q_y)_{\mu_0} $.
	Then the following claim is all we need to show:
	\begin{Claim}
		There exist $ V_0 \subseteq V $ with $ y \in V_0 $, $ \lambda_0 \in \mathcal{V}_{QV_0} $, 
		$ t_0,t'_0 \in (QV_0)_{\lambda_0} $ with $ (QV_0)_{\lambda_0} \models \rho(t_0,t'_0) $, and a commutative diagram
		\[ \tikz{
			\node (UL) at (0,2) {$ QV $};
			\node (UR) at (3,2) {$ QV_0 $};
			\node (UOut) at (6,2) {$ Q_y $};
			\node (DL) at (0,0) {$ (QV)_{\lambda} $};
			\node (DR) at (3,0) {$ (QV_0)_{\lambda_0} $};
			\node (DOut) at (6,0) {$ (Q_y)_{\mu_0} $};
			\draw[->] (UL) to node[right] {$ \lambda $} (DL);
			\draw[->] (UL) to (UR);
			\draw[->] (DL) to (DR);
			\draw[->] (UR) to node[right] {$ \lambda_0 $} (DR);
			\draw[->] (UR) to (UOut);
			\draw[->] (DR) to (DOut);
			\draw[->] (UOut) to node[right] {$ \mu_0 $} (DOut);
			\node (po) at (5.5,0.5) {$ \ulcorner $};
			
			\node (BL) at (0,-1) {$ t,t' $};
			\node (BR) at (3,-1) {$ t_0,t'_0 $};
			\node (BOut) at (6,-1) {$ r,r' $};
			\draw[|->] (BL) to (BR);
			\draw[|->] (BR) to (BOut);
			\path (BL) to node[sloped] {$ \in $} (DL);
			\path (BR) to node[sloped] {$ \in $} (DR);
			\path (BOut) to node[sloped] {$ \in $} (DOut);
		},  \]
		where the left square need not be a pushout. 
		In particular, $ (y,I) \in D(V_0,\lambda_0) $. 
		
		\ \hfill $ \square $
	\end{Claim}
	Once we can prove this claim, then $ D(V_0,\lambda_0) \subseteq D(V,\lambda) $ (see the diagram below).
	For any $ (z,J) \in D(V_0,\lambda_0) $, by considering the diagram
	\[ \tikzset{cross/.style={preaction={-,draw=white,line width=6pt}}}
	
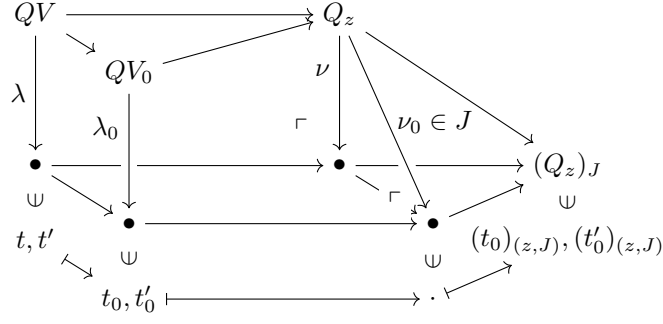
\tikz{
		\node (QV) at (0,2,0) {$ QV $};
		\node (QV0) at (2,2,2) {$ QV_0 $};
		\node (Qz) at (4,2,0) {$ Q_z $};
		\node (lam) at (0,0,0) {$ \bullet $};
		\node (lam0) at (2,0,2) {$ \bullet $};
		\node (nu) at (4,0,0) {$ \bullet $};
		\node (nu0) at (6,0,2) {$ \bullet $};
		\node (QJ) at (7,0,0) {$ (Q_z)_J $};
		
		\draw[->] (QV) to (Qz);
		\draw[->] (QV) to node[left] {$ \lambda $} (lam);
		\draw[->] (lam) to (nu);
		\draw[->] (Qz) to node[pos=0.3,left] {$ \nu $} (nu);
		\draw[->] (Qz) to (QJ);
		\draw[->] (nu) to (QJ);
		
		\draw[->] (QV) to (QV0);
		\draw[->,cross] (QV0) to node[pos=0.3,left] {$ \lambda_0 $} (lam0);
		\draw[->] (lam) to (lam0);
		\draw[->] (QV0) to (Qz);
		\draw[->] (lam0) to (nu0);
		\draw[->] (nu) to (nu0);
		\draw[->,cross] (Qz) to node[right] {$ \nu_0 \in J $} (nu0);
		\draw[->] (nu0) to (QJ);
		
		\node (po1) at (3.5,0.5,0) {$ \ulcorner $};
		\node[fill=white] (po2) at (5.5,0.3,2) {$ \ulcorner $};
		
		\node (Blam) at (0,-1,0) {$ t,t' $};
		\node (Blam0) at (2,-1,2) {$ t_0,t'_0 $};
		\node (Bnu0) at (6,-1,2) {$ \cdot $};
		\node (BQJ) at (7,-1,0) {$ (t_0)_{(z,J)}, (t'_0)_{(z,J)} $};
		\draw[|->] (Blam) to (Blam0);
		\draw[|->] (Blam0) to (Bnu0);
		\draw[|->] (Bnu0) to (BQJ);
		\path (Blam) to node[sloped] {$ \in $} (lam);
		\path (Blam0) to node[sloped] {$ \in $} (lam0);
		\path (Bnu0) to node[sloped] {$ \in $} (nu0);
		\path (BQJ) to node[sloped] {$ \in $} (QJ);
	}, \]
	we have $ t_{(z,J)} = (t_0)_{(z,J)} $, $ t'_{(z,J)} = (t'_0)_{(z,J)} $, and $ (Q_z)_J \models \rho(t_{(z,J)}, t'_{(z,J)}) $.
	Therefore, $ S(D(V_0,\lambda_0)) \models \rho(s|_{D(V_0,\lambda_0)},s'|_{D(V_0,\lambda_0)}) $ as desired.
	
	We now proceed to the proof of the claim.
	Like the proof of surjectivity, using the lifting lemma, we can find a diagram
	
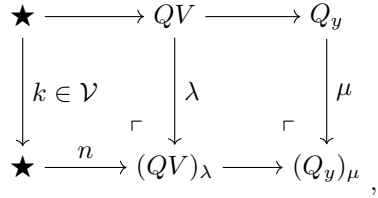
\[ \tikz[auto]{
		\node (UL) at (0,2) {$ \bigstar $};
		\node (UC) at (2,2) {$ QV $};
		\node (UR) at (4,2) {$ Q_y $};
		\node (DL) at (0,0) {$ \bigstar $};
		\node (DC) at (2,0) {$ (QV)_{\lambda} $};
		\node (DR) at (4,0) {$ (Q_y)_{\mu} $};
		\draw[->] (UL) to node {$ k \in \mathcal{V} $} (DL);
		\draw[->] (UL) to (UC);
		\draw[->] (DL) to node {$ n $} (DC);
		\draw[->] (UC) to node {$ \lambda $} (DC);
		\draw[->] (UC) to (UR);
		\draw[->] (DC) to (DR);
		\draw[->] (UR) to node {$ \mu $} (DR);
		\node (po1) at (1.5,0.5) {$ \ulcorner $};
		\node (po2) at (3.5,0.5) {$ \ulcorner $};
	},  \]
	where $ t,t' $ are in the image of $ n $.
	Choose $ \tau,\tau' $ from the preimages of $ t,t' $, respectively.
	By the proof of right cancellation in \emph{Lemma \ref{lem:closure-properties-of-V_A}},
	we can also find a diagram
	\[ \tikzset{cross/.style={preaction={-,draw=white,line width=6pt}}}
	
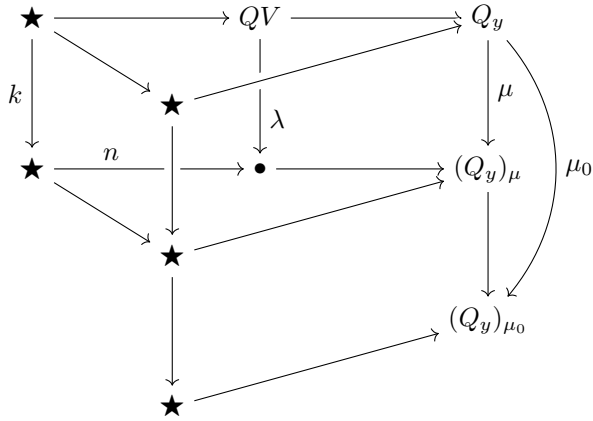
\tikz[auto]{
		\node (k1) at (0,2,0) {$ \bigstar $};
		\node (k2) at (0,0,0) {$ \bigstar $};
		\node (k'1) at (3,2,3) {$ \bigstar $};
		\node (k'2) at (3,0,3) {$ \bigstar $};
		\node (QV) at (3,2,0) {$ QV$};
		\node (lam) at (3,0,0) {$ \bullet $};
		\node (Qy) at (6,2,0) {$ Q_y $};
		\node (mu) at (6,0,0) {$ (Q_y)_{\mu} $};
		\node (k'3) at (3,-2,3) {$ \bigstar $};
		\node (mu0) at (6,-2,0) {$ (Q_y)_{\mu_0} $};
		
		\draw[->] (k1) to (QV);
		\draw[->] (QV) to node[right,pos=0.7] {$ \lambda $} (lam);
		\draw[->] (k1) to node[left] {$ k $} (k2);
		\draw[->] (k2) to node[above,pos=0.3] {$ n $} (lam);
		\draw[->] (QV) to (Qy);
		\draw[->] (lam) to (mu);
		\draw[->] (Qy) to node {$ \mu $} (mu);
		
		\draw[->] (k1) to (k'1);
		\draw[->] (k2) to (k'2);
		\draw[->,cross] (k'1) to (k'2);
		\draw[->,cross] (k'1) to (Qy);
		\draw[->] (k'2) to (mu);
		\draw[->] (k'2) to (k'3);
		\draw[->] (k'3) to (mu0);
		\draw[->] (mu) to (mu0);
		\draw[bend left=40pt,->] (Qy) to node {$ \mu_0 $} (mu0);
	}, \]
	where each vertical arrow of the form $ \bigstar \to \bigstar $ is in $ \mathcal{V} $, and each vertical square is a pushout.
	By applying \emph{Lemma \ref{lem:useful-lemma}} to the upper face,
	we obtain a factorization of the upper front square as
	\[ \tikzset{cross/.style={preaction={-,draw=white,line width=6pt}}}
	
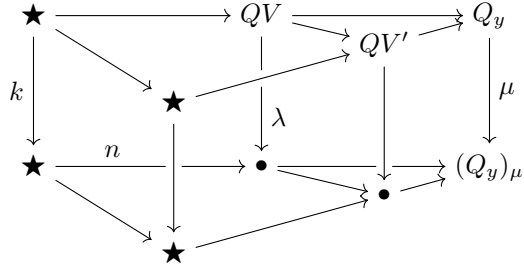
\tikz[auto]{
		\node (k1) at (0,2,0) {$ \bigstar $};
		\node (k2) at (0,0,0) {$ \bigstar $};
		\node (k'1) at (3,2,3) {$ \bigstar $};
		\node (k'2) at (3,0,3) {$ \bigstar $};
		\node (QV) at (3,2,0) {$ QV$};
		\node (lam) at (3,0,0) {$ \bullet $};
		\node (QV') at (5,2,1) {$ QV'$};
		\node (lam') at (5,0,1) {$ \bullet $};
		\node (Qy) at (6,2,0) {$ Q_y $};
		\node (mu) at (6,0,0) {$ (Q_y)_{\mu} $};
		
		\draw[->] (k1) to (QV);
		\draw[->] (QV) to node[right,pos=0.7] {$ \lambda $} (lam);
		\draw[->] (k1) to node[left] {$ k $} (k2);
		\draw[->] (k2) to node[above,pos=0.3] {$ n $} (lam);
		\draw[->] (QV) to (Qy);
		\draw[->] (lam) to (mu);
		\draw[->] (Qy) to node {$ \mu $} (mu);
		\draw[->,cross] (QV') to (lam');
		\draw[->] (QV') to (Qy);
		\draw[->] (lam') to (mu);
		
		\draw[->] (k1) to (k'1);
		\draw[->] (k2) to (k'2);
		\draw[->,cross] (k'1) to (k'2);
		\draw[->,cross] (k'1) to (QV');
		\draw[->] (k'2) to (lam');
		\draw[->] (QV) to (QV');
		\draw[->] (lam) to (lam');
	}. \]
	Thus, by replacing $ V$, $\lambda $, etc., we may assume that there exist pushouts
	
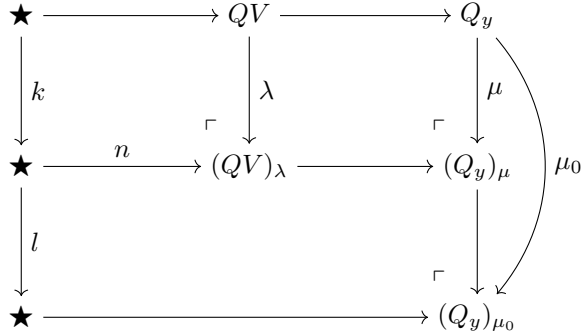
\[ \tikz{
		\node (UL) at (0,4) {$ \bigstar $};
		\node (UC) at (3,4) {$ QV $};
		\node (UR) at (6,4) {$ Q_y $};
		\node (DL) at (0,2) {$ \bigstar $};
		\node (DC) at (3,2) {$ (QV)_{\lambda} $};
		\node (DR) at (6,2) {$ (Q_y)_{\mu} $};
		\draw[->] (UL) to node[right] {$ k $} (DL);
		\draw[->] (UL) to (UC);
		\draw[->] (DL) to node[above] {$ n $} (DC);
		\draw[->] (UC) to node[right] {$ \lambda $} (DC);
		\draw[->] (UC) to (UR);
		\draw[->] (DC) to (DR);
		\draw[->] (UR) to node[right] {$ \mu $} (DR);
		
		\node (BL) at (0,0) {$ \bigstar $};
		\node (BR) at (6,0) {$ (Q_y)_{\mu_0} $};
		\draw[->] (BL) to (BR);
		\draw[->] (DL) to node[right] {$ l $} (BL);
		\draw[->] (DR) to (BR);
		
		\path (2.5,2.5) node {$ \ulcorner $};
		\path (5.5,2.5) node {$ \ulcorner $};
		\path (5.5,0.5) node {$ \ulcorner $};
		\draw[bend left=40pt,->] (UR) to node[right] {$ \mu_0 $} (BR);
	}.  \]
	Put $ \bar{\tau} := l(\tau),\, \bar{\tau}' := l(\tau') $.
	Then $ \bar{\tau},\bar{\tau}' $ are sent to $ r,r' $, respectively, by the bottom homomorphism.
	Applying the above argument to the outer rectangle, we obtain
	\[ \tikzset{cross/.style={preaction={-,draw=white,line width=6pt}}}
	
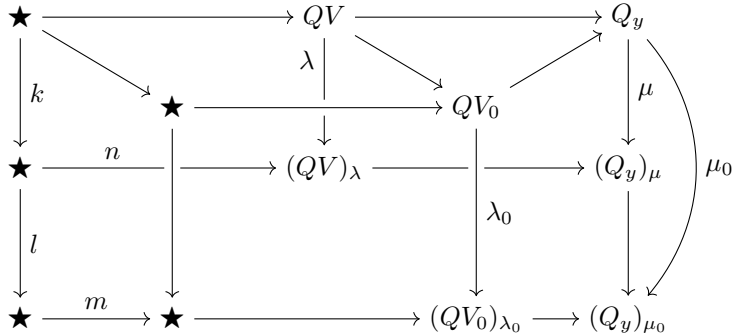
\tikz{
		\node (UL) at (0,4) {$ \bigstar $};
		\node (UC) at (4,4) {$ QV $};
		\node (UR) at (8,4) {$ Q_y $};
		\node (DL) at (0,2) {$ \bigstar $};
		\node (DC) at (4,2) {$ (QV)_{\lambda} $};
		\node (DR) at (8,2) {$ (Q_y)_{\mu} $};
		
		\node (fUL) at (2,2.8) {$ \bigstar $};
		\node (fDL) at (2,0) {$ \bigstar $};
		\node (fUR) at (6,2.8) {$ QV_0 $};
		\node (fDR) at (6,0) {$ (QV_0)_{\lambda_0} $};
		
		\draw[->] (UL) to node[right] {$ k $} (DL);
		\draw[->] (UL) to (UC);
		\draw[->] (DL) to node[above,pos=0.3] {$ n $} (DC);
		\draw[->] (UC) to node[left,pos=0.2] {$ \lambda $} (DC);
		\draw[->] (UC) to (UR);
		\draw[->] (DC) to (DR);
		\draw[->] (UR) to node[right] {$ \mu $} (DR);
		
		\node (BL) at (0,0) {$ \bigstar $};
		\node (BR) at (8,0) {$ (Q_y)_{\mu_0} $};
		\draw[->] (BL) to node[above] {$ m $} (fDL); 
		\draw[->] (fDL) to (fDR);
		\draw[->] (fDR) to (BR);
		\draw[->] (DL) to node[right] {$ l $} (BL);
		\draw[->] (DR) to (BR);
		
		\draw[->] (UL) to (fUL);
		\draw[->] (UC) to (fUR);
		\draw[->] (fUR) to (UR);
		\draw[->,cross] (fUL) to (fUR);
		\draw[->,cross] (fUL) to (fDL);
		\draw[->,cross] (fUR) to node[right] {$ \lambda_0 $} (fDR);
		\draw[bend left=40pt,->] (UR) to node[right] {$ \mu_0 $} (BR);
	}.  \]
	Here, we may further assume that $ \rho(m(\bar{\tau}), m(\bar{\tau}')) $ holds, for $ (Q_y)_{\mu_0} \models \rho(r,r') $
	and $ \mu_0 $ is the filtered colimit of factorizations of the outer rectangle.
	By the universality of the above left pushout, we have a homomorphism $ (QV)_{\lambda} \to (QV_0)_{\lambda_0} $.
	Let $ t_0,t'_0 $ be the images of $ t,t' $ under $ (QV)_{\lambda} \to (QV_0)_{\lambda_0} $, respectively.
	Then $ \rho(m(\bar{\tau}), m(\bar{\tau}')) $ implies $ (QV_0)_{\lambda_0} \models \rho(t_0,t'_0) $.
	We thus have finished the proof of the claim and the proposition.
\end{proof}

The canonical projection $ \Spec(Y,Q) \to Y $ is continuous because the preimage of $ V \in \mathcal{O}(Y) $ is $ D(V,\id_{QV}) $.
We define a morphism $ g \colon (\Spec(Y,Q),S) \to (Y,Q)$ by putting
\[ g^{\sharp}_V \colon QV \ni t \mapsto \{ t_{(y,I)} \}_{(y,I) \in D(V,\id)} \in S(D(V,\id_{QV})). \]
We can alternatively define $ g $ by putting $ g^{\flat}_{(y,I)} \colon Q_y \to (Q_y)_I \cong S_{(y,I)} $ for each $ (y,I) \in \Spec(Y,Q) $.

\begin{Prop}
	The composite $ fg \colon (\Spec(Y,Q),S) \to (X,P) $ is a morphism in $ \modsp{\mathbb{A}}{\mathbf{Sp}} $, 
	i.e.\ the composite $ (fg)^{\flat}_{(y,I)} \colon P_{fy} \to Q_y \to (Q_y)_I $ is admissible for each $ (y,I) \in \Spec(Y,Q) $.
\end{Prop}
\begin{proof}
	Suppose that there exists a commutative diagram
	\[ \tikz[auto]{
		\node (UL) at (0,3) {$ \bigstar $};
		\node (UR) at (2,3) {$ P_{fy} $};
		\node (DL) at (0,0) {$ \bigstar $};
		\node (DR) at (2,0) {$ (Q_y)_{I} $};
		\node (MR) at (2,1.5) {$ Q_y $};
		\draw[->] (UL) to (UR);
		\draw[->] (UL) to node[left] {$ \mathcal{V} \ni k $} (DL);
		\draw[->] (DL) to node {} (DR);
		\draw[->] (UR) to  node {$ f^{\flat}_y $} (MR);
		\draw[->] (MR) to  node {$ g^{\flat}_{(y,I)} $} (DR);
	}.  \]
	Taking pushouts of $ k $, we obtain
	\[ \tikz[auto]{
		\node (UL) at (0,2) {$ \bigstar $};
		\node (UC) at (2,2) {$ P_{fy} $};
		\node (UR) at (4,2) {$ Q_y $};
		\node (DL) at (0,0) {$ \bigstar $};
		\node (DC) at (2,0) {$ \bullet $};
		\node (DR) at (4,0) {$ \bullet $};
		\node (Out) at (6,0) {$ (Q_y)_I $};
		\draw[->] (UL) to node {$ k $} (DL);
		\draw[->] (UL) to (UC);
		\draw[->] (DL) to (DC);
		\draw[->] (UC) to node {$ \lambda $} (DC);
		\draw[->] (UC) to node {$ f^{\flat}_y $} (UR);
		\draw[->] (DC) to (DR);
		\draw[->] (UR) to node {$ \mu $} (DR);
		\draw[bend left=20pt,->] (UR) to  node {$ g^{\flat}_{(y,I)} $} (Out);
		\draw[->] (DR) to (Out);
		\node (po1) at (1.5,0.5) {$ \ulcorner $};
		\node (po2) at (3.5,0.5) {$ \ulcorner $};
	}.  \]
	By an argument similar to the proof of \emph{Proposition \ref{prop:hom-induces-morph-bw-Spec}},
	we can show $ \mu \in I $.
	By the definition of $ \Spec(Y,Q) $, the morphism $ \lambda $ must be the identity.
	Thus, we obtain the desired filler.
\end{proof}

\paragraph{The Adjunction}

\begin{Prop}
	For any $ k \colon (Z,R) \to (X,P) $ in $ \modsp{\mathbb{A}}{\mathbf{Sp}} $, there exists a bijection between $ \Hom $-sets
	\[ \modsp{T_0}{\mathbf{Sp}}_{/(X,P)}((Z,R),(Y,Q)) \cong \modsp{\mathbb{A}}{\mathbf{Sp}}_{/(X,P)}((Z,R),(\Spec(Y,Q),S)). \]
\end{Prop}
\begin{proof}
	Suppose we have a commutative diagram in $ \modsp{T_0}{\mathbf{Sp}} $
	\[ \tikz[auto]{
		\node (UL) at (0,2) {$ (Z,R) $};
		\node (UR) at (4,2) {$ (Y,Q) $};
		\node (DC) at (2,0) {$ (X,P) $};
		\draw[->] (UL) to node {$ h $} (UR);
		\draw[->] (UL) to node[swap] {$ k $} (DC);
		\draw[->] (UR) to node {$ f $} (DC);
	}. \]
	First, we would like to get a factorization of $ h $ through $ g \colon \Spec(Y,Q) \to Y $ as
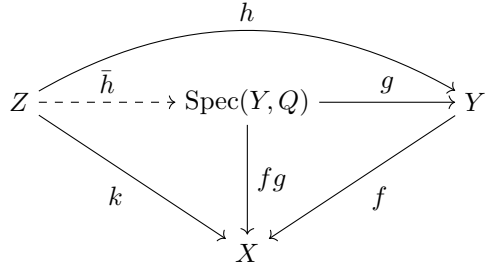
	\[ \tikz[auto]{
		\node (UL) at (0,2) {$ Z $};
		\node (UC) at (3,2) {$ \Spec(Y,Q) $};
		\node (UR) at (6,2) {$ Y $};
		\node (DC) at (3,0) {$ X $};
		\draw[dashed,->] (UL) to node {$ \bar{h} $} (UC);
		\draw[->] (UC) to node {$ fg $} (DC);
		\draw[->] (UC) to node {$ g $} (UR);
		\draw[->] (UL) to node[swap] {$ k $} (DC);
		\draw[->] (UR) to node {$ f $} (DC);
		\draw[bend left=30pt,->] (UL) to node {$ h $} (UR);
	}. \]
	Notice that, for each $ z \in Z $, we have a commutative diagram of $ T_0 $-models
	
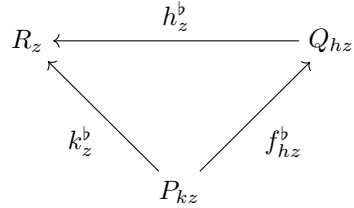
\[ \tikz[auto]{
		\node (UL) at (0,2) {$ R_z $};
		\node (UR) at (4,2) {$ Q_{hz} $};
		\node (DC) at (2,0) {$ P_{kz} $};
		\draw[->] (UR) to node[above] {$ h^{\flat}_z $} (UL);
		\draw[->] (DC) to node {$ k^{\flat}_z $} (UL);
		\draw[->] (DC) to node[swap] {$ f^{\flat}_{hz} $} (UR);
	}. \]
	Let us put $ \bar{h}(z) := (hz,J_z) $, where $ J_z := \tilde{h^{\flat}_z}(\{\id_{R_z}\}) \in \Spec(Q_{hz}) $.
	The reader can easily show that, for any admissible morphism $ \alpha \colon A \to B $ between $ T $-models,
	the induced map $ \tilde{\alpha} $ sends the ideal $ \{\id_B \} $ to $ \{\id_A \} $.
	Since $ k^{\flat}_z $ is admissible, we have $ \bar{h}(z) \in \Spec(Y,Q) $.
	Thus, we obtain a map $ \bar{h} \colon Z \to \Spec(Y,Q) $ between the underlying sets.
	
	To show continuity of $ \bar{h} $, we will prove that $ \bar{h}^{-1}(D(V,\lambda)) $ is open.
	For any point $ w \in h^{-1}V $, it belongs to $ \bar{h}^{-1}(D(V,\lambda)) $ exactly when there exists a factorization 
	
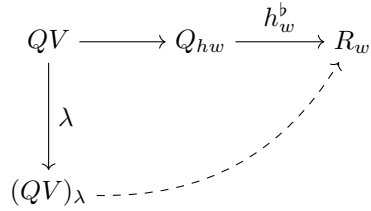
\[ \tikz[auto]{
		\node (UL) at (0,2) {$ QV $};
		\node (UC) at (2,2) {$ Q_{hw} $};
		\node (UR) at (4,2) {$ R_w $};
		\node (DL) at (0,0) {$ (QV)_{\lambda} $};
		\draw[->] (UL) to node {$ \lambda $} (DL);
		\draw[->] (UL) to (UC);
		\draw[->] (UC) to node {$ h^{\flat}_w $} (UR);
		\draw[dashed,bend right=30pt,->] (DL) to (UR);
	}.  \]
	Fix $ z \in \bar{h}^{-1}(D(V,\lambda)) $. Then we can find an open set $ W \subseteq \bar{h}^{-1}(D(V,\lambda)) $ containing $ z $ as follows.
	The diagram
	
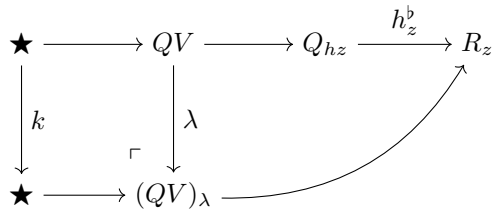
\[ \tikz[auto]{
		\node (UL) at (0,2) {$ QV $};
		\node (UC) at (2,2) {$ Q_{hz} $};
		\node (UR) at (4,2) {$ R_z $};
		\node (DL) at (0,0) {$ (QV)_{\lambda} $};
		\draw[->] (UL) to node {$ \lambda $} (DL);
		\draw[->] (UL) to (UC);
		\draw[->] (UC) to node {$ h^{\flat}_z $} (UR);
		\draw[bend right=30pt,->] (DL) to (UR);
		
		\node (ULL) at (-2,2) {$ \bigstar $};
		\node (DLL) at (-2,0) {$ \bigstar $};
		\draw[->] (ULL) to node {$ k $} (DLL);
		\draw[->] (ULL) to (UL);
		\draw[->] (DLL) to (DL);
		\node (po1) at (-0.5,0.5) {$ \ulcorner $};
	}  \]
	can be rewritten into a commutative square
	\[ \tikz[auto]{
		\node (UL) at (0,2) {$ QV $};
		\node (UC) at (3,2) {$ R(h^{-1}V) $};
		\node (UR) at (3,0) {$ R_z $};
		\node (DL) at (0,0) {$ (QV)_{\lambda} $};
		\draw[->] (UL) to node {$ h^{\sharp}_V $} (UC);
		\draw[->] (UC) to (UR);
		\draw[->] (DL) to (UR);
		
		\node (ULL) at (-2,2) {$ \bigstar $};
		\node (DLL) at (-2,0) {$ \bigstar $};
		\draw[->] (ULL) to node {$ k $} (DLL);
		\draw[->] (ULL) to (UL);
		\draw[->] (DLL) to (DL);
	}.  \]
	By \emph{Lemma \ref{lem:useful-lemma}}, we obtain $ W \subseteq h^{-1}V $ containing $ z $ and a commutative diagram
	\[ \tikz[auto]{
		\node (UL) at (0,2) {$ QV $};
		\node (UC) at (3,2) {$ R(h^{-1}V) $};
		\node (DR) at (3,0) {$ RW $};
		\node (DL) at (0,0) {$ (QV)_{\lambda} $};
		\draw[->] (UL) to node {$ h^{\sharp}_V $} (UC);
		\draw[->] (UC) to (DR);
		\draw[->] (DL) to (DR);
		\draw[->] (UL) to node {$ \lambda $} (DL);
	}.  \]
	This $ W $ has the desired property.
	
	Our next task is giving an admissible morphism $ \bar{h}^{\flat}_z \colon S_{(hz,J_z)} \cong (Q_{hz})_{J_z} \to R_z $ for each $ z \in Z $, but this is easy.
	For each $ \mu \in J_z $, we have a unique factorization of $ h^{\flat}_z \colon Q_{hz} \to R_z $ through $ \mu $,
	and these factorizing morphisms yield a morphism $ (Q_{hz})_{J_z} \to R_z $ as desired.
	It is routine work to verify admissibility of $ \bar{h}^{\flat}_z $, and we leave it to the reader.
	
	We now have obtained $ \bar{h} $ in $ \modsp{\mathbb{A}}{\mathbf{Sp}} $ such that the diagram below in $ \modsp{T_0}{\mathbf{Sp}} $ is commutative
	\[ \tikz[auto]{
		\node (UL) at (0,2) {$ (Z,R) $};
		\node (UC) at (3,2) {$ (\Spec(Y,Q),S) $};
		\node (UR) at (6,2) {$ (Y,Q) $};
		\node (DC) at (3,0) {$ (X,P) $};
		\draw[dashed,->] (UL) to node {$ \bar{h} $} (UC);
		\draw[->] (UC) to node {$ fg $} (DC);
		\draw[->] (UC) to node {$ g $} (UR);
		\draw[->] (UL) to node[swap] {$ k $} (DC);
		\draw[->] (UR) to node {$ f $} (DC);
		\draw[bend left=30pt,->] (UL) to node {$ h $} (UR);
	}. \]
	It is straightforward to see that such an $ \bar{h} $ is uniquely determined.
	Hence, we have established the desired bijection.
\end{proof}
For each morphism $ h \colon (Y,Q) \to (Y',Q') $ of $ T_0 $-modelled spaces, by applying the previous proposition to $ hg \colon \Spec(Y,Q) \to (Y',Q') $,
we can obtain a morphism $ \Spec(Y,Q) \to \Spec(Y',Q') $ in $ \modsp{\mathbb{A}}{\mathbf{Sp}}/(X,P) $.
Thus, the $ \Spec $-construction can be extended to a functor $ \Spec \colon \modsp{T_0}{\mathbf{Sp}}/(X,P) \to \modsp{\mathbb{A}}{\mathbf{Sp}}/(X,P) $.
The previous proposition also shows that $ \Spec $ is a right adjoint to the forgetful functor.
We have finished proving \emph{Theorem \ref{thm:spectra-of-modelled-spaces}}.

Closing this section, we fulfill the promise we made at the end of \S\ref{subsec:introducing-modelled-spaces}.
If we eliminate all the above mentions to $ (X,P) $ and $ f $, the same proofs show that there exists an adjunction
\[ \tikz{
	\node (L) at (0,0) {$ \modsp{\mathbb{A}}{\mathbf{Sp}} $};
	\node (R) at (4,0) {$ \modsp{T_0}{\mathbf{Sp}} $};
	\draw[bend left=15pt,->] (L) to node[above] {forgetful} node[midway,name=U] {} (R);
	\draw[bend left=15pt,->] (R) to node[below] {$ \Spec $} node[midway,name=D] {} (L);
	\path (D) to node[sloped] {$ \vdash $} (U);
},  \]
where $ \Spec(Y,Q) = \Set{(y,I) \smid y \in Y,\, I \in \Spec(Q_y)} $.
On the other hand, the global section functor $ \Gamma \colon \modsp{T_0}{\mathbf{Sp}} \to T_0\mhyphen\mathbf{Mod}^{\op} $ has the right adjoint
sending a $ T_0 $-model $ A $ to the $ T_0 $-modelled space $ (\{*\},A) $.
By \emph{Definition \ref{def:local-sections-of-str-sheaf}},
the composite of these right adjoints $ T_0\mhyphen\mathbf{Mod}^{\op} \to \modsp{T_0}{\mathbf{Sp}} \xrightarrow{\Spec} \modsp{\mathbb{A}}{\mathbf{Sp}} $
is naturally isomorphic to the functor $ \Spec $ we considered in \S\ref{subsec:spectra-of-models}.
Composing these adjunctions, we obtain the desired adjunction
\[ \tikz{
	\node (L) at (0,0) {$ \modsp{\mathbb{A}}{\mathbf{Sp}} $};
	\node (R) at (4,0) {$ T_0\mhyphen\mathbf{Mod}^{\op} $};
	\draw[bend left=15pt,->] (L) to node[above] {$ \Gamma $} node[midway,name=U] {} (R);
	\draw[bend left=15pt,->] (R) to node[below] {$ \Spec $} node[midway,name=D] {} (L);
	\path (D) to node[sloped] {$ \vdash $} (U);
}.  \]

\section{Limits of Modelled Spaces} \label{sec:limits}
We will construct limits in $ \modsp{\mathbb{A}}{\mathbf{Sp}} $ from limits in $ \modsp{T_0}{\mathbf{Sp}} $ by using relative spectra.
We will basically follow the arguments of \cite[Theorem 2.1]{BrandLimits}.
Blechschmidt \cite[\S12.7]{BlechThesis} follows Gillam's argument \cite{Gillam2011}
and also gives a counter-example showing that $ \mathbf{LRS} $ is not closed under pullbacks in $ \mathbf{RS} $.

We first describe limits in $ \modsp{T_0}{\mathbf{Sp}} $.
As we remarked at the beginning of \S\ref{sec:colimits}, limits in $ \modsp{T_0}{\mathbf{Sp}} $ should be preserved 
by the forgetful functor $ \modsp{T_0}{\mathbf{Sp}} \to \mathbf{Sp} $.
One can obtain those limits by a general construction (see \emph{Remark \ref{rmk:limits-of-T0-modelled-spaces}} below), but we first take a closer look at a few special cases.
The terminal object is the single-point space $ \{ * \} $ equipped with the initial $ T_0 $-model as a structure sheaf.

For (infinitary) products, let $ \{ (X_i,P_i) \}_{i \in I} $ be a family of $ T_0 $-modelled spaces.
The product $ (Z,R) $ has the underlying space $ Z = \prod_i X_i $.
Given open sets $ U_i \subseteq X_i $ so that $ U_i = X_i $ for all but finite $ i $,
we will regard $ \prod_i U_i $ as a basic open subset of $ Z $.
We define the sheaf $ R $ of $ T_0 $-models on $ Z $ as follows: for any open $ W \subseteq Z $,
\[ RW := \Set{s \in \prod_{w \in W}\coprod_i (P_i)_{w_i} \smid 
	\begin{aligned}
		\forall z \in W,\, \exists \prod_i U_i \;\text{with}\; z \in \prod_i U_i \subseteq W, \\
		\exists t \in \coprod_i P_i U_i ,\, \forall w \in \prod_i U_i ,\, s_w = t_w
	\end{aligned}
}, \]
where the map $ \coprod_i P_i U_i \ni t \mapsto t_w \in \coprod_i (P_i)_{w_i} $ is the homomorphism making the following diagram commutative for each $ i $
\[ \tikz[auto]{
	\node (UL) at (0,1.5) {$ P_i U_i $};
	\node (DL) at (0,0) {$ \coprod_i P_i U_i $};
	\node (UR) at (3,1.5) {$ (P_i)_{w_i} $};
	\node (DR) at (3,0) {$ \coprod_i (P_i)_{w_i} $};
	
	\draw[->] (UL) to (DL);
	\draw[->] (UR) to (DR);
	\draw[dashed,->] (DL) to (DR);
	\draw[->] (UL) to (UR);
}. \]

Given a point $ z \in Z $ and an open $ W \ni z $, we have a homomorphism $ RW \ni s \mapsto s_z \in \coprod_i (P_i)_{z_i} $.
Thus, there exists a canonical homomorphism $ R_z \to \coprod_i (P_i)_{z_i} $.

\begin{Claim}
	The above homomorphism is an isomorphism: $ R_z \cong \coprod_i (P_i)_{z_i} $.
\end{Claim}
\begin{proof}
	Note that the coproduct $ \coprod_i (P_i)_{z_i} $ is isomorphic to the filtered colimit 
	\[ \rlim_{\prod_i U_i \ni z} \coprod_i P_i U_i. \]
	This is because the family $ \{ (P_i)_{z_i}\}_i $ is a filtered colimit of the diagram
	\[ \Set{\text{basic open sets containing $ z $}}^{\op} \ni \prod_i U_i \quad \mapsto \quad \{ P_i U_i \}_i \in T_0\mhyphen\mathbf{Mod}^I. \]
	
	\vspace{5pt}
	\noindent\emph{Surjectivity}:
	Any element of $ \coprod_i (P_i)_{z_i} $ is in the image of some homomorphism $ \coprod_i P_i U_i \to \coprod_i (P_i)_{z_i} $.
	This homomorphism equals to the composite of $ \coprod_i P_i U_i \to R(\prod_i U_i) \to \coprod_i (P_i)_{z_i} $.
	Thus, $ R_z \to \coprod_i (P_i)_{z_i} $ is surjective.
	
	\vspace{5pt}
	\noindent\emph{Embedding}:
	Let $ \rho $ be a binary relation symbol.
	Suppose that an open $ W $ contains $ z $, and $ s,s' \in RW $ satisfy $ \coprod_i (P_i)_{z_i} \models \rho(s_z,s'_z)  $.
	By the definition of $ RW $, we can take  a basic open $ \prod_i U_i \subseteq W $ containing $ z $ and $ t,t' \in \coprod_i P_i U_i $
	satisfying $ s_w = t_w $ and $ s'_w= t'_w $ for each $ w \in \prod_i U_i $.
	Since $ \coprod_i (P_i)_{z_i} \models \rho(t_z, t'_z) $ and $ \coprod_i (P_i)_{z_i} $ is isomorphic to $ \rlim_{\prod_i V_i \ni z} \coprod_i P_i V_i $,
	we can replace $ \prod_i U_i \ni z $ by a smaller $ \prod_i V_i $ so that $ \coprod_i P_i V_i \models \rho(t,t')$.
	Hence, $ R(\prod_i V_i) \models \rho(s,s') $, and this shows that $ R_z \to \coprod_i (P_i)_{z_i} $ is an embedding.
\end{proof}

Let $ p_i \colon Z \to X_i $ be the projection map.
For any open $ U \subseteq X_i $, we define a homomorphism $ P_i U \to R(p^{-1}_i U) $ as follows:
put $ U_i = U $ and $ U_j = X_j $ for $ j \neq i $ so that $ \prod_j U_j = p^{-1}_i U $.
The family $ \{ P_i U \to \coprod_i P_i U_i \to \coprod_i (P_i)_{w_i} \}_{w \in \prod_j U_j} $ yields
a homomorphism $ P_i U \to \prod_{w \in \prod_j U_j} \coprod_i (P_i)_{w_i} $ whose image is contained in $ R(p^{-1}_i U) $.
Thus, we can equip $ p_i $ with a morphism $ p^{\sharp}_i \colon P_i \to (p_i)_*R $,
and then $ (p^{\flat}_i)_z \colon (P_i)_{z_i} \to R_z $ is the coprojection.
Once we know $ R_z \cong \coprod_i (P_i)_{z_i} $,
the universality of these $ p_i \colon (Z,R) \to (X_i,P_i) $ as a product follows immediately.

The construction of equalizers in $ \modsp{T_0}{\mathbf{Sp}} $ is similar to the above,
but we need to take into consideration a condition for basic open sets (see the remark below).
We leave the details to the reader.
We thus have shown the following:
\begin{Prop*} \label{prop:T0-ModSp-is-complete}
	$ \modsp{T_0}{\mathbf{Sp}} $ has all small limits.
\end{Prop*}

\begin{Rmk} \label{rmk:limits-of-T0-modelled-spaces}
	An analogous construction of the structure sheaf works for finitary pullbacks, but seemingly not for arbitrary limits.
	Suppose we have a diagram $ \mathcal{I} \ni i \mapsto (X_i,P_i) \in \modsp{T_0}{\mathbf{Sp}} $.
	Put $ Z := \llim_{i} X_i $ with projections $ p_i \colon Z \to X_i $.
	If $ \prod_i U_i \subseteq \prod_i X_i $ is a basic open set,
	we write $ \llim_i U_i $ for the intersection $ \prod_i U_i \cap \llim_i X_i \subseteq \prod_i X_i $.
	To define a sheaf $ R $ on $ Z $ as above, we have to restrict ourselves to consider $ \llim_i U_i $ satisfying $ \alpha_{ij}U_i \subseteq U_j $ 
	for each map $ \alpha_{ij} \colon X_i \to X_j $ in the diagram,
	for we have to take a colimit $ \rlim_i P_i U_i $ over the diagram whose transition map is $ P_j U_j \to P_i (\alpha^{-1}_{ij}U_j) \to P_i U_i $.
	However, this restriction violates the argument in the proof of the above claim.
	For the diagram $ i \mapsto (P_i)_{z_i} $ to be a filtered colimit of 
	\[ \Set{\text{``good'' basic open sets containing $ z $}}^{\op} \ni \llim_i U_i \quad \mapsto \quad (i \mapsto P_i U_i) \]
	in the functor category $ T_0\mhyphen\mathbf{Mod}^{\mathcal{I}^{\op}} $,
	there should exist, for each $ i $ and an open $ U \ni z_i $, such a good basic open $ \llim_i U_i \ni z $ with $ U_i \subseteq U $.
	This is not necessarily the case for infinitary pullbacks since there can exist continuous maps $ \{f_i \colon X_i \to Y \}_i $ and an open $ U \subseteq Y $
	such that $ f_i^{-1}U \subsetneq X_i $ for all $ i $.
	
	One can also obtain a limit in $ \modsp{T_0}{\mathbf{Sp}} $ by taking a colimit $ R := \rlim_i p^*_i P_i $ in $ T_0\mhyphen\mathbf{Mod}(\mathbf{Sh}(Z)) $.
	Here, cocompleteness of $ T_0\mhyphen\mathbf{Mod}(\mathbf{Sh}(Z)) $ can be seen via the associated sheaf functor
	\[ \mathbf{a} \colon T_0\mhyphen\mathbf{Mod}^{\mathcal{O}(Z)^{\op}} \simeq T_0\mhyphen\mathbf{Mod}(\mathbf{Set}^{\mathcal{O}(Z)^{\op}}) 
	\to T_0\mhyphen\mathbf{Mod}(\mathbf{Sh}(Z)). \]
	We remark that, from this description of $ R $, one can always have $ R_z \cong \rlim_i (P_i)_{z_i} $ since the functor
	$ \mathrm{Stalk}_z \colon T_0\mhyphen\mathbf{Mod}(\mathbf{Sh}(Z)) \to T_0\mhyphen\mathbf{Mod} $ preserves colimits.
	By the usual computation of colimits of sheaves, $ R $ is the associated sheaf of the presheaf
	\[ \mathcal{O}(Z)^{\op} \ni W \quad \mapsto \quad \rlim_{i \in \mathcal{I}^{\op}} \;\rlim_{p_i W \subseteq V_i} P_i V_i. \]
	Thus, to show these two descriptions coincide, 
	what remains unclear is whether local definitions of a section $ s \in \prod_{w \in W} \rlim_i (P_i)_{w_i} $ by elements of $ \rlim_i \;\rlim_{V_i} P_i V_i $ 
	can be replaced by elements of $ \rlim_i P_i U_i $ for $ \llim_i U_i \subseteq W $.
\end{Rmk}

Now, we prove our final theorem.
\begin{Thm}
	For any spatial Coste context $ (T_0,T,\Lambda) $, $ \modsp{\mathbb{A}}{\mathbf{Sp}} $ has all small limits.
\end{Thm}
\begin{proof}
	Let $ D \colon \mathcal{I} \to \modsp{\mathbb{A}}{\mathbf{Sp}} $ be a functor with $ \mathcal{I} $ small,
	and $ UD \colon \mathcal{I} \to \modsp{T_0}{\mathbf{Sp}}  $ be the composite of $ D $ and the forgetful functor.
	Then we have a forgetful functor $ \mathbf{Cone}(D) \to \mathbf{Cone}(UD) $ between the categories of cones over $ D $ and $ UD $.
	Examining the proof of \emph{Theorem \ref{thm:spectra-of-modelled-spaces}},
	we can also construct a right adjoint to this functor,
	i.e.\ the functor sending a cone $ \{ f_i \colon (Y,Q) \to (X_i,P_i) \}_i $ over $ UD $ to the cone over $ D $ whose vertex is
	\[ \Spec(Y,Q) = \Set{(y,I) \smid y \in Y,\, I \in \Spec(Q_y) \;\text{such that}\; \forall i,\, \widetilde{(f_i^{\flat})_y}(I)=\{\id_{(P_i)_{f_i y}}\}}. \]
	By \emph{Proposition \ref{prop:T0-ModSp-is-complete}}, $ \mathbf{Cone}(UD) $ has a terminal object,
	and the right adjoint $ \mathbf{Cone}(UD) \to \mathbf{Cone}(D) $ preserves it.
	Therefore, we have proved that $ D $ has a limit.
\end{proof}
Recalling the examples of Coste contexts in \S\ref{subsec:Coste-context}, we obtain the following consequences:
\begin{Cor*}
	The following categories are complete and cocomplete:
	\begin{itemize}
		\item $ \mathbf{RS} $, $ \mathbf{LRS} $,
		\item the category of ringed spaces with integral stalks and morphisms whose components are injective,
		\item the category of ringed spaces whose stalks are fields,
		\item the category of ringed spaces with indecomposable stalks and morphisms whose components are injective on idempotents,
		\item the category of distributive-latticed spaces with local stalks and morphisms whose components reflect $ 1 $. \qedhere
	\end{itemize}
\end{Cor*}

We can also show a result on pullbacks of spectra:
\begin{Cor}
	Let $ g \colon (Z,R) \to (X,P) $ be a morphism in $ \modsp{\mathbb{A}}{\mathbf{Sp}} $,
	and $ f \colon (Y,Q) \to (X,P) $ be a morphism in $ \modsp{T_0}{\mathbf{Sp}} $.
	If we take a pullback in $ \modsp{T_0}{\mathbf{Sp}} $
	\[ \tikz[auto]{
		\node (UL) at (0,2) {$ (W,S) $};
		\node (DL) at (0,0) {$ (Z,R) $};
		\node (UR) at (2,2) {$ (Y,Q) $};
		\node (DR) at (2,0) {$ (X,P) $};
		
		\draw[->] (UL) to node {$ h $} (DL);
		\draw[->] (UR) to node {$ f $} (DR);
		\draw[->] (DL) to node {$ g $}(DR);
		\draw[->] (UL) to (UR);
	}, \]
	then there exists a pullback in $ \modsp{\mathbb{A}}{\mathbf{Sp}} $
	\[ \tikz[auto]{
		\node (UL) at (0,2) {$ \Spec(h) $};
		\node (DL) at (0,0) {$ (Z,R) $};
		\node (UR) at (2,2) {$ \Spec(f) $};
		\node (DR) at (2,0) {$ (X,P) $};
		
		\draw[->] (UL) to (DL);
		\draw[->] (UR) to (DR);
		\draw[->] (DL) to node {$ g $}(DR);
		\draw[->] (UL) to (UR);
	}. \]
\end{Cor}
\begin{proof}
	Since $ \modsp{T_0}{\mathbf{Sp}} $ and $ \modsp{\mathbb{A}}{\mathbf{Sp}} $ has pullbacks,
	the vertical functors below between slice categories have right adjoints:
	\[ \tikz[auto]{
		\node (UL) at (0,2) {$ \modsp{\mathbb{A}}{\mathbf{Sp}}/(Z,R) $};
		\node (DL) at (0,0) {$ \modsp{\mathbb{A}}{\mathbf{Sp}}/(X,P) $};
		\node (UR) at (6,2) {$ \modsp{T_0}{\mathbf{Sp}}/(Z,R) $};
		\node (DR) at (6,0) {$ \modsp{T_0}{\mathbf{Sp}}/(X,P) $};
		
		\draw[->] (UL) to node {$ g \circ (-) $} (DL);
		\draw[->] (UR) to node {$ g \circ (-) $} (DR);
		\draw[->] (DL) to node {}(DR);
		\draw[->] (UL) to (UR);
	}. \]
	This square is commutative.
	By taking right adjoints, the following diagram commutes up to natural isomorphisms
	\[ \tikz[auto]{
		\node (UL) at (0,2) {$ \modsp{\mathbb{A}}{\mathbf{Sp}}/(Z,R) $};
		\node (DL) at (0,0) {$ \modsp{\mathbb{A}}{\mathbf{Sp}}/(X,P) $};
		\node (UR) at (6,2) {$ \modsp{T_0}{\mathbf{Sp}}/(Z,R) $};
		\node (DR) at (6,0) {$ \modsp{T_0}{\mathbf{Sp}}/(X,P) $};
		
		\draw[<-] (UL) to node {$ g^* $} (DL);
		\draw[<-] (UR) to node {$ g^* $} (DR);
		\draw[<-] (DL) to node {$ \Spec $} (DR);
		\draw[<-] (UL) to node {$ \Spec $} (UR);
	}. \]
\end{proof}

\paragraph{Acknowledgment}
The author would like to thank Soichiro Fujii and Yuki Imamura for helpful comments on a draft of this paper.
This work was supported by JST \mbox{ERATO} \mbox{HASUO} Metamathematics for Systems Design Project (No.\ JPMJER1603).

\printbibliography

\end{document}